\documentclass[10pt, a4paper, reqno]{amsart}

\usepackage[T1]{fontenc}
\usepackage{lmodern}
\usepackage[english]{babel}
\usepackage{mathrsfs}
\usepackage{booktabs}
\usepackage{float}
\usepackage{enumitem}
\usepackage[dvipsnames]{xcolor}
\usepackage{tikz,pgf}
\usetikzlibrary{calc}
\usepackage[all]{xy}
\usepackage{multirow}
\usepackage{array}
\usepackage{amssymb}
\usepackage{algorithm}
\usepackage{algpseudocode} 
\usepackage{mathdots}
\usepackage{yhmath}
\usepackage{calc}
\usepackage{hyperref}
\hypersetup{
    colorlinks,
    linkcolor={red!50!black},
    citecolor={blue!50!black},
    urlcolor={blue!80!black}
}
\usepackage{cleveref}

\algrenewcommand\algorithmicwhile{\textbf{While}}
\algrenewcommand\algorithmicfor{\textbf{For}}
\algrenewcommand\algorithmicdo{\textbf{Do}}
\algrenewcommand\algorithmicif{\textbf{If}}
\algrenewcommand\algorithmicthen{\textbf{Then}}
\algrenewcommand\algorithmicelse{\textbf{Else}}
\algrenewcommand\algorithmicend{\textbf{End}}
\algrenewcommand\algorithmicreturn{\textbf{Return}}

\theoremstyle{plain}
\newtheorem{lemma}{Lemma}[section]
\newtheorem{proposition}[lemma]{\textbf{Proposition}}
\newtheorem{theorem}[lemma]{\textbf{Theorem}}
\newtheorem{corollary}[lemma]{\textbf{Corollary}}

\theoremstyle{definition}
\newtheorem{definition}[lemma]{\textbf{Definition}}
\newtheorem{example}[lemma]{\textbf{Example}}
\newtheorem*{notation}{\textbf{Notation}}
\newtheorem{remark}[lemma]{Remark}
\newtheoremstyle{instyle}{}{}{}{}{}{}{0pt}{{}}
\theoremstyle{instyle}

\title{On the existence of paradoxical motions of generically rigid graphs on the sphere}
\thanks{The final version of this work was published on SIAM J. Discrete Math., 35(1), 325--361, \href{https://doi.org/10.1137/19M1289467}{DOI:10.1137/19M1289467}.}

\author[M. Gallet]{%
Matteo Gallet$^{\ast, \circ,\diamond}$}
\author[G. Grasegger]{%
Georg Grasegger$^{\ast,\triangleright}$}
\author[J. Legersk\'y]{%
Jan Legersk\'y$^{\dagger}$}
\author[J. Schicho]{%
Josef Schicho$^{\ast, \circ}$}
\thanks{$^\ast$ Supported by the Austrian Science Fund (FWF): W1214-N15, 
 Project DK9.} 
\thanks{$^\circ$ Supported by the Austrian Science Fund (FWF): P31061.}
\thanks{$^\diamond$ Supported by the Austrian Science Fund (FWF): Erwin 
Schr\"odinger Fellowship J4253.}
\thanks{$^\triangleright$ Supported by the Austrian Science Fund (FWF): P31888.}
\thanks{$^\dagger$ This project has received funding from the European Union's
Horizon 2020 research and innovation programme under the Marie Sk\l{}odowska-Curie grant
agreement No 675789.}
\address[MG]{International School for Advanced Studies/Scuola Internazionale Superiore di Studi Avanzati (ISAS/SISSA),
Via Bonomea 265, 34136 Trieste, Italy}
\email{mgallet@sissa.it}
\address[GG]{Johann Radon Institute for Computation and Applied Mathematics 
(RICAM), Austrian Academy of Sciences, Linz, Austria}
\email{georg.grasegger@ricam.oeaw.ac.at}
\address[JL, JS]{Johannes Kepler University Linz, Research Institute for Symbolic Computation (RISC), Austria}
\email{jan.legersky@risc.jku.at, jschicho@risc.jku.at}
\address[JL]{Department of Applied Mathematics, Faculty of Information Technology, Czech Technical University in Prague, Czech Republic.}

\newcommand{\Z}{\mathbb{Z}}

\newcommand{\R}{\mathbb{R}}
\newcommand{\C}{\mathbb{C}}

\newcommand{\p}{\mathbb{P}}

\newcommand{\mscr}{\mathscr}
\newcommand{\mcal}{\mathcal}

\newcommand{\SO}{\operatorname{SO}}
\newcommand{\pgl}{\p\!\operatorname{GL}}
\newcommand{\red}{\mathrm{red}}
\newcommand{\blue}{\mathrm{blue}}

\newcommand{\M}{\overline{\mscr{M}}}

\newcommand{\tomap}[2]{#1:#2}
\newcommand{\twotoone}{\tomap{2}{1}}

\newcommand{\casesQ}[1]{$\mathfrak{#1}$}
\newcommand{\caseG}{\casesQ{g}}
\newcommand{\caseO}{\casesQ{o}}
\newcommand{\caseE}{\casesQ{e}}
\newcommand{\caseR}{\casesQ{r}}
\newcommand{\caseL}{\casesQ{l}}

\newcommand{\cuts}[1]{\mathrm{#1}}
\newcommand{\cutou}{T_\cuts{ou}}
\newcommand{\cuteu}{T_\cuts{eu}}
\newcommand{\cutom}{T_\cuts{om}}
\newcommand{\cutem}{T_\cuts{em}}

\newcommand{\divou}{D_\cuts{ou}}
\newcommand{\diveu}{D_\cuts{eu}}
\newcommand{\divom}{D_\cuts{om}}
\newcommand{\divem}{D_\cuts{em}}

\newcommand{\curveC}{\mathcal{C}}
\newcommand{\curveD}{\mathcal{D}}
\newcommand{\curveZ}{\mathcal{Z}}

\DeclareMathOperator{\crossrat}{cr}

\newcommand{\CM}{C\!M}
\newcommand*\phantomas[3][c]{%
\ifmmode
\makebox[\widthof{$#2$}][#1]{$#3$}%
\else
\makebox[\widthof{#2}][#1]{#3}%
\fi
}

\tikzstyle{vertex}=[circle, draw, fill=black, inner sep=0pt, minimum size=4pt]
\tikzstyle{edge}=[line width=1.5pt,black!50!white]
\tikzstyle{gridp}=[inner sep=1pt,circle,fill=black!70!white]
\tikzstyle{gridl}=[black!50!white]

\tikzstyle{lnode}=[circle,white,draw=black!60!white,fill=black!60!white,inner sep=1pt]
\tikzstyle{cnode}=[circle,draw=black!60!white,fill=black!60!white,inner sep=1.5pt]
\tikzstyle{redge}=[edge,Red]
\tikzstyle{bedge}=[edge,NavyBlue]

\colorlet{ncol}{Green!60!black}
\tikzstyle{nvertex}=[vertex, draw=ncol, fill=ncol]
\tikzstyle{edgeq}=[edge,gray!60,densely dashed]
\tikzstyle{nedge}=[edge,ncol]
\tikzstyle{oedge}=[edge,Red!60!black]

\tikzstyle{curveline}=[line width=0.4mm]
\tikzstyle{markedpoint}=[circle, draw, fill,inner sep=0.08cm]

\newcounter{cases}
\renewcommand{\thecases}{(\arabic{cases})}
\newcommand{\newCase}{\refstepcounter{cases}\thecases}
\crefname{cases}{case}{case}

\begin{document}

\begin{abstract}
 We interpret realizations of a graph on the sphere up to rotations as elements 
of a moduli space of curves of genus zero. We focus on those graphs that admit an assignment of 
edge lengths on the sphere resulting in a flexible object. Our interpretation of realizations allows us to provide a 
combinatorial characterization of these graphs in terms of the existence of particular colorings of 
the edges. Moreover, we determine necessary relations for flexibility between the spherical lengths of the 
edges. We conclude by classifying all possible motions on the sphere of the complete bipartite graph with $3+3$ vertices where 
no two vertices coincide or are antipodal.
\end{abstract}

\maketitle

\section{Introduction}

The study of mobile spherical mechanisms is an active area of research, as 
confirmed by several recent publications on the topic, see for 
example~\cite{Corinaldi2018,Essomba2018,Nawratil2011,Nawratil2012,Nawratil2010,
Sun2018}. One of the reasons for this interest is that serial manipulators for 
which all the axes of the joints are concurrent can be considered as spherical 
mechanisms. In this paper, we focus on the combinatorial structure of these 
mechanisms, namely on the graphs whose vertices correspond to joints and whose 
edges correspond to bars. We will always suppose that the graphs are 
connected, without multiedges or self-loops. We describe combinatorial 
properties of these graphs 
ensuring that there exist assignments of edge lengths that make them flexible on 
the sphere. By the latter we mean that there are infinitely many essentially 
distinct ways to realize the vertices of the graph on the sphere so that the 
distance of adjacent vertices is the given edge length. Here by ``essentially 
distinct'' we mean that any two such realizations do not differ by an isometry 
of the sphere. 
Notice that only non-adjacent vertices are allowed to coincide or to be antipodal.
Moreover, we deal with complex realizations, namely we allow the coordinates of 
the points defining a realization of a graph to be complex numbers. This opens 
the door to tools from algebraic geometry that can be employed in the 
study of this problem.

\begin{figure}[ht]
  \begin{center}
    \includegraphics[width=3.8cm]{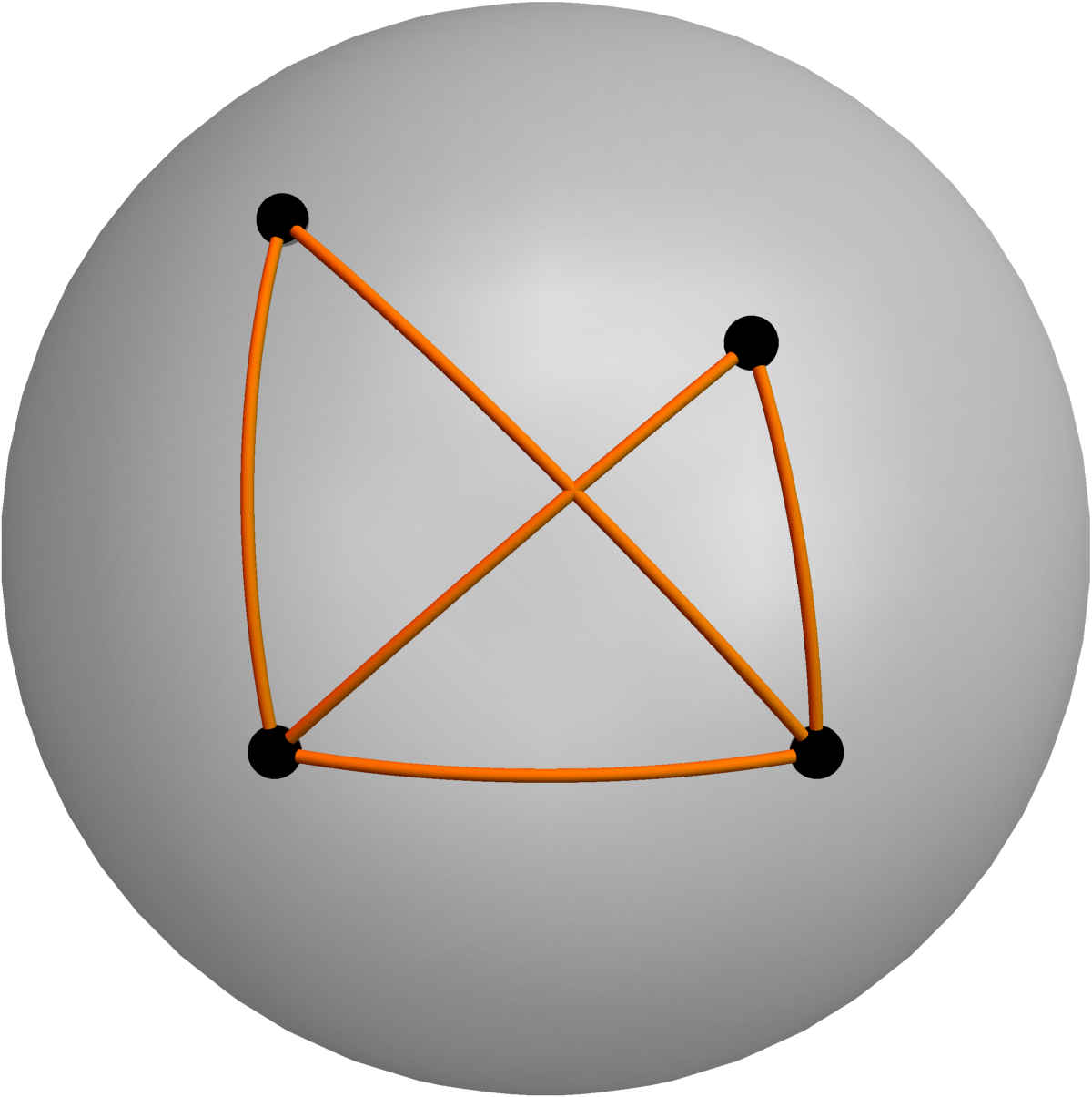} \hspace{0.2cm}
    \includegraphics[width=3.8cm]{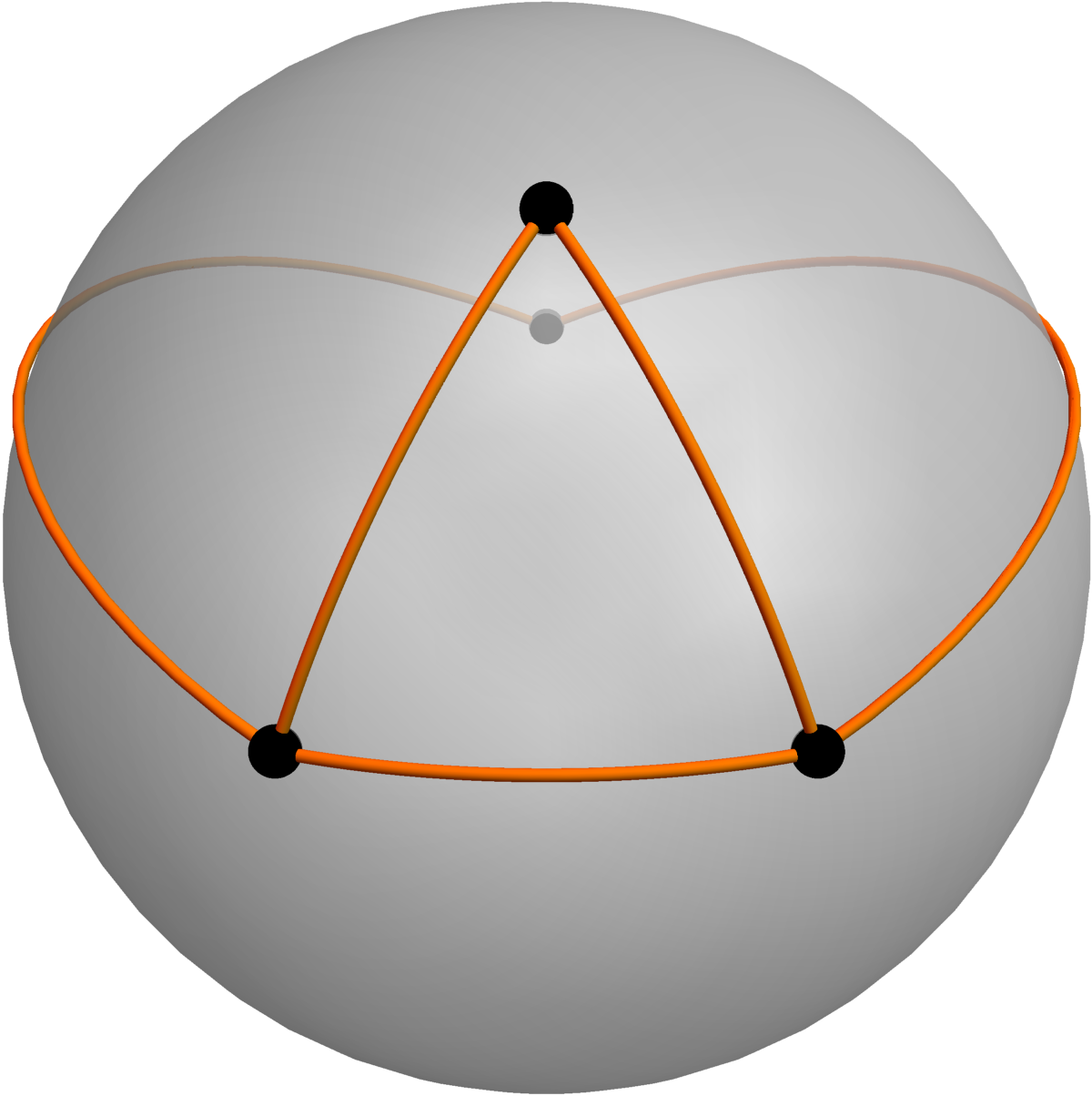} \hspace{0.2cm}
    \includegraphics[width=3.8cm]{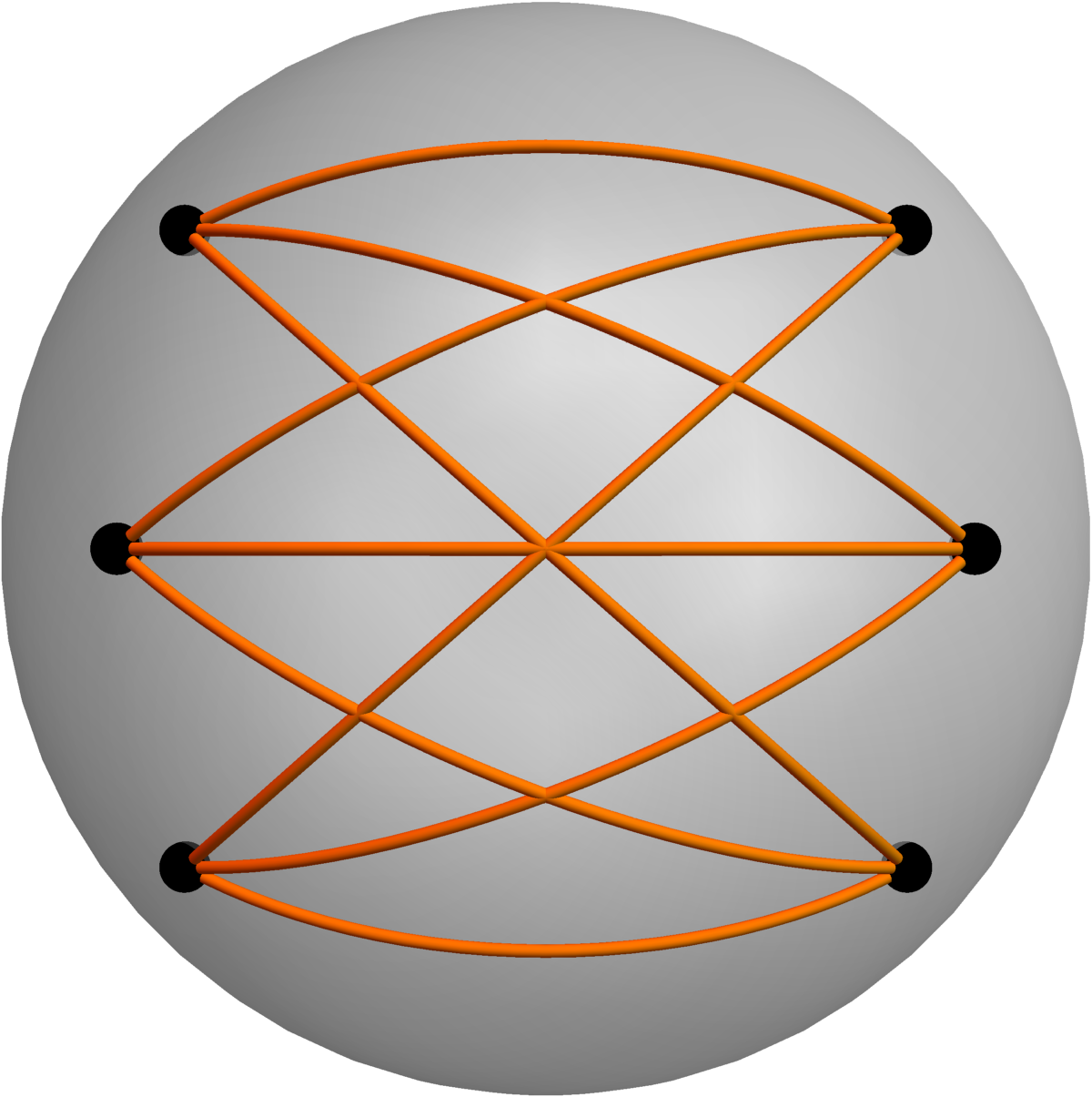}
    \caption{Realizations of graphs on the sphere: the unique graph with 4 vertices and 5 edges (left and middle)
    and the bipartite complete graph $K_{3,3}$ (on the right).}
  \end{center}
\end{figure}

A lot of the current research on rigidity and flexibility of graphs, polyhedra, and mechanisms
analyzes the situation in two- or three-dimensional Euclidean space. 
However, there have been already substantial contributions concerning 
rigidity or flexibility in other ambient spaces, for example 
the torus \cite{NixonRoss2014, Ross2014, Ross2015}, 
the cylinder \cite{NixonOwenPower2012}, 
revolute surfaces \cite{NixonOwenPower2014}, and 
other geometries (spherical, hyperbolic, projective) \cite{Alexandrov1997, Gaifullin2014, Gaifullin2015, Izmestiev, SaliolaWhitely, Stachel2001, Stachel2006}.

Although the combinatorial characterization we obtain applies to all graphs, it 
is of particular interest in the case of \emph{generically rigid} graphs. 
These are graphs such that for a general realization on the sphere there are 
only finitely many essentially distinct realizations having the same edge 
lengths. It is known that generically rigid graphs on the sphere coincide with 
generically rigid graphs in the plane; for a modern account on this topic, 
see for example~\cite{Eftekhari2018}. 
Pollaczek-Geiringer~\cite{Geiringer1927} and Laman~\cite{Laman1970} proved that
generically rigid graphs are precisely those spanned by graphs $G = (V,E)$ such that
$|E| = 2 |V| - 3$ and $|E'| \leq 2 |V'| - 3$ for every nontrivial induced subgraph 
$G' = (V', E')$ of~$G$.
Graphs of the latter kind are called \emph{Laman graphs}.

The two papers \cite{Grasegger2018, Grasegger2018a} study flexible assignments 
of edge lengths of graphs in the plane. The authors provide there a combinatorial description 
of those graphs admitting flexible assignments in terms of special colorings 
of their edges. In the spherical case, we use a subclass of the edge colorings for the planar case. A 
fundamental construction in~\cite{Grasegger2018} is based on parallelograms, which are not 
available on the sphere, hence we need to rely on new ideas.

The main technique we use takes inspiration from \emph{bond 
theory}, which has been developed to study paradoxical motions of serial 
manipulators, for example $5R$  closed chains with two degrees of freedom, or 
mobile $6R$ closed chains (see \cite{Hegedues2015, Hegedues2013, Li2018}). The 
core idea is to consider a compactification of the space of configurations of a 
manipulator which admits some nice algebraic properties. This introduces elements 
``on the boundary'', namely limits of configurations that are, in turn, not 
configurations. The analysis of these elements, called \emph{bonds}, provides 
information about the geometry of the manipulator.

In our case, we reduce the study of realizations of a graph on the unit sphere, 
up to rotations, to the study of elements of a so-called \emph{moduli space of 
stable curves with marked points}. The latter is a well-known object in 
algebraic geometry, and we use it to define the notion of bond for graphs that 
allow a flexible assignment of edge lengths, and to prove the main results of 
this paper. The reduction has been already explained and used 
in~\cite{GalletGraseggerSchicho}, and it is based on the observation that the 
spherical distance between two points on the unit 
sphere can be expressed in terms of the cross-ratio of four points on a conic.
Moreover, the action of rotations on the points of the sphere is translated 
into an action of the general linear group on the points on the conic. Once this 
reduction is established, we can apply the general philosophy of bond theory: 
if a graph admits a flexible assignment, then there is at least a one-dimensional set 
of realizations in the moduli space; this one-dimensional set intersects the boundary 
of the moduli space, originating bonds. It turns out that bonds have a 
rich combinatorial structure, and they can be used to define a special class of 
bichromatic edge colorings on the graph we started with. We call these colorings 
\emph{NAP} (\emph{``No Alternating Path''}) since no $3$-path in the graph has 
an alternating coloring. Our main result is then (see \cref{theorem:NAP_movable}):

\smallskip
\textbf{Theorem.} Let $G$ be a connected graph without multiedges, or 
self-loops. Then $G$ 
admits a flexible assignment of edge lengths on the sphere if and only if it 
admits a NAP-coloring.

\smallskip
Notice that, as we have already remarked, although the previous theorem holds 
in full generality, it is of particular interest when we consider generically 
rigid graphs. In fact, not all Laman graphs admit NAP-colorings; when 
they do, there is a way to realize them on the sphere so that they 
are flexible. This is an example of what we mean by \emph{paradoxical 
motion} in the title of this paper.

We then delve into a particular case, the one of the bipartite graph~$K_{3,3}$,
which is the smallest ``interesting'' graph that is at the same time generically rigid
and admits NAP-colorings. Here ``interesting'' means that we have realizations 
such that no two vertices collapse or are antipodal. 
We obtain a classification of all the possible motions of~$K_{3,3}$ on the 
sphere. Here we stress that, unlike the rest of the paper, we focus on real 
motions. The classification is obtained by carefully analyzing how the 
subgraphs of~$K_{3,3}$ isomorphic to~$K_{2,2}$ move during a motion 
of~$K_{3,3}$. This analysis employs both constraints imposed by the possible 
bonds that may arise for~$K_{3,3}$, and elementary considerations in sphere 
geometry: these two approaches combined narrow the cases to be considered so 
that they can be examined one by one. In contrast to the planar situation ---
where there are only two such motions, described by Dixon more than a hundred 
years ago~\cite{Dixon1899} --- on the sphere there are three possibilities (see 
\cref{theorem:classification_K33}): 

\smallskip
\textbf{Theorem.}  All the possible motions of~$K_{3,3}$ on the sphere, for 
which no two vertices coincide or are antipodal, are the spherical 
counterparts of the planar Dixon motions, and a third new kind of motion, 
in which the angle between the two diagonals of a quadrilateral in~$K_{3,3}$
stays constant during the motion.

\smallskip
\textbf{Structure of the paper.} 
\Cref{realizations_as_moduli} recalls from~\cite{GalletGraseggerSchicho} 
how (complex) realizations on the sphere up to rotations can be interpreted as 
elements of a moduli space of stable curves with marked points. 
\Cref{colorings} introduces the notion of NAP-coloring, and proves the 
characterization of graphs admitting a flexible (complex) assignment on the 
sphere via NAP-colorings. The section ends by explaining how to obtain algebraic 
relations between the spherical lengths of a flexible assignment of a graph 
that has ``enough'' edges (e.g., a graph  spanned by a Laman graph). 
\Cref{K33} analyzes all possible (real) motions of the graph~$K_{3,3}$ on 
the sphere for which no two vertices collapse or are antipodal.

\section{Realizations on the sphere as elements of a moduli space}
\label{realizations_as_moduli}

The aim of this section is to recall how we can interpret a realization of a 
graph on the sphere, up to sphere isometries, as an element of the moduli 
space of stable curves of genus zero with marked points. This interpretation is 
described in \cite[Sections~1--3]{GalletGraseggerSchicho} and here we 
recall its main aspects to make the paper more self-contained; no new results 
are present in this section.

First of all, let us define what we mean by realization of a graph, which we 
suppose connected, and without multiedges or self-loops. As we 
pointed out in the introduction, we work over the complex numbers, and we 
denote by~$S^2_{\C}$ the complexification of the unit sphere~$S^2$, namely
\[
 S^2_{\C} := 
 \bigl\{ 
  (x,y,z) \in \C^3 \, \colon \, x^2 + y^2 + z^2 = 1 
 \bigr\} \,.
\]
On the real sphere~$S^2 \subset \R^3$ we can compute the distance between two 
points as the cosine of the angle they form with respect to the center of the 
sphere. Since this function is algebraic, we can use it also in the complex 
setting. 

\begin{definition}
\label{definition:spherical_distance}
Given two points $T, U \in S_{\C}^2$, we define 
\[
 \delta_{S^2}(T,U) := \left\langle T, U \right\rangle = 
 \sum_{i=1}^{3} T_i U_i \,.
\]
Notice that, if $\bar{U} := -U$ denotes the antipodal point of~$U$, then 
$\delta_{S^2}(T,U) \! = \! -\delta_{S^2}(T, \bar{U})$. 
For technical reasons, we 
define the \emph{spherical distance} as 
\[
 d_{S^2}(T,U) := \frac{1 - \left\langle T, U \right\rangle}{2} \,.
\]
If $T$ and $U$ are real points, then $\delta_{S^2}$ ranges from $-1$ to $1$, 
and attains the value $0$ when the vectors corresponding to $T$ and $U$ are orthogonal, 
while $d_{S^2}$ ranges from~$0$ (when $T$ and $U$ coincide) to~$1$ (when $T$ and $U$ are antipodal).
\end{definition}

Realizations of graphs on the sphere are then assignments of points on the 
sphere to the vertices of the graph.

\begin{definition}
\label{definition:realization}
 Let $G = (V,E)$ be a graph. A \emph{realization} of~$G$ is a function \mbox{$\rho 
\colon V \longrightarrow S^2_{\C}$}. Given an assignment $\lambda \colon E 
\longrightarrow \C \setminus \{0,1\}$ of spherical distances for the edges, we 
say that a realization~$\rho$ is \emph{compatible} with~$\lambda$ if
\[
 d_{S^2} \bigl( \rho(a), \rho(b) \bigr) = \lambda(\{a,b\}) \quad 
 \text{for all } \{a,b\} \in E \,.
\]
Recall that we allow non-adjacent vertices to coincide/be antipodal in such a realization,
whereas the adjacent ones cannot coincide/be antipodal since the spherical distance 
is not allowed to be $0$ or $1$.
\end{definition}

Once we know what a realization is, we can define when an assignment of 
spherical distances is flexible.

\begin{definition}
 Let $G = (V,E)$ be a graph, and let $\lambda \colon E \longrightarrow \C 
\setminus \{0,1\}$ be an 
assignment of edge lengths. We say that $\lambda$ is a \emph{flexible 
assignment} if there exist infinitely many realizations $\rho \colon V 
\longrightarrow S^2_{\C}$ compatible with~$\lambda$ that are essentially 
distinct. We say that two realizations $\rho_1, \rho_2$ of~$G$ are 
\emph{essentially distinct} if there exists no element $\sigma \in \SO_3(\C)$ 
such that $\rho_1 = \sigma \circ \rho_2$, where
\[
 \SO_3(\C) := 
 \bigl\{
  \sigma \in \C^{3 \times 3} \, \colon \, 
  \sigma \sigma^{t} = \sigma^{t} \sigma = \mathrm{id}, \, 
  \det(\sigma) = 1
 \bigr\}
\]
is the complexification of the group of rotations on the unit sphere.
\end{definition}

The first step towards the moduli space is to associate to each point 
in~$S^2_{\C}$ two points on a rational curve. We do this by considering the 
projective closure of~$S^{2}_{\C}$:
\[
 \mathbb{S}^2_{\C} := 
 \bigl\{ 
  (x:y:z:w) \in \p^3_{\C} \, \colon \, 
  x^2 + y^2 + z^2 - w^2 = 0 
 \bigr\}
\]
and defining the conic~$\mcal{A}$ as the intersection
\[
 \mcal{A} := \mathbb{S}^2_{\C} \cap \{ w = 0 \} \,.
\]
Since $\mathbb{S}^2_{\C}$ is a smooth quadric, there are two families of lines 
in~$\mathbb{S}^2_{\C}$, and each point on~$\mathbb{S}^2_{\C}$ lies on exactly 
one line in one family and on exactly one line in the other family. 

\begin{definition}
\label{definition:lifts}
Let $T \in \mathbb{S}^2_{\C}$ and let~$L_1$ and~$L_2$ be the two lines 
in~$\mathbb{S}^2_{\C}$ passing through~$T$. We define the \emph{left} and the 
\emph{right lifts} of~$T$, denoted by~$T_l$ and~$T_r$, as the intersections 
of~$L_1$ and~$L_2$ with the conic~$\mcal{A}$.
\end{definition}

Since the conic~$\mcal{A}$ is a smooth rational curve, it is isomorphic to~$\p^1_{\C}$,
hence the cross-ratio of four points in~$\mcal{A}$ is defined. 
If the four points are the left and right lifts of two points $T, U 
\in S^2_{\C}$, one finds that the spherical distance between~$T$ and~$U$ can be 
expressed in terms of the cross-ratio of the lifts (see 
\cite[Lemma~3.3]{GalletGraseggerSchicho}): 
\[
 d_{S^2}(T, U) = 
 \frac{\mathrm{cr}\bigl( T_l, T_r, U_l, U_r \bigr)}{\mathrm{cr}\bigl( T_l, T_r, U_l, U_r \bigr) - 1} = 
 \mathrm{cr}\bigl( T_l, U_r, U_l, T_r \bigr) \,.
\]
The key step for moving from realizations to elements of a moduli space is that 
two $n$-tuples of points in~$S^2_{\C}$ differ by an element of~$\SO_{3}(\C)$ 
if and only if the corresponding two $2n$-tuples of lifts differ by an element 
of~$\pgl_{2}(\C)$ when all left and right lifts are distinct (see 
\cite[Proposition~3.4]{GalletGraseggerSchicho}).

The previous results clarify that dealing with realizations of a graph on the 
sphere, up to~$\SO_3(\C)$, is equivalent to dealing with tuples of points on a 
smooth rational curve, up to automorphisms, on which constraints for the cross-ratios 
are imposed, provided that the realizations give rise to distinct left and right 
lifts. Tuples of distinct $m$ points on a smooth rational curve up to automorphisms are 
parametrized by the moduli space~$\mscr{M}_{0,m}$, which is isomorphic to
\[
 \bigl( \p^1_{\C} \setminus \{(0:1), (1:0), (1:1)\} \bigr)^{m-3} 
 \setminus
 \left\{
 \begin{array}{c}
  \text{tuples } (r_1, \dotsc, r_{m-3}) \text{ where } \\
  r_i = r_j \text{ for some } i \neq j
 \end{array}
 \right\} \,.
\]
To see why~$\mscr{M}_{0,m}$ is isomorphic to the latter space, 
recall that all smooth rational curves are isomorphic to~$\p^1_{\C}$,
and that its automorphism group~$\pgl_2(\C)$ is $3$-transitive on points, 
namely if we have an $m$-tuple of distinct points in~$\p^1_{\C}$,
up to isomorphism we can suppose that the first three points
are equal to $(0:1)$, $(1:0)$, and~$(1:1)$.
For our purposes, however, we need a \emph{compact} moduli space,
and this is not the case for~$\mscr{M}_{0,m}$. 
We could take $\bigl( \p^1_{\C} \bigr)^{m-3}$ as a candidate for such a compactification, 
but in this way we would not know how to interpret
the elements of $\bigl( \p^1_{\C} \bigr)^{m-3} \setminus \mscr{M}_{0,m}$
as tuples of points on a rational curve.
As in~\cite{GalletGraseggerSchicho}, we use a smooth compactification of~$\mscr{M}_{0,m}$, 
proposed by Knudsen~\cite{Knudsen1983} and denoted by~$\M_{0,m}$. This 
space is composed of the so-called \emph{stable curves of genus zero with $m$ 
marked points} (see \cite[Introduction]{Keel1992}). These are reduced, at worst 
nodal curves~$X$ with $m$ distinct marked points $O_1, \dotsc, O_m$ such that 
the irreducible components of~$X$ form a tree of rational curves, the points 
$\{ O_i \}_{i=1}^{m}$ lie on the smooth locus of~$X$, and for each irreducible 
component of~$X$, there are at least $3$ points which are either singular, or 
marked. The moduli space $\mscr{M}_{0,m}$ sits inside~$\M_{0,m}$ as an open set, 
since every rational curve with distinct marked points is naturally a stable 
curve. The complement of~$\mscr{M}_{0,m}$ in~$\M_{0,m}$ can be interpreted as a 
``boundary'' for~$\mscr{M}_{0,m}$. The boundary 
of~$\M_{0,m}$ will play a prominent role in the next sections. For example, 
$\mscr{M}_{0,4}$ is isomorphic to $\p^1_{\C} \setminus \{(0:1), (1:0), (1:1)\}$ 
and $\M_{0,4} \cong \p^1_{\C}$, where the boundary is constituted of the following 
three stable curves:
\begin{center}
\begin{tikzpicture}
  \draw[curveline] (-1,1) -- (1,-1);
  \draw[curveline] (-1,-1) -- (1,1);
  \draw (-1/2,1/2) node[markedpoint, label=west:$\mathbf{O_1}$] {};
  \draw (1/2,-1/2) node[markedpoint, label=east:$\mathbf{O_2}$] {};
  \draw (-1/2,-1/2) node[markedpoint, label=west:$\mathbf{O_3}$] {};
  \draw (1/2,1/2) node[markedpoint, label=east:$\mathbf{O_4}$] {};
\end{tikzpicture}
\qquad
\begin{tikzpicture}
  \draw[curveline] (-1,1) -- (1,-1);
  \draw[curveline] (-1,-1) -- (1,1);
  \draw (-1/2,1/2) node[markedpoint, label=west:$\mathbf{O_1}$] {};
  \draw (1/2,-1/2) node[markedpoint, label=east:$\mathbf{O_3}$] {};
  \draw (-1/2,-1/2) node[markedpoint, label=west:$\mathbf{O_2}$] {};
  \draw (1/2,1/2) node[markedpoint, label=east:$\mathbf{O_4}$] {};
\end{tikzpicture}
\qquad
\begin{tikzpicture}
  \draw[curveline] (-1,1) -- (1,-1);
  \draw[curveline] (-1,-1) -- (1,1);
  \draw (-1/2,1/2) node[markedpoint, label=west:$\mathbf{O_1}$] {};
  \draw (1/2,-1/2) node[markedpoint, label=east:$\mathbf{O_4}$] {};
  \draw (-1/2,-1/2) node[markedpoint, label=west:$\mathbf{O_3}$] {};
  \draw (1/2,1/2) node[markedpoint, label=east:$\mathbf{O_2}$] {};
\end{tikzpicture}
\end{center}

\phantomsection
\label{page:boundary}
The boundary makes~$\M_{0,m}$ a compact space, so in particular there exist
limits of sequences of stable curves where two marked points get closer and closer:
the limit of this sequence is a stable curve where a new component is added,
and the two marked points are located on this new component (see \cite[pages 546--547]{Keel1992}).
Notice that the elements of~$\M_{0,m}$ are (stable) curves, but of course, 
being~$\M_{0,m}$ an algebraic variety, we can speak of curves \emph{inside} $\M_{0,m}$, 
meaning one-dimensional subvarieties of~$\M_{0,m}$. 
Actually, we will see that flexible assignments of graphs determine curves in~$\M_{0,m}$,
in the sense of one-dimensional subvarieties.
To prevent confusion and distinguish these two concepts, whenever we think of 
an element of~$\M_{0,m}$ as a curve, we always attach the adjective \emph{stable},
while when we speak about subvarieties, we just say \emph{curves}.

We will encode realizations of a graph on the sphere as fibers of a map between moduli spaces. 
The discussion after Definition~\ref{definition:lifts} explains that, 
when realizations determine distinct left and right lifts, 
we can associate them to elements of~$\mscr{M}_{0,2n} \subset \M_{0,2n}$. 
When this is not the case (for example, when we have a realization that is not injective) 
this association is not straightforward. 
However, we are going to see (\cref{lemma:surjective}) that the formalism of~$\M_{0,2n}$ 
allows us to deal with all cases in a uniform way.

\begin{notation}
We denote by $P_1, \dotsc, P_n$, $Q_1, \dotsc, Q_n$ the marked points of stable curves of the moduli space~$\M_{0,2n}$. 
The labeling here reflects the fact that we think about the $2n$ marked points 
as the $2n$ lifts of $n$ points on the sphere, listing first the left lifts, and then the right lifts.
\end{notation}

\begin{definition}[{\cite[Definition~5.1]{GalletGraseggerSchicho}}]
\label{definition:map}
 Let $G = (V,E)$ be a graph, and suppose that $V = \{1, \dotsc, n\}$. We 
define the following morphism
\[
 \Phi_{G} \colon \M_{0,2n} 
 \longrightarrow 
 \prod_{\{a,b\} \in E} \M_{0,4}^{a, b}
 \cong
 \bigl( \p^1_{\C} \bigr)^{|E|}
\]
whose components are the maps $\pi_{a,b} \colon \M_{0,2n} 
\longrightarrow \M_{0,4}^{a,b}$ that forget all the marked points except for the ones 
labeled by $P_a$, $P_b$, $Q_a$, and~$Q_b$. A complex assignment $\lambda \colon E 
\longrightarrow \C \setminus \{0,1\}$ of~$G$ corresponds to an element~$\Lambda$ in the codomain of~$\Phi_{G}$;
due to the correspondence between elements of~$\mscr{M}_{0,2n}$ and points on the sphere, 
the preimage $\Phi_{G}^{-1}(\Lambda) \cap \mathscr{M}_{0,2n}$ encodes realizations of~$G$
on the sphere compatible with the assignment~$\lambda$, up to rotations.
\end{definition}

We conclude this section by describing some divisors in the moduli 
space~$\M_{0,2n}$ which are crucial for our investigations.

\begin{definition}[{\cite[Introduction]{Keel1992}}]
\label{definition:vital_divisor}
 Let $(I,J)$ be a partition of the marked points~$\{P_1, \dotsc, P_n, Q_1, \dotsc, Q_n\}$ where $|I| \geq 2$ and $|J| 
\geq 2$. We define the divisor~$D_{I,J}$ in~$\M_{0,2n}$ to be the divisor whose 
general element is a stable curve with two irreducible components such that the 
marked points in~$I$ lie on one component, while the marked points in~$J$ lie on the other component.
Hence $D_{I,J}$ contains all those stable curves, together with stable curves with more components, 
where the marked points are distributed according to a refinement of the partition $(I,J)$.
\end{definition}

\begin{proposition}[Knudsen, {\cite[Introduction and Fact~2]{Keel1992}}]
\label{proposition:vital_product}
 Any divisor $D_{I,J}$ as in \cref{definition:vital_divisor} is 
isomorphic to the product~$\M_{0,|I|+1} \times \M_{0, |J| + 1}$. Under this 
isomorphism, the point of intersection of the two components of a general stable 
curve in~$D_{I,J}$ counts as an extra marked point in each of the two factors of 
the product. \end{proposition}

\begin{remark}
\label{remark:keel_relations}
 Recall from \cite{Keel1992} that if $a,b \in \{1, \dotsc, n\}$, then 
there is a way to identify $\p^1_{\C}$ with $\M_{0,4}$ so that the 
preimages of the points~$(0:1)$, $(1:0)$, and~$(1:1)$ under the forgetful
map~$\pi_{a,b} \colon \M_{0,2n} \longrightarrow \M_{0,4}$ are given by
\[
 \bigcup_{\substack{P_a,P_b \in I \\ Q_a,Q_b \in J}} D_{I,J}, \quad
 \bigcup_{\substack{P_a,Q_a \in I \\ P_b,Q_b \in J}} D_{I,J}, \quad
 \text{and }
 \bigcup_{\substack{P_a,Q_b \in I \\ P_b,Q_a \in J}} D_{I,J}.
\]
\end{remark}

\section{Bonds and NAP-colorings}
\label{colorings}

In this section we characterize combinatorially the property for a graph of 
admitting a flexible assignment of spherical distances for the edges. We do so 
by proving that flexibility is equivalent to the existence of particular 
colorings. This equivalence is shown by considering how configuration curves of 
mobile graphs intersect the boundary of~$\M_{0,m}$. 

\begin{definition}
\label{definition:cut}
Consider a stable curve~$X$ of genus zero with $2n$ marked points $P_1, \dotsc, P_n$, $Q_1, \dotsc, Q_n$ 
that has at least two irreducible components, namely $X$ belongs to the boundary $\M_{0,2n} \setminus \mscr{M}_{0,2n}$. 
Pick a singular point~$z$ of~$X$. Then $X \setminus \{z\}$ has exactly two connected 
components~$X_1$ and~$X_2$. We define the \emph{cut of $P_1, \dotsc, P_n$, $Q_1, \dotsc, Q_n$ 
determined by~$(X,z)$} to be the partition $(I,J)$ of the marked points $\{P_1, \dotsc, P_n$, $Q_1, \dotsc, Q_n\}$ 
given by
\[
 I := \{P_1, \dotsc, P_n, Q_1, \dotsc, Q_n\} \cap X_1,
 \qquad
 J := \{P_1, \dotsc, P_n, Q_1, \dotsc, Q_n\} \cap X_2.
\]
\end{definition}

\begin{remark}
\label{remark:cut_divisor}
 Notice that if $(X,z)$ induces a cut~$(I,J)$, then $X \in D_{I,J}$.
Moreover, if in a cut~$(I, J)$ we swap~$P_k$ with~$Q_k$ for all~$k$,
obtaining another cut~$(\bar{I}, \bar{J})$, then the divisor~$D_{\bar{I},
\bar{J}}$ is the complex conjugate of~$D_{I,J}$.
\end{remark}

As we explained in the introduction, one of the key ideas of this paper is to 
look at limit elements of the configuration set of a graph with a flexible 
assignment of spherical edge lengths.

\begin{definition}
 Let $G = (V,E)$ be a graph. A \emph{bond} of~$G$ is any stable curve with at least 
two irreducible components in a fiber~$\Phi_G^{-1}(\Lambda)$ 
(see \cref{definition:map}) where no component of~$\Lambda \in \left(\p^1_{\C}\right)^{|E|}$ is 
$(0:1)$, $(1:0)$, or $(1:1)$, where the identification $\M_{0,4} \cong \p^1_{\C}$ 
is made according to \cref{remark:keel_relations}. 
\end{definition}

Bonds determine special cuts:

\begin{lemma}
\label{lemma:cuts_no_two}
 Let $G = (V,E)$ be a graph and let $X$ be a bond of~$G$. Pick a singular 
point~$z$ of~$X$. Then from \cref{definition:cut} we get a 
cut~$(I,J)$, which is a partition of the $2n$-tuple of the marked points $P_1, 
\dotsc, P_n$, $Q_1, \dotsc, Q_n$ of the stable curves in~$\M_{0,2n}$. Such a cut 
satisfies the following condition: for every edge $\{a,b\} \in E$, the 
cardinality of the intersection $I \cap \{ P_a, P_b, Q_a, Q_b\}$ is never~$2$, 
and the same holds for $J \cap \{ P_a, P_b, Q_a, Q_b\}$.
\end{lemma}
\begin{proof}
 This lemma follows from \cref{remark:keel_relations}: in fact, if the 
statement did not hold for an edge~$\{a,b\}$, we would have that $\pi_{a,b}$ takes one of the 
three values $(0:1)$, $(1:0)$, or $(1:1)$ on~$X$. On the other hand, 
since $X$ is a bond and $\{a,b\}$ is an edge, $\pi_{a,b}(X)$ cannot be 
one of these three values.
\end{proof}

Our first goal now is to show that a graph admitting a flexible assignment has 
a bond with good properties (\cref{proposition:flexible_bond}). To 
do so, we introduce the \emph{Cayley-Menger variety}. For $n \geq 2$, we 
define $\CM_n$ to be the Zariski closure in $(\p^1_{\C})^{\binom{n}{2}}$ of 
\[ 
 \Bigl\{ 
  \bigl( d_{a,b}(W_1, \dotsc, W_n) \bigr)_{1 \leq a < b \leq n} 
  \text{ for } (W_1, \dotsc, W_n) \in (S^2_{\C})^n 
 \Bigr\} \,,
\]
where $d_{a,b} \colon (S^2_{\C})^n \longrightarrow \p^1_{\C}$ maps a 
tuple~$(W_1, \dotsc, W_n)$ to the point $\bigl( d_{S^2}(W_a, W_b) : 1 
\bigr)$ in~$\p^1_{\C}$. Hence, by definition, the Cayley-Menger variety~$\CM_n$ 
comes with $\binom{n}{2}$ projection maps $\tau_{a,b} \colon \CM_n 
\longrightarrow \p^1_{\C}$. Given a graph $G = (V,E)$ where $V = \{1, 
\dotsc, n\}$, we denote the product $\prod_{\{a,b\} \in E} \tau_{a,b}$ 
by~$\tau_G$.

\begin{lemma}
\label{lemma:cayley_menger}
 Let $\lambda \in \C^{E} \setminus \{0,1\}$ be a flexible assignment of a graph~$G = (V,E)$. 
Then there exists a non-edge $\{c,d\}$ such that the restriction 
of~$\tau_{c,d}$ to $\tau_G^{-1}(\Lambda)$ is surjective, where $\Lambda$ is the 
point in $(\p^1_{\C})^{|E|}$ corresponding to~$\lambda$.
\end{lemma}
\begin{proof}
 This follows from the definition of flexible assignment: in fact, we know that 
the image of $\tau_{c,d}|_{\tau_G^{-1}(\Lambda)}$ is composed of 
infinitely many elements for some non-edge $\{c,d\}$, because the distance of at 
least two vertices must vary among the infinitely many essentially distinct 
realizations. Hence the map must be surjective being a morphism between 
projective varieties.
\end{proof}

\begin{lemma}
\label{lemma:surjective}
 If a graph~$G = (V,E)$ has a flexible assignment on the 
sphere, then there exists 
$\Lambda \in \prod_{\{a,b\} \in E} \M_{0,4}^{a, b}$ and a 
non-edge $\{c,d\}$ such that the restriction of~$\pi_{c,d}$ to
$\Phi^{-1}_{G}(\Lambda)$ is surjective.
\end{lemma}
\begin{proof}
By \cref{lemma:cayley_menger} there exists $\Lambda \in \bigl( \p^1_{\C} 
\bigr)^{|E|}$ and a non-edge $\{c,d\}$ such that 
$\tau_{c,d}|_{\tau_G^{-1}(\Lambda)}$ is surjective. We construct a map 
$\eta \colon \M_{0,2n} \longrightarrow \CM_n$ such that the following diagram 
is commutative:
\begin{center}
  \begin{tikzpicture}
    \node (Mn) at (0,-1.1) {$\M_{0,2n}$};
    \node (CM) at (0,1.1) {$\CM_n$};
    \node[anchor=east] (P) at (-1,0) {$\p^1_{\C}$};
    \node[align=left,anchor=west] (PM) at (2,0) {$\displaystyle\bigl( \p^1_{\C} \bigr)^{|E|} \cong \prod_{\{a,b\} \in E}^{\phantom{\{a,b\} \in E}} \M_{0,4}^{a, b}$};
    \draw[->] (Mn) to node[left,font=\scriptsize] {$\eta$} (CM);
    \draw[->] (Mn) to node[below left,font=\scriptsize] {$\pi_{c,d}$} (P);
    \draw[->] (Mn) to node[below,font=\scriptsize] {$\Phi_G$} (PM);
    \draw[->] (CM) to node[above left,font=\scriptsize] {$\tau_{c,d}$} (P);
    \draw[->] (CM) to node[above,font=\scriptsize] {$\tau_G$} (PM);
  \end{tikzpicture}
\end{center}
The morphism $\eta$ is defined as follows: take the product of all maps 
$\pi_{a,b}$ where $a,b \in \{1, \dotsc, n\}$ with $a \neq b$. This 
gives a map $\M_{0,2n} \longrightarrow \bigl( \p^1_{\C} \bigr)^{\binom{n}{2}}$, 
once we identify each $\M_{0,4}^{a,b}$ with $\p^1_{\C}$. The image of 
$\mscr{M}_{0,2n}$ is contained in~$\CM_n$ and contains an open subset 
of~$\CM_n$, because from an element in~$\mscr{M}_{0,2n}$ we can construct a set 
of $n$ points in~$S^2_{\C}$, and conversely a general choice of $n$ distinct 
points in~$S^2_\C$ determines a general element of~$\CM_n$. Since 
both~$\M_{0,2n}$ and $\CM_n$ are irreducible and of the same dimension, the map 
$\eta$ has the desired codomain and is surjective. By construction, the 
previous diagram is commutative. Now the statement follows by inspecting the 
diagram.
\end{proof}

\begin{proposition}
\label{proposition:flexible_bond}
 If $G$ has a flexible assignment~$\lambda$ on the sphere, then it admits a 
bond. Moreover, we can choose the bond so that it induces a cut~$(I,J)$ such 
that there exists a non-edge $\{c,d\}$ for which $P_c, Q_c \in I$ and $P_d, Q_d 
\in J$. 
\end{proposition}
\begin{proof}
 By \cref{lemma:surjective}, there exists a non-edge $\{c,d\}$ such that 
$\pi_{c,d}$ is surjective on $\Phi_G^{-1}(\Lambda)$. Hence, there 
exists an element in $\Phi_G^{-1}(\Lambda) \cap 
\pi_{c,d}^{-1}\bigl((1:0)\bigr)$. This is a bond fulfilling the 
conditions in the statement (see also \cref{remark:keel_relations}).
\end{proof}

The special cuts arising from bonds allow us to define colorings of the edges of a graph.

\begin{definition}
\label{definition:bond_coloring}
 Let $G = (V,E)$ be a graph and let $(I,J)$ be a cut induced by a bond of~$G$.
We define a coloring $\varepsilon_{I,J} \colon E \longrightarrow \{ \red, 
\blue\}$ as follows: the edge $\{a,b\}$ is colored red if at least three among 
$P_a, P_b, Q_a, Q_b$ are in~$I$; 
it is colored blue if at least three among $P_a, P_b, Q_a, Q_b$ are in~$J$. 
\Cref{lemma:cuts_no_two} guarantees that every edge has a color.
\end{definition}

\begin{remark}
 Notice that if in \cref{definition:bond_coloring} we swap~$I$ 
and~$J$, then we swap red and blue in the coloring.
\end{remark}

We now describe a crucial property of these colorings~$\varepsilon_{I,J}$. To 
do so, we start with a little detour to the planar situation.

\emph{No almost cycle (NAC) colorings} were introduced in~\cite{Grasegger2018} 
as a combinatorial tool to characterize graphs having flexible assignments in 
the plane. We recall their definition.

\begin{definition}
 Let $G = (V,E)$ be a graph, and let $\varepsilon \colon E \longrightarrow \{ 
\red, \blue \}$ be a coloring of its edges. We say that a cycle is an \emph{almost 
red cycle} if exactly one of its edges is blue, and analogously for almost blue 
cycles. 
We say that $\varepsilon$ is a \emph{NAC-coloring} if it is surjective 
and there are no almost red cycles or almost blue cycles in~$G$.
In other words, every cycle is either monochromatic or has at least two edges 
for each color.
\end{definition}

The main property of NAC-colorings can be stated as follows (see 
\cite[Theorem~3.1]{Grasegger2018}): a graph admits a flexible assignment in the 
plane if and only if it admits a NAC-coloring.

Here we introduce another kind of coloring, called NAP-coloring. 

\begin{definition}
 Let $G = (V,E)$ be a graph, and let $\varepsilon \colon E \longrightarrow \{ 
\red, \blue \}$ be a coloring of its edges. We say that a path $(v,w,z,t)$ in~$G$ is 
an \emph{alternating path} if $\{v,w\}$ and $\{z,t\}$ have the same color and 
$\{w,z\}$ has the opposite color. We say that $\varepsilon$ is a 
\emph{NAP-coloring} (``no alternating path'') if it is surjective, all the $3$-cycles in~$G$
are monochromatic, and there are no alternating paths in~$G$. 
In other words, every edge has an incident vertex such that all its 
incident edges have the same color.
\end{definition}

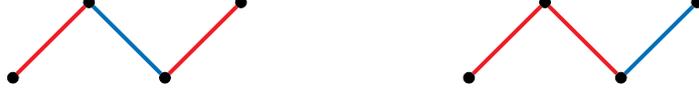
\begin{figure}[ht]
  \begin{center}
    \begin{tikzpicture}
      \begin{scope}
        \node[vertex] (a) at (0,0) {};
				\node[vertex] (b) at (1,1) {};
				\node[vertex] (c) at (2,0) {};
				\node[vertex] (d) at (3,1) {};
				
				\draw[redge] (a)edge(b) (c)edge(d);
				\draw[bedge] (b)edge(c);
      \end{scope}
      \begin{scope}[xshift=6cm]
        \node[vertex] (a) at (0,0) {};
				\node[vertex] (b) at (1,1) {};
				\node[vertex] (c) at (2,0) {};
				\node[vertex] (d) at (3,1) {};
				
				\draw[redge] (a)edge(b) (b)edge(c);
				\draw[bedge] (c)edge(d);
      \end{scope}
    \end{tikzpicture}
  \end{center}
  \caption{On the left, an alternating path. On the right, a path which is 
not alternating.}
\end{figure}

\begin{figure}[ht]
  \begin{center}
    \begin{tikzpicture}
      \begin{scope}
        \node[vertex] (a) at (0,0) {};
				\node[vertex] (b) at (2,1) {};
				\node[vertex] (c) at (2,-1) {};
				\node[vertex] (d) at (4,0) {};
				\node[vertex] (e) at (6,0) {};
				
				\draw[redge] (a)edge(b) (b)edge(d) (b)edge(e);
				\draw[bedge] (a)edge(c) (c)edge(d) (c)edge(e);
      \end{scope}
    \end{tikzpicture}
  \end{center}
  \caption{An example of a NAP-coloring of the complete bipartite graph $K_{3,2}$.}
\end{figure}
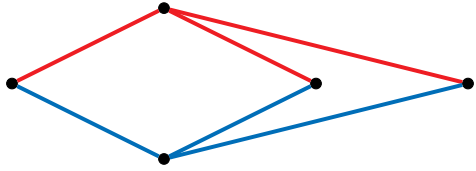
\begin{remark}
\label{remark:NAP_implies_NAC}
 Every NAP-coloring is a NAC-coloring.
\end{remark}

\begin{proposition}
\label{proposition:bond_NAP}
 Let $(I,J)$ be a cut induced by a bond of a graph~$G$. The 
coloring~$\varepsilon_{I,J}$ from \cref{definition:bond_coloring} is a NAP-coloring if and only if there exists a 
non-edge $\{c,d\}$ for which $P_c, Q_c \in I$ and $P_d, Q_d \in J$.
\end{proposition}
\begin{proof}
 Assume that there exists such a non-edge $\{c,d\}$. 
Let us suppose that there exist edges $\{v,w\}$, $\{w,z\}$, and $\{z,t\}$ in~$G$ 
such that $\{v,w\}$ is red, $\{w,z\}$ is blue, and $\{z,t\}$ is red.
Notice that here it may happen that $v = t$. 
Then at least three out of $P_v, P_w, Q_v, Q_w$ are in~$I$, 
and the same holds for $P_z, P_t, Q_z, Q_t$. 
This contradicts the fact that at least three of $P_w, P_z, Q_w, Q_z$ are in~$J$. 
The same argument works if we swap red and blue. 
To conclude the proof, we just need to prove that the coloring is surjective. 
Consider now the non-edge $\{c,d\}$. 
Since the graph is connected, there exist edges $\{c,e\}$ and $\{d,f\}$. 
By construction, $\varepsilon_{I,J}(\{c,e\}) \neq \varepsilon_{I,J}(\{d,f\})$, 
so the coloring is surjective.

Now assume that $\varepsilon_{I,J}$ is a NAP-coloring. 
Then there exist two incident edges $\{c,e\}$ and $\{d,e\}$ of different colors, 
say $\varepsilon_{I,J}(\{c,e\}) = \red$ and $\varepsilon_{I,J}(\{d,e\}) = \blue$. 
Because of the NAP hypothesis, $\{c,d\}$ cannot be an edge of the graph. 
Then we must have that either $P_c, Q_c, P_e \in I$ and $P_d, Q_d, Q_e \in J$, 
or $P_c, Q_c, Q_e \in I$ and $P_d, Q_d, P_e \in J$. 
In both cases the statement holds.
\end{proof}

\begin{theorem}
\label{theorem:NAP_movable}
 Let $G$ be a connected graph without multiedges, or 
self-loops. Then $G$ has a flexible assignment on the sphere 
if and only if it admits a NAP-coloring.
\end{theorem}
\begin{proof}
 Suppose that $G$ has a flexible assignment on the sphere. Then
\cref{proposition:flexible_bond} combined with 
\cref{proposition:bond_NAP} shows that $G$ has a NAP-coloring.

Suppose now that $G$ admits a NAP-coloring. Let $\mathcal{T}$ be 
the set of vertices which are incident to a red as well as a blue edge. No two 
vertices in~$\mathcal{T}$ are adjacent, because this would violate the NAP 
condition. Let $\rho \colon V \longrightarrow S_{\C}^2$ be any realization 
mapping each vertex from~$\mathcal{T}$ either to the North pole $(0,0,1)$ or to
the South pole~$(0,0,-1)$, i.e., $\rho(\mathcal{T}) = \{ (0,0,1), (0,0,-1)\}$. 
Then the blue part of the graph can rotate around the polar axis, 
while keeping the red part fixed (see \cref{figure:flexible_realization}).
This shows that $G$ has a flexible assignment on the sphere.
\end{proof}
\begin{figure}[ht]
  \begin{center}
    \begin{tabular}{m{5cm}m{6cm}}
      \includegraphics[width=4.5cm]{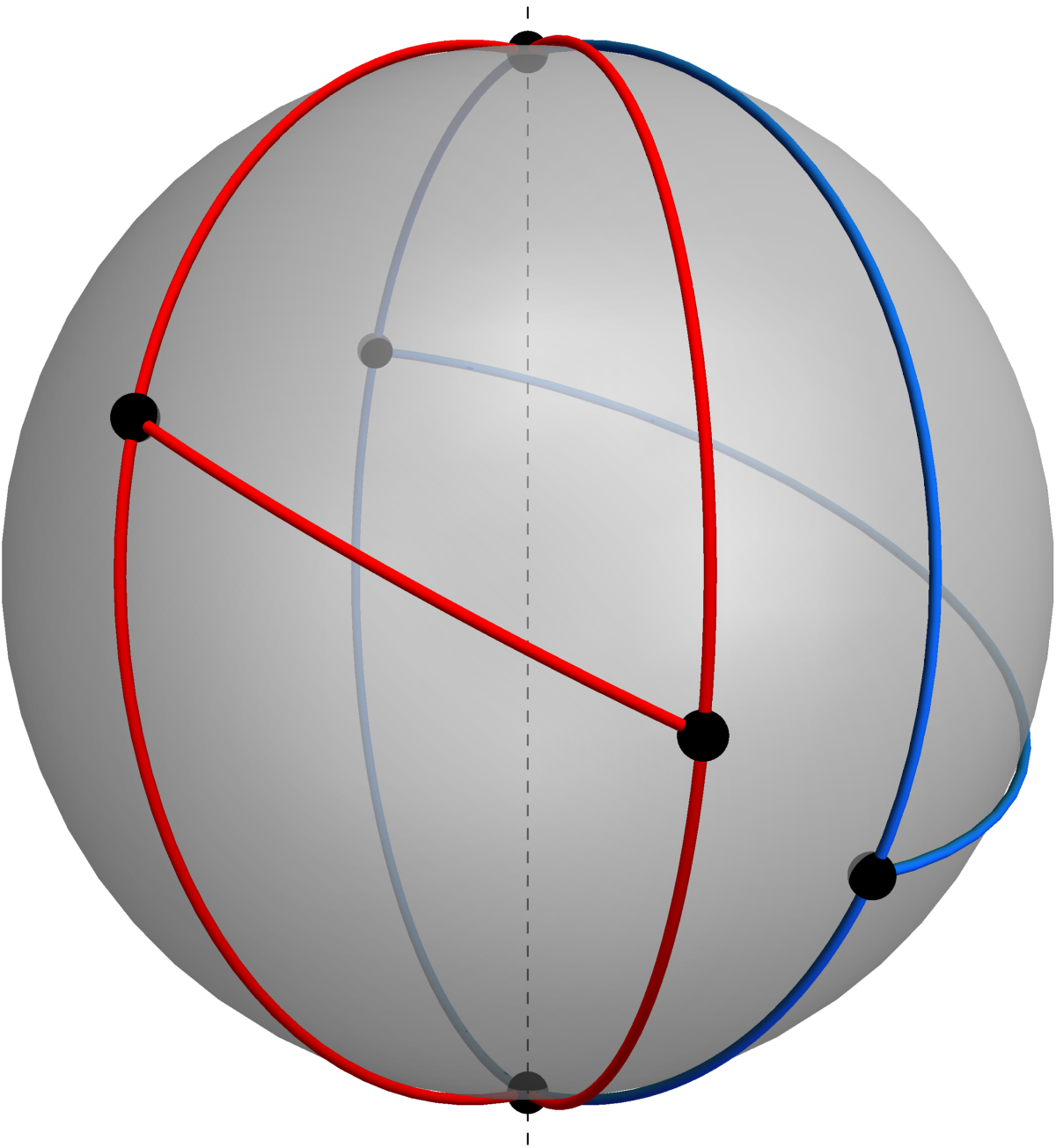} &
      \begin{tikzpicture}[scale=1.5]
        \fill[white] (-0.5,-0.5) rectangle (3,2.5);
        \node[vertex] (a) at (0,0) {};
				\node[vertex] (b) at (2,1) {};
				\node[vertex] (c) at (0,2) {};
				\node[vertex] (d) at (1,0) {};
				\node[vertex] (f) at (1,2) {};
				\node[vertex] (g) at (3,1) {};
				
				\draw[bedge] (a)edge(d) (a)edge(f) (a)edge(c);
				\draw[redge] (b)edge(d) (b)edge(f) (b)edge(g) (d)edge(g) (f)edge(g);
				\draw[bedge] (c)edge(d) (c)edge(f);
      \end{tikzpicture}\\
      \includegraphics[width=4.5cm]{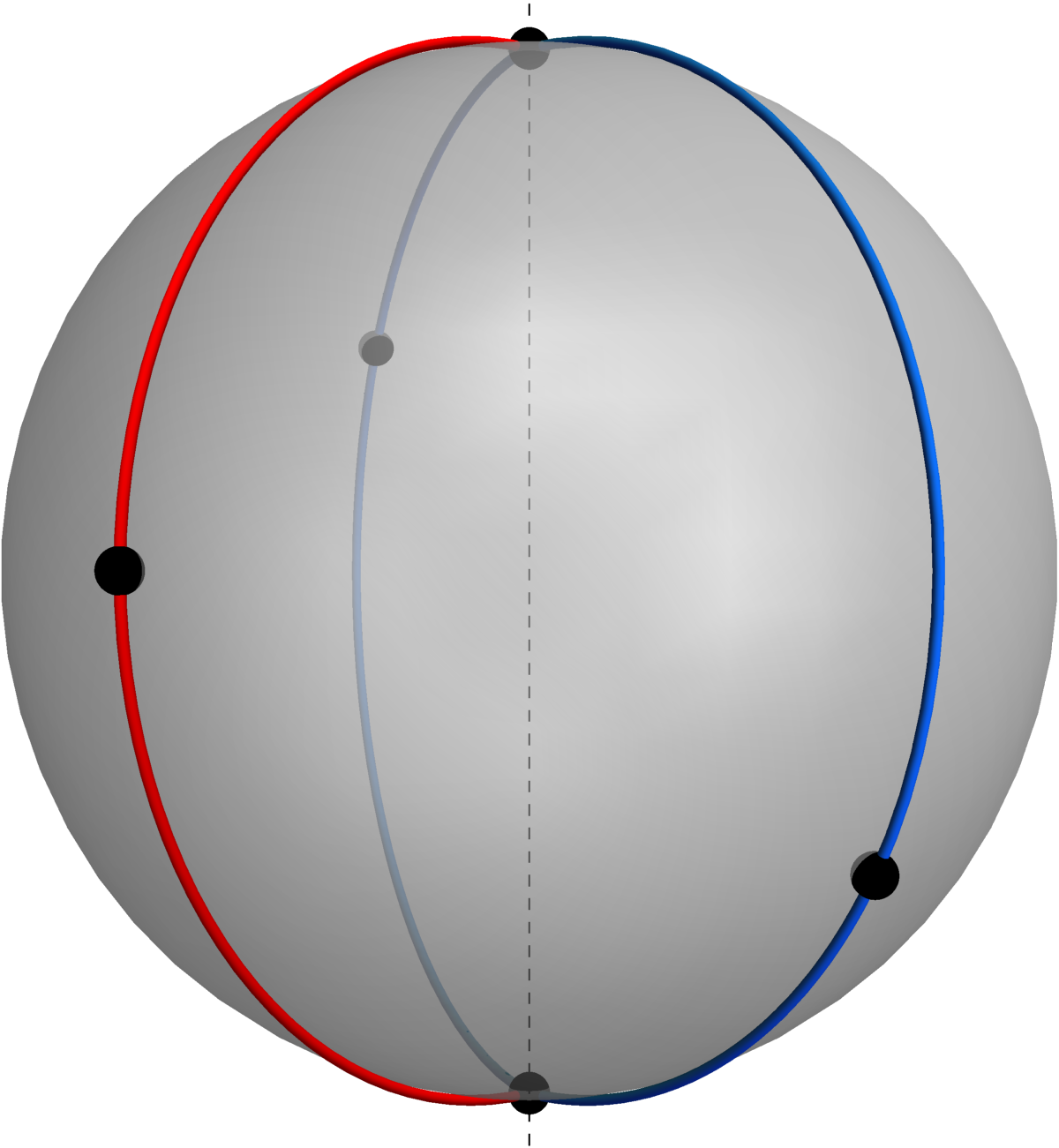} &
      \begin{tikzpicture}[scale=1.5]
        \fill[white] (-0.5,-0.5) rectangle (3,2.5);
        \node[vertex] (a) at (0,0) {};
				\node[vertex] (b) at (2,1) {};
				\node[vertex] (c) at (0,2) {};
				\node[vertex] (d) at (1,0) {};
				\node[vertex] (e) at (1,1) {};
				\node[vertex] (f) at (1,2) {};
				
				\draw[bedge] (a)edge(d) (a)edge(e) (a)edge(f);
				\draw[redge] (b)edge(d) (b)edge(e) (b)edge(f);
				\draw[bedge] (c)edge(d) (c)edge(e) (c)edge(f);
      \end{tikzpicture}
    \end{tabular}
    \caption{The NAP-colorings on the right induce flexible assignments: the two (respectively, three) vertices incident to both blue and red edges
are mapped to two antipodal points on the sphere. For the second graph, this leads to identification of two points: the vertex on top is sent to the
North Pole, and the vertices in the middle and at the bottom are sent to the South Pole.}
    \label{figure:flexible_realization}
  \end{center}
\end{figure}
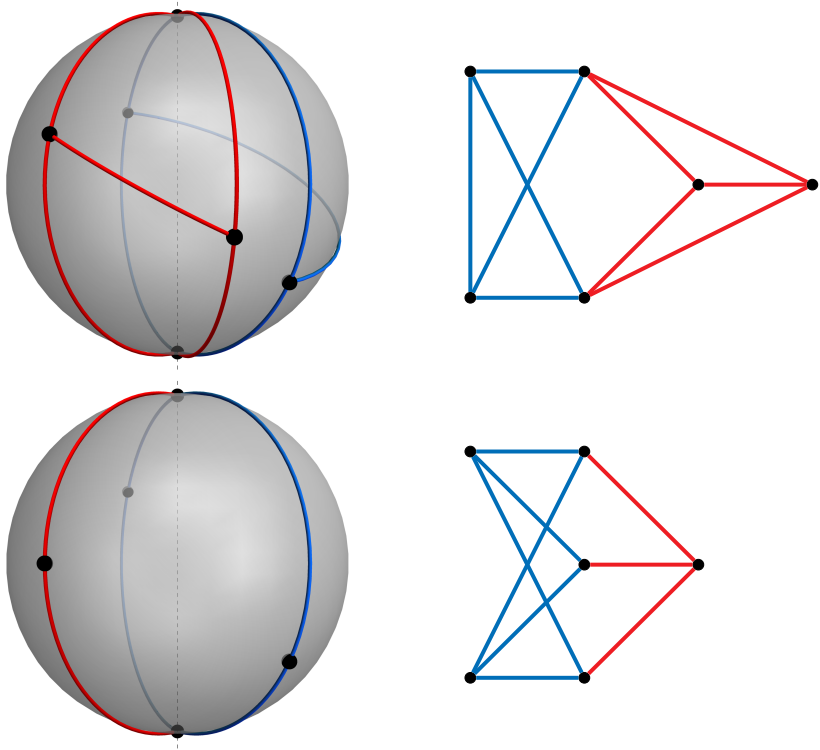

\begin{remark}
 \Cref{theorem:NAP_movable} holds also over the reals in the following sense:
 if we have a flexible assignment with infinitely many, up to rotations, real compatible realizations, 
 then in particular we have a complex one, so we obtain a NAP-coloring;
 vice versa, the construction of the flexible assignment starting from a NAP-coloring 
 can be performed also on the real unit sphere~$S^2 \subset \R^3$.
\end{remark}

From \cref{remark:NAP_implies_NAC}, we get:

\begin{corollary}
 If a graph has a flexible assignment on the sphere, then it has a flexible 
assignment in the plane.
\end{corollary}

In the next section, we describe the real motions on the sphere of the bipartite graph~$K_{3,3}$ for which no two vertices coincide or are antipodal. 
The graph~$K_{3,3}$ is the smallest ``interesting'' Laman graph admitting a NAP-coloring.
In fact, the graph in \cref{figure:NAP_254} is the smallest minimally rigid graph with a NAP-coloring but, as we can easily see, its flexible assignments force some of its vertices to be either coincident or antipodal.
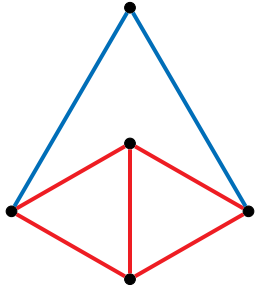
\begin{figure}[ht]
  \begin{center}
    \begin{tikzpicture}[scale=1.8]
      \node[vertex] (1) at (-0.866025,0) {};
      \node[vertex] (2) at (0,0.5) {};
      \node[vertex] (3) at (0,-0.5) {};
      \node[vertex] (4) at (0.866025,0) {};
      \node[vertex] (5) at (0,1.5) {};
      
      \draw[redge] (1)edge(2) (1)edge(3) (2)edge(3) (2)edge(4) (3)edge(4);
      \draw[bedge] (1)edge(5) (4)edge(5);
    \end{tikzpicture}
  \end{center}
  \caption{Smallest minimally rigid graph with a NAP-coloring. In the corresponding flexible assignment, the leftmost and the rightmost
	vertex are sent to the same point or to two antipodal points on the sphere.}
  \label{figure:NAP_254}
\end{figure}

We conclude this section by pointing out how our setup allows us to 
determine necessary conditions for the spherical lengths of the edges of a 
graph having a flexible assignment.

Let us assume that a graph~$G = (V,E)$ has a flexible assignment on the 
sphere, and suppose that $|E| > 2|V|-4$. By assumption, $G$ admits a bond. 
Let $(I,J)$ be a cut determined by this bond, and consider the coloring~$\varepsilon_{I,J}$. 
Consider the divisor~$D_{I,J}$ from \cref{definition:vital_divisor}. 
By \cref{proposition:vital_product}, the divisor~$D_{I,J}$ is 
isomorphic to~$\M_{0, |I|+1} \times \M_{0, |J|+1}$. For each red edge~$\{a,b\}$ 
in~$G$, the cross-ratio $\mathrm{cr}(P_a, P_b, Q_a, Q_b)$ is defined 
on~$\M_{0, |I|+1}$. Therefore, we have a map 
\[
 \M_{0, |I|+1} \longrightarrow 
 \prod_{\substack{\{a,b\} \in E \\ \varepsilon_{I,J}(\{a,b\}) = \red}} 
 \M_{0,4}^{a, b} \,.
\]
If we assume that the number of red edges is bigger than $|I|-2$, the previous 
map cannot be dominant. Hence, its image satisfies at least one algebraic 
equation. Analogously, if the number of blue edges is bigger than $|J|-2$, we 
obtain an algebraic equation between the spherical lengths of blue edges. 

By assumption, the number of edges is bigger than $|I|+|J|-4$, so at least one 
of the previous conditions is met, hence we get an equation for the spherical 
lengths of edges. 
Notice that a class for which the condition $|E|>2|V|-4$ holds is the one of 
Laman graphs (i.e.\ minimally rigid graphs in the plane or, equivalently,
on the sphere).

\section{Flexibility of \texorpdfstring{$K_{3,3}$}{the complete bipartite 
graph 3,3}}
\label{K33}

The goal of this section is to analyze assignments for the spherical lengths of 
the edges of~$K_{3,3}$, the complete bipartite graph with $3+3$ vertices, that 
are flexible on the sphere. In contrast to the first part of the paper, here 
we only consider \emph{real} realizations, and this hypothesis will be 
crucial in several steps of our analysis. We want to classify all motions on 
the sphere of $K_{3,3}$ for which no two vertices coincide or are antipodal. 
After some preliminary discussions, we analyze the possible motions of the 
subgraphs of~$K_{3,3}$ isomorphic to~$K_{2,2}$. After that, we prove that there 
is only a limited number of cases that we need to consider in order to 
classify all motions of~$K_{3,3}$. 
Eventually, we analyze the possible cases one by one.

We fix once and for all the following labeling for the vertices of~$K_{3,3}$:
  \begin{center}
    \begin{tikzpicture}[scale=1.5]
      \begin{scope}
        \node[vertex, label=west:5] (a) at (0,0) {};
				\node[vertex, label=west:3] (b) at (0,1) {};
				\node[vertex, label=west:1] (c) at (0,2) {};
				\node[vertex, label=east:6] (d) at (1,0) {};
				\node[vertex, label=east:4] (e) at (1,1) {};
				\node[vertex, label=east:2] (f) at (1,2) {};
				
				\draw[edge] (a)edge(d) (a)edge(e) (a)edge(f);
				\draw[edge] (b)edge(d) (b)edge(e) (b)edge(f);
				\draw[edge] (c)edge(d) (c)edge(e) (c)edge(f);
      \end{scope}
    \end{tikzpicture}
  \end{center}

This choice has as a consequence, that any subgraph of~$K_{3,3}$ that is isomorphic
to~$K_{2,2}$ has exactly two vertices with even label and two vertices with odd 
label. These are called \emph{even} and \emph{odd vertices}, respectively.

Recall that Dixon described two motions of~$K_{3,3}$ in the plane more than a 
hundred years ago~\cite{Dixon1899}, and Walter and Husty proved that these are 
the only possible ones~\cite{Walter2007}. For the first motion, the 
odd vertices $1$, $3$, and $5$ are placed on a line and the even vertices $2$, 
$4$, and $6$ are placed on another line, perpendicular to the first one. Then, 
the induced distances between adjacent vertices are taken as edge lengths, and 
this constitutes a flexible assignment. The second construction works as 
follows: we consider two rectangles with the same intersection of the diagonals 
such that the edges are parallel/orthogonal to each other. Consider the 
realization of~$K_{4,4}$ where the partition of the vertex set is given by the 
vertices of the two rectangles. This provides a flexible assignment 
for~$K_{4,4}$, from which one can extract an assignment for~$K_{3,3}$ by 
forgetting a pair of vertices. This construction can be described in terms of 
symmetry: take two orthogonal lines, and consider two points in a quadrant 
determined by the two lines; then construct  the eight points given by 
reflecting the two original ones around the two lines. Dixon's motions can be 
also described on the sphere \cite{Wunderlich1976} (compare also 
\cite{Bottema1960, Stachel2013}).

\begin{figure}[ht]
  \begin{center}
    \begin{tikzpicture}[axes/.style={black!50!white,dashed}]
      \begin{scope}
        \draw[axes] (-1.5,0) -- (2.5,0);
				\draw[axes] (0,-2) -- (0,3.5);
        
        \node[vertex, label=west:5] (a) at (0,-1.5) {};
				\node[vertex, label=west:3] (b) at (0,1.8) {};
				\node[vertex, label=west:1] (c) at (0,3) {};
				\node[vertex, label=south:6] (d) at (1.2,0) {};
				\node[vertex, label=south:4] (e) at (-1.2,0) {};
				\node[vertex, label=south:2] (f) at (2,0) {};
				
				\draw[edge] (a)edge(d) (a)edge(e) (a)edge(f);
				\draw[edge] (b)edge(d) (b)edge(e) (b)edge(f);
				\draw[edge] (c)edge(d) (c)edge(e) (c)edge(f);
				
      \end{scope}
    \end{tikzpicture}
    \quad \quad
    \begin{tikzpicture}[axes/.style={black!50!white,dashed}]
      \begin{scope}
        \node[vertex, label=west:5] (a) at (0,0) {};
				\node[vertex, label=east:3] (b) at (3,0) {};
				\node[vertex, label=east:1] (c) at (3,2) {};
				\node[vertex, label=west:6] (d) at (1,-1.5) {};
				\node[vertex, label=east:4] (e) at (2,-1.5) {};
				\node[vertex, label=east:2] (f) at (2,3.5) {};
				\node[vertex, label=west:7] (g) at (0,2) {};
				\node[vertex, label=west:8] (h) at (1,3.5) {};
				
				\draw[axes] (c)--(b)--(a)--(g)--(c)--cycle;
				\draw[axes] (f)--(e)--(d)--(h)--(f)--cycle;
				
				\draw[edge] (a)edge(d) (a)edge(e) (a)edge(f);
				\draw[edge] (b)edge(d) (b)edge(e) (b)edge(f);
				\draw[edge] (c)edge(d) (c)edge(e) (c)edge(f);
				
      \end{scope}
    \end{tikzpicture} 
  \end{center}
  \caption{The two motions of the complete graph $K_{3,3}$ discovered by Dixon. In the motion on the left, three points are
  moving on the $x$-axis, and three points are moving on the $y$-axis. The right is a symmetric motion of $K_{4,4}$ consisting
  of the vertices of two rectangles sharing their symmetry axes, with two points omitted.}
\end{figure}
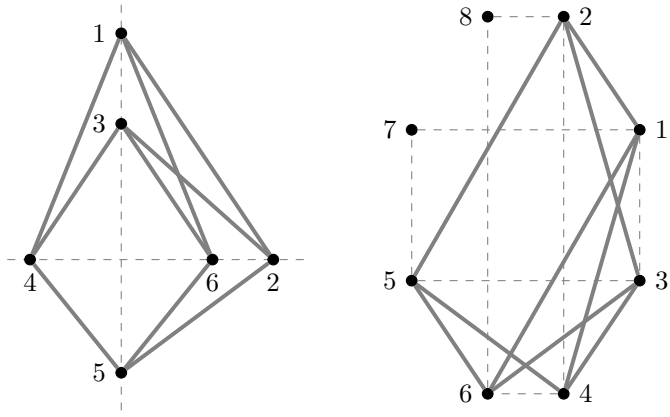

The main result we obtain in this section is the following (see 
\cref{theorem:classification_K33}):

\smallskip
{\noindent \textbf{Classification.}} The only (real) motions of~$K_{3,3}$ on 
the sphere such that no two vertices coincide, or are antipodal, are the two 
analogous motions of the planar ones, and a third new motion.

\smallskip
The classification follows from a careful analysis of the relations between the 
curves of realizations in the moduli spaces~$\M_{0,2n}$ of~$K_{3,3}$ and of its 
subgraphs.

\begin{remark}
One can check that all the NAP-colorings (up to swapping the two colors) of~$K_{3,3}$ are the ones listed in 
\cref{figure:NAP_33}. \Cref{theorem:NAP_movable} ensures the 
existence of flexible assignments, but the construction provided there yields 
non-injective realizations.
\begin{figure}[H]
  \begin{center}
    \begin{tikzpicture}
      \begin{scope}
        \node[vertex] (a) at (0,0) {};
				\node[vertex] (b) at (0,1) {};
				\node[vertex] (c) at (0,2) {};
				\node[vertex] (d) at (1,0) {};
				\node[vertex] (e) at (1,1) {};
				\node[vertex] (f) at (1,2) {};
				
				\draw[redge] (a)edge(d) (a)edge(e) (a)edge(f);
				\draw[bedge] (b)edge(d) (b)edge(e) (b)edge(f);
				\draw[bedge] (c)edge(d) (c)edge(e) (c)edge(f);
      \end{scope}
      \begin{scope}[xshift=2cm]
        \node[vertex] (a) at (0,0) {};
				\node[vertex] (b) at (0,1) {};
				\node[vertex] (c) at (0,2) {};
				\node[vertex] (d) at (1,0) {};
				\node[vertex] (e) at (1,1) {};
				\node[vertex] (f) at (1,2) {};
				
				\draw[bedge] (a)edge(d) (a)edge(e) (a)edge(f);
				\draw[redge] (b)edge(d) (b)edge(e) (b)edge(f);
				\draw[bedge] (c)edge(d) (c)edge(e) (c)edge(f);
      \end{scope}
      \begin{scope}[xshift=4cm]
        \node[vertex] (a) at (0,0) {};
				\node[vertex] (b) at (0,1) {};
				\node[vertex] (c) at (0,2) {};
				\node[vertex] (d) at (1,0) {};
				\node[vertex] (e) at (1,1) {};
				\node[vertex] (f) at (1,2) {};
				
				\draw[bedge] (a)edge(d) (a)edge(e) (a)edge(f);
				\draw[bedge] (b)edge(d) (b)edge(e) (b)edge(f);
				\draw[redge] (c)edge(d) (c)edge(e) (c)edge(f);
      \end{scope}
      
      \begin{scope}[xshift=6cm]
        \node[vertex] (a) at (0,0) {};
				\node[vertex] (b) at (0,1) {};
				\node[vertex] (c) at (0,2) {};
				\node[vertex] (d) at (1,0) {};
				\node[vertex] (e) at (1,1) {};
				\node[vertex] (f) at (1,2) {};
				
				\draw[bedge] (d)edge(a) (d)edge(b) (d)edge(c);
				\draw[bedge] (e)edge(a) (e)edge(b) (e)edge(c);
				\draw[redge] (f)edge(a) (f)edge(b) (f)edge(c);
      \end{scope}
      \begin{scope}[xshift=8cm]
        \node[vertex] (a) at (0,0) {};
				\node[vertex] (b) at (0,1) {};
				\node[vertex] (c) at (0,2) {};
				\node[vertex] (d) at (1,0) {};
				\node[vertex] (e) at (1,1) {};
				\node[vertex] (f) at (1,2) {};
				
				\draw[bedge] (d)edge(a) (d)edge(b) (d)edge(c);
				\draw[redge] (e)edge(a) (e)edge(b) (e)edge(c);
				\draw[bedge] (f)edge(a) (f)edge(b) (f)edge(c);
      \end{scope}
      \begin{scope}[xshift=10cm]
        \node[vertex] (a) at (0,0) {};
				\node[vertex] (b) at (0,1) {};
				\node[vertex] (c) at (0,2) {};
				\node[vertex] (d) at (1,0) {};
				\node[vertex] (e) at (1,1) {};
				\node[vertex] (f) at (1,2) {};
				
				\draw[redge] (d)edge(a) (d)edge(b) (d)edge(c);
				\draw[bedge] (e)edge(a) (e)edge(b) (e)edge(c);
				\draw[bedge] (f)edge(a) (f)edge(b) (f)edge(c);
      \end{scope}
    \end{tikzpicture}
  \end{center}
  \caption{The complete bipartite graph $K_{3,3}$ has exactly 6 NAP-colorings.}
  \label{figure:NAP_33}
\end{figure}
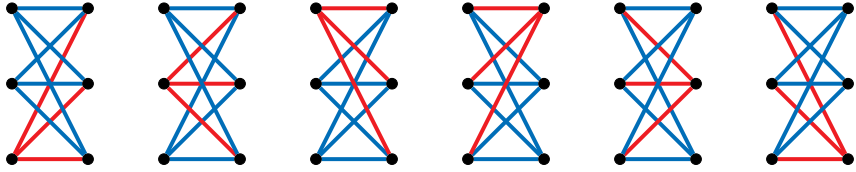
\end{remark}

We start by defining what we mean by ``motion''. Here we make clear that we 
consider real assignments of edge lengths, and real points on the sphere. 
Furthermore, we focus on motions for which no two vertices coincide or are 
antipodal. 

\begin{definition}
 Given a flexible assignment $\lambda \colon E \longrightarrow (0,1)$ for the 
spherical lengths of~$K_{3,3}$, let $\curveD$ be the corresponding fiber of 
the map~$\Phi_{K_{3,3}}$ from \cref{definition:map}, which we call the \emph{configuration curve} of~$K_{3,3}$ 
determined by that assignment. Any positive-dimensional real irreducible 
component~$\curveC$ of~$\curveD$ with infinitely many real elements is called a 
\emph{motion} of~$K_{3,3}$. For every subgraph $G' \subset K_{3,3}$ with 
vertex set~$V' \subseteq \{1, \dotsc, 6\}$, let $\{i_1, \dotsc, i_{6 - |V'|}\} 
:= \{1, \dotsc, 6\} \setminus V'$. We denote by~$\curveC_{i_1 \cdots i_{6 - 
|V'|}}$ the projection of~$\curveC$ via the forgetful map 
 \[
  \M_{0,12} \longrightarrow \M_{0,2|V'|} \,. 
 \]
\end{definition}

We are going to use extensively some particular maps, hence we introduce 
a special notation for them.

\begin{definition}
\label{definition:maps}
 Let $\curveC$ be a motion of~$K_{3,3}$. For all distinct $i,j,k,\ell \in 
\{1, \dotsc, 6\}$, we define the forgetful maps:
 \begin{align*}
  p_i &\colon \curveC \longrightarrow \curveC_{i}, & 
  p_{ij} &\colon \curveC \longrightarrow \curveC_{ij}, \\
  q_i^k &\colon \curveC_{k} \longrightarrow \curveC_{ki}, &
  q_{ij}^k &\colon \curveC_{k} \longrightarrow \curveC_{kij}, \\
  r_i^{k\ell} &\colon \curveC_{k\ell} \longrightarrow \curveC_{k\ell i}, &
  p_{ijk} &\colon \curveC \longrightarrow \curveC_{ijk}.
 \end{align*}
\end{definition}

\begin{definition}
 We say that a motion~$\curveC$ of~$K_{3,3}$ is \emph{proper} if the 
intersection of~$\curveC$ with~$\mscr{M}_{0,12}$ has infinitely many real elements
corresponding to real realizations of~$K_{3,3}$ on the sphere.
Notice that here we ask that the points come from $\mscr{M}_{0,12}$, 
the open subset of~$\M_{0,12}$ whose elements are stable curves with a single component.
\end{definition}

\begin{remark}
 Being proper for a motion is equivalent to the fact that 
there are infinitely many essentially distinct realizations of~$K_{3,3}$ 
where no two vertices coincide or are antipodal. 
In fact, the properness of the motion is equivalent to the fact that we have infinitely many realizations of~$K_{3,3}$ on the (real) sphere where the $6$ vertices determine $12$ distinct left and right lifts.
Now, two points on the sphere have the same left (or right) lift if and only if they belong to the same line on~$S^2_{\C}$. However, two distinct real points on the sphere cannot lie on the same line on~$S^2_{\C}$,
since this line would be real, and no real line is contained in~$S^2_{\C}$.
Moreover, let $T$ and $U$ be real points on the sphere such that the left lift of~$T$
coincides with the right lift of~$U$. Let $T'$ be the antipodal of~$T$.
Since taking antipodals swaps left and right lifts, the points $U$ and $T'$ have the same right lift,
so by what we have just proved they must coincide.
Therefore $U$ is the antipodal of~$T$. 
Hence, asking for a realization of $6$ real points on the sphere 
to have $12$ distinct left and right lifts is equivalent to 
asking that no two points coincide or are antipodal.
\end{remark}

\begin{remark}
 For proper motions, the maps~$p_i$, $q_i^k$, and $r_i^{k \ell}$ from \cref{definition:maps} 
are either birational or \twotoone.
In fact, the possible positions of vertex~$i$ in their fibers 
--- whose number is the degree of the map --- 
are obtained by intersecting circles on the sphere, 
and this may give either one or two intersections.
\end{remark}

We define the spherical analogues of the two motions described by Dixon.
For this, we consider $K_{4,4}$ to be the graph with vertices $\{1, \dotsc, 8\}$
and edges $\bigl\{ \{i,j\} \colon i \in \{1, 3,5,7\}, j \in \{2,4,6,8\} \bigr\}$.

\begin{definition}
 A \emph{Dixon~1 motion} of~$K_{3,3}$ on the sphere is a proper motion in 
which the vertices of the graph move along two orthogonal great circles.
 A \emph{Dixon~2 motion} of~$K_{4,4}$ on the sphere is a proper motion in 
which, up to relabeling the vertices and swapping any vertex with its antipode, 
there exist three involutions (i.e.\ isometries whose square is the identity) 
$o_1$, $o_2$ and~$o_3$ such that $o_1 \, o_2 \, o_3 = \mathrm{id}_{S^2}$ and
\begin{gather*}
 \begin{array}{rcc}
 o_1 \colon &
 \begin{array}{ccc}
  R_1 & \leftrightarrow & R_3, \\
  R_5 & \leftrightarrow & R_7, 
 \end{array}
 &
 \begin{array}{ccc}
  R_2 & \leftrightarrow & R_4, \\
  R_6 & \leftrightarrow & R_8, 
 \end{array}
 \end{array}
 \\
 \begin{array}{rcc}
 o_2 \colon &
 \begin{array}{ccc}
  R_1 & \leftrightarrow & R_7, \\
  R_3 & \leftrightarrow & R_5, 
 \end{array}
 &
 \begin{array}{ccc}
  R_2 & \leftrightarrow & R_8, \\
  R_4 & \leftrightarrow & R_6, 
 \end{array}
 \end{array}
 \\
 \begin{array}{rcc}
 o_3 \colon &
 \begin{array}{ccc}
  R_1 & \leftrightarrow & R_5, \\
  R_3 & \leftrightarrow & R_7, 
 \end{array}
 &
 \begin{array}{ccc}
  R_2 & \leftrightarrow & R_6, \\
  R_4 & \leftrightarrow & R_8, 
 \end{array}
 \end{array}
\end{gather*}
where $(R_1, \dotsc, R_8)$ is any realization of~$K_{4,4}$ in the motion, see \cref{figure:involutions}.
A \emph{Dixon~$2$ motion} of~$K_{3,3}$ is any motion obtained by removing a pair of vertices in a Dixon $2$ motion of~$K_{4,4}$.

\begin{figure}[htb]
\centering
    \begin{tikzpicture}
      \node at (0,0) {\includegraphics[width=.3\textwidth]{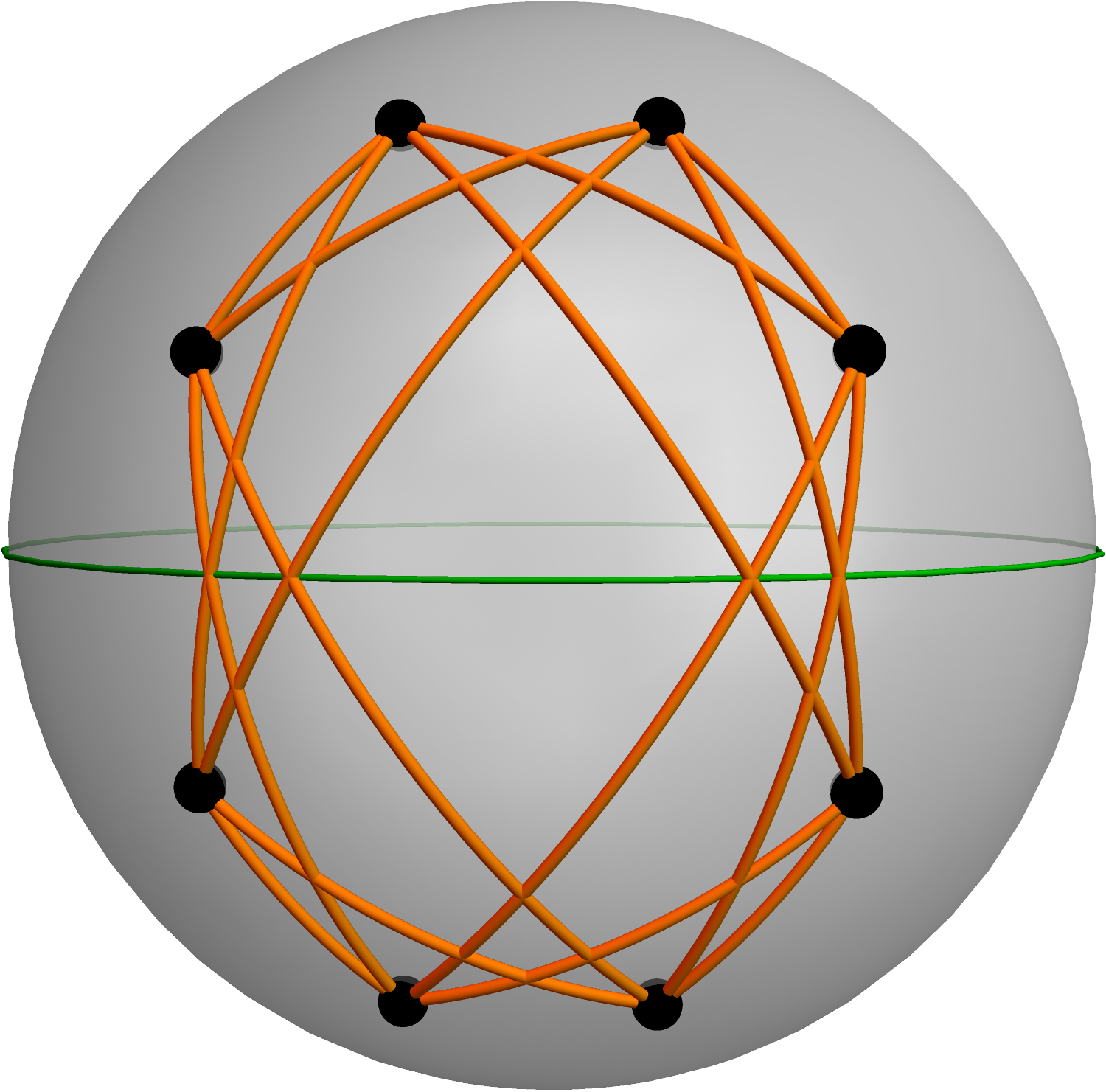}};
      \node at (1.35,0.8) {1};
      \node at (1.35,-0.8) {3};
      \node at (-1.5,-0.8) {5};
      \node at (-1.5,0.8) {7};
      \node at (0.65,1.6) {2};
      \node at (-0.8,1.6) {8};
      \node at (0.65,-1.62) {4};
      \node at (-0.8,-1.62) {6};
	  \node[draw=black] at (0,-2.5) {$o_1$};
	  \node[ForestGreen,scale=1.5] (phantom2) at (-0.12,-0.13) {$\updownarrow$};
    \end{tikzpicture}
    \begin{tikzpicture}
      \node at (0,0) {\includegraphics[width=.3\textwidth]{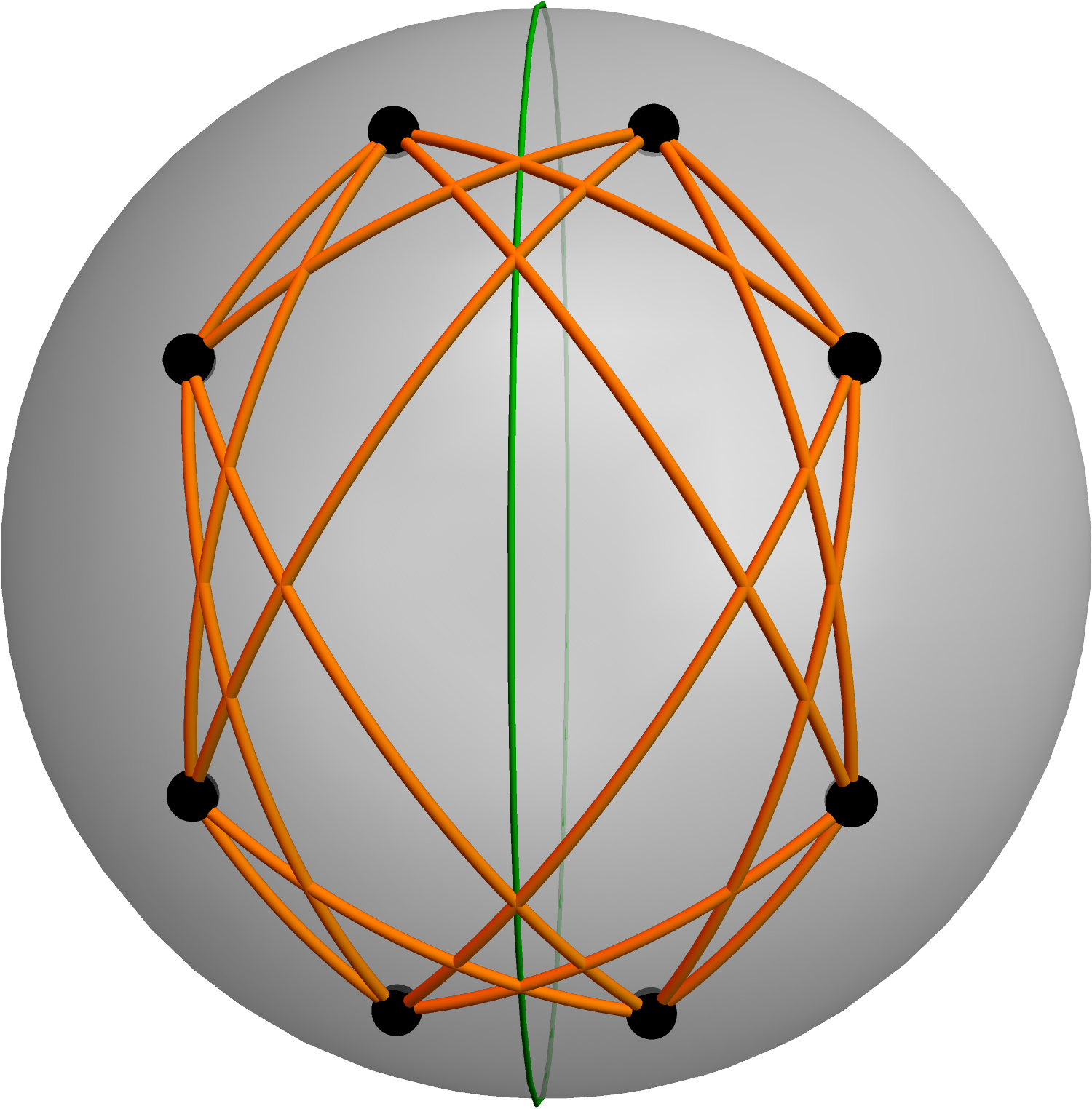}};
      \node at (1.35,0.8) {1};
      \node at (1.35,-0.8) {3};
      \node at (-1.5,-0.8) {5};
      \node at (-1.5,0.8) {7};
      \node at (0.65,1.6) {2};
      \node at (-0.8,1.6) {8};
      \node at (0.65,-1.63) {4};
      \node at (-0.8,-1.63) {6};
	  \node[draw=black] at (0,-2.5) {$o_2$};
	  \node[ForestGreen,scale=1.5] (phantom) at (-0.12,-0.12) {$\leftrightarrow$};
    \end{tikzpicture}
    \begin{tikzpicture}
      \node at (0,0) {\includegraphics[width=.3\textwidth]{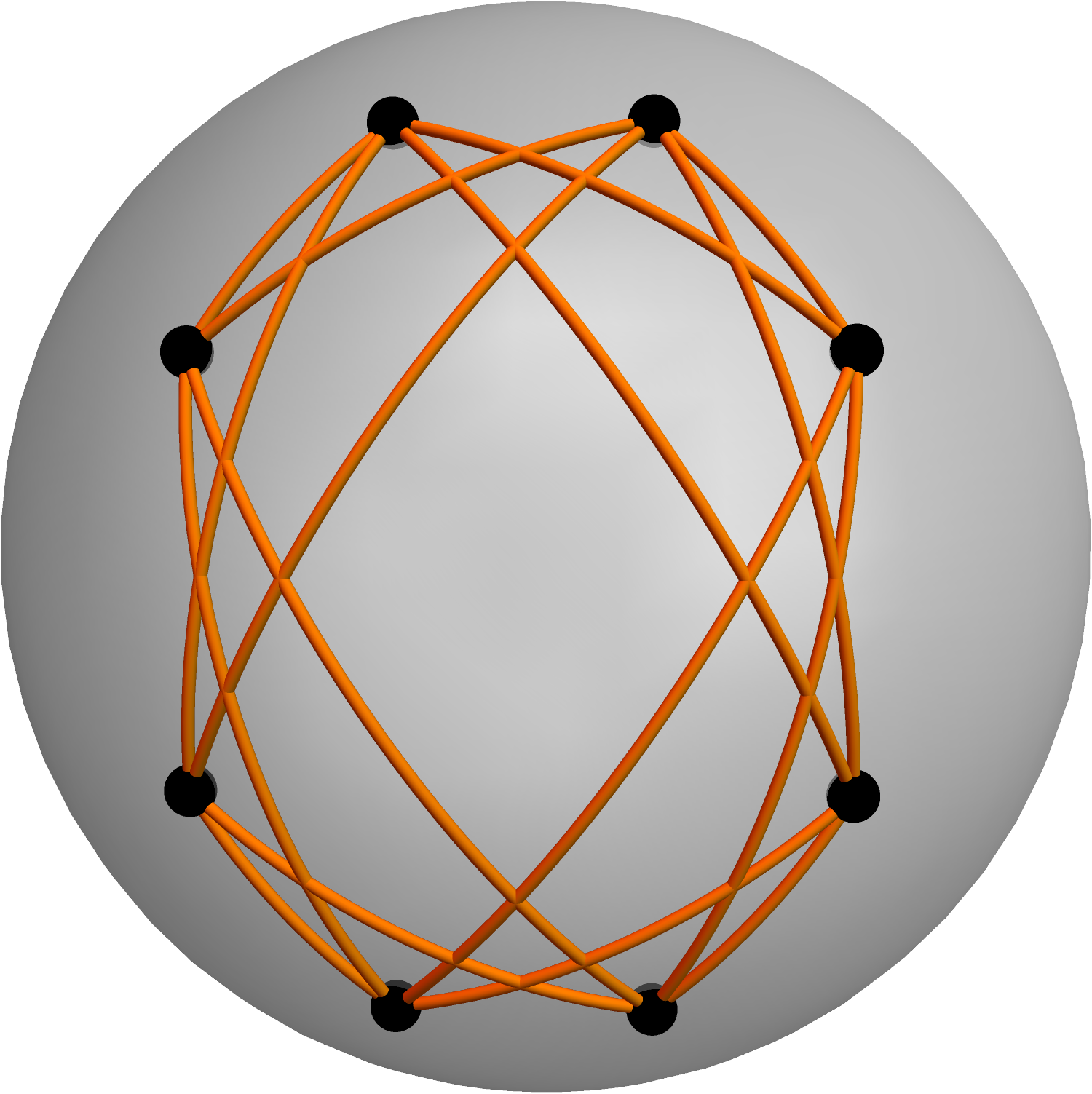}};
      \node at (1.35,0.8) {1};
      \node at (1.35,-0.8) {3};
      \node at (-1.5,-0.8) {5};
      \node at (-1.5,0.8) {7};
      \node at (0.65,1.6) {2};
      \node at (-0.8,1.6) {8};
      \node at (0.65,-1.63) {4};
      \node at (-0.8,-1.63) {6};
	  \node[draw=black] at (0,-2.5) {$o_3$};
	  \node[ForestGreen,scale=1.5] (phantom) at (-0.12,-0.12) {$\circlearrowright$};
    \end{tikzpicture}
    \caption{The three involutions of the sphere describing a Dixon 2 motion.}
    \label{figure:involutions}
\end{figure}
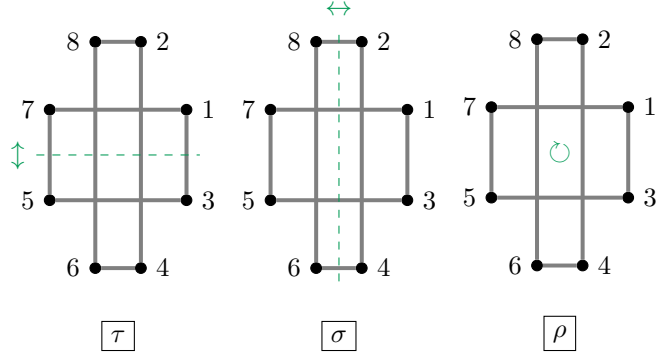
\end{definition}

\begin{lemma}
\label{lemma:dixon}
 Let $\curveC$ be a proper motion of~$K_{3,3}$. If $\deg p_{56} \geq 3$, then 
$\curveC$ is a Dixon~$1$ motion. 
\end{lemma}
\begin{proof}
 Since $p_{56}$ is the composition of~$p_5$ and~$q_6^5$, and since both of 
these maps can be either birational or \twotoone, it follows that the only 
possibility if $\deg p_{56} \geq 3$ is that $\deg p_{56} = 4$. Pick a realization 
$(R_1, \dotsc, R_6) \in (S^2)^6$ of the vertices $1, 2, 3, 4, 5, 6$ in the motion~$\curveC$. 
Let $R_5'$ and $R_6'$ be the two other points appearing in the four realizations 
given by $p_{56}^{-1}\bigl( p_{56}(c) \bigr)$. Then we have
\begin{align*}
 d_{S^2}(R_1,R_6) &= d_{S^2}(R_1,R_6')\,, \\
 d_{S^2}(R_3,R_6) &= d_{S^2}(R_3,R_6')\,, \\ 
 d_{S^2}(R_5,R_6) &= d_{S^2}(R_5,R_6')\,, 
\end{align*}
and this shows that $R_1$, $R_3$, and~$R_5$ are cocircular\footnote{Here by saying that the three points are \emph{cocircular} we mean that they lie on a common geodesic on the sphere. The terminology is motivated by the fact that the geodesics are exactly the great circles.}, because this is the only case where three distinct and not antipodal points can be at the same distances from two other points. By a symmetric 
argument, $R_2$, $R_4$, and~$R_6$ are 
cocircular as well. Moreover, since
\[
 d_{S^2}(R_5',R_6) = d_{S^2}(R_5',R_6') = d_{S^2}(R_5,R_6') \,,
\]
then also $R_6'$ is cocircular with~$R_2$, $R_4$, and~$R_6$.
We now need to prove that these two great circles are orthogonal. Notice that 
the two positions~$R_6$ and~$R_6'$ must be equidistant 
from the great circle~$\wideparen{R_1R_3R_5}$; hence the great 
circle $\wideparen{R_6R_6'}$ is orthogonal 
to~$\wideparen{R_1R_3R_5}$. This shows then that 
$\wideparen{R_2R_4R_6}$ is orthogonal 
to~$\wideparen{R_1R_3R_5}$, and so we have a Dixon~$1$ motion. 
\end{proof}

So from now on we can (and we will) suppose $\deg p_{ij} \leq 2$ for all $i, 
j$. Before proceeding further in the analysis of the mobility of the 
whole~$K_{3,3}$, we focus on how its subgraphs isomorphic to~$K_{2,2}$ may 
move. 

\subsection{Mobility of quadrilaterals}
\label{K33:quadrilaterals}

A thorough discussion of the mobility of quadrilaterals on the sphere, and in particular of their configuration spaces and their irreducible components, is provided in~\cite{Gibson1988a, Gibson1988b}. 
We introduce the following nomenclature for mobile subgraphs of~$K_{3,3}$ isomorphic to~$K_{2,2}$. 
Recall from \cref{definition:spherical_distance} that for~$T, U \in S^2$ we denote $\delta_{S^2}(T,U) := \left\langle T, U \right\rangle$.

\begin{definition}
\label{definition:classification_quadrilaterals}
 Given a proper motion~$\curveC \subseteq \M_{0,12}$ of~$K_{3,3}$ with edge length assignment~$\lambda$, we define 
five families of proper motions of subgraphs of~$K_{3,3}$ isomorphic to~$K_{2,2}$. We call these graphs \emph{quadrilaterals}. To give the definitions, we focus for simplicity on the subgraph~$H_{56}$ induced by the vertices $\{1,2,3,4\}$. Let 
$\curveD_{56} \subseteq \M_{0,8}$ be the configuration curve of~$H_{56}$ defined by the edge lengths induced by~$\lambda$. Note that $\curveD_{56}$ may have several irreducible components. We call the 
motion~$\curveC_{56} := p_{56}(\curveC)$ of~$K_{2,2}$ (which is irreducible by construction):
 \begin{description}
  \item[\caseG:] \emph{general} if and only if $\curveD_{56} = \curveC_{56}$. Here, all the maps $r_1^{56}$, $r_2^{56}$, $r_3^{56}$, and~$r_4^{56}$ are~\twotoone.
  \item[\caseO:] \emph{odd deltoid} if and only if $\curveD_{56}$ has two 
				components, $\curveD_{56} = \curveC_{56} \cup \curveZ$, and $\curveZ$ is a 
				(degenerate) motion. The component $\curveZ$ entirely lies on the boundary; its elements correspond to realizations where the odd vertices coincide or are antipodal (see \cref{figure:degeneratedaltoid}).
				\begin{figure}[ht]
				  \begin{center}
				    \begin{tikzpicture}
				      \node at (0,0) {\includegraphics[width=.3\textwidth]{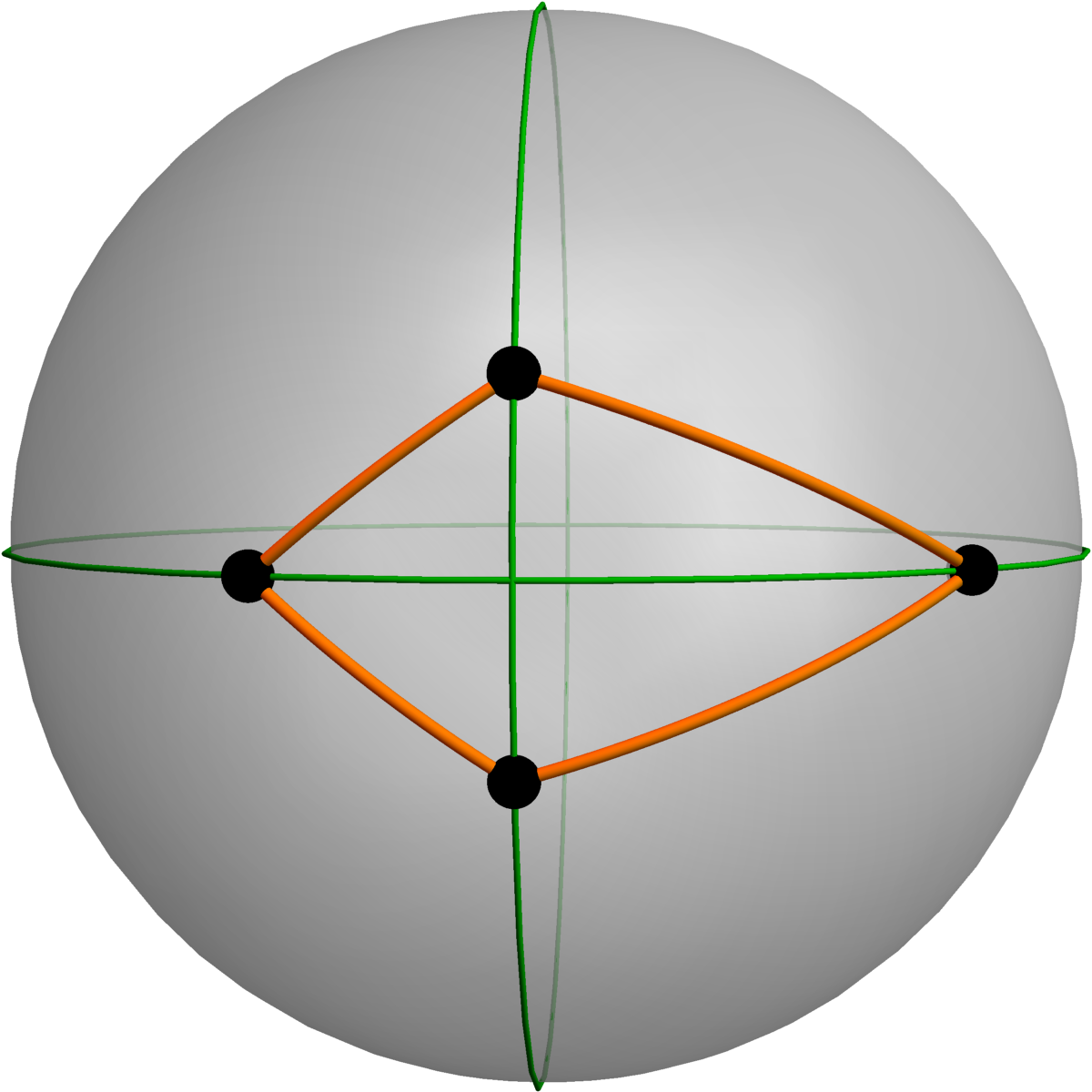}};
				      \node at (0.21,-1.1) {1};
				      \node at (-1.2,0.2) {2};
				      \node at (0.21,0.85) {3};
				      \node at (1.6,0.2) {4};
				    \end{tikzpicture}
				    \begin{tikzpicture}
				      \node at (0,0) {\includegraphics[width=.3\textwidth]{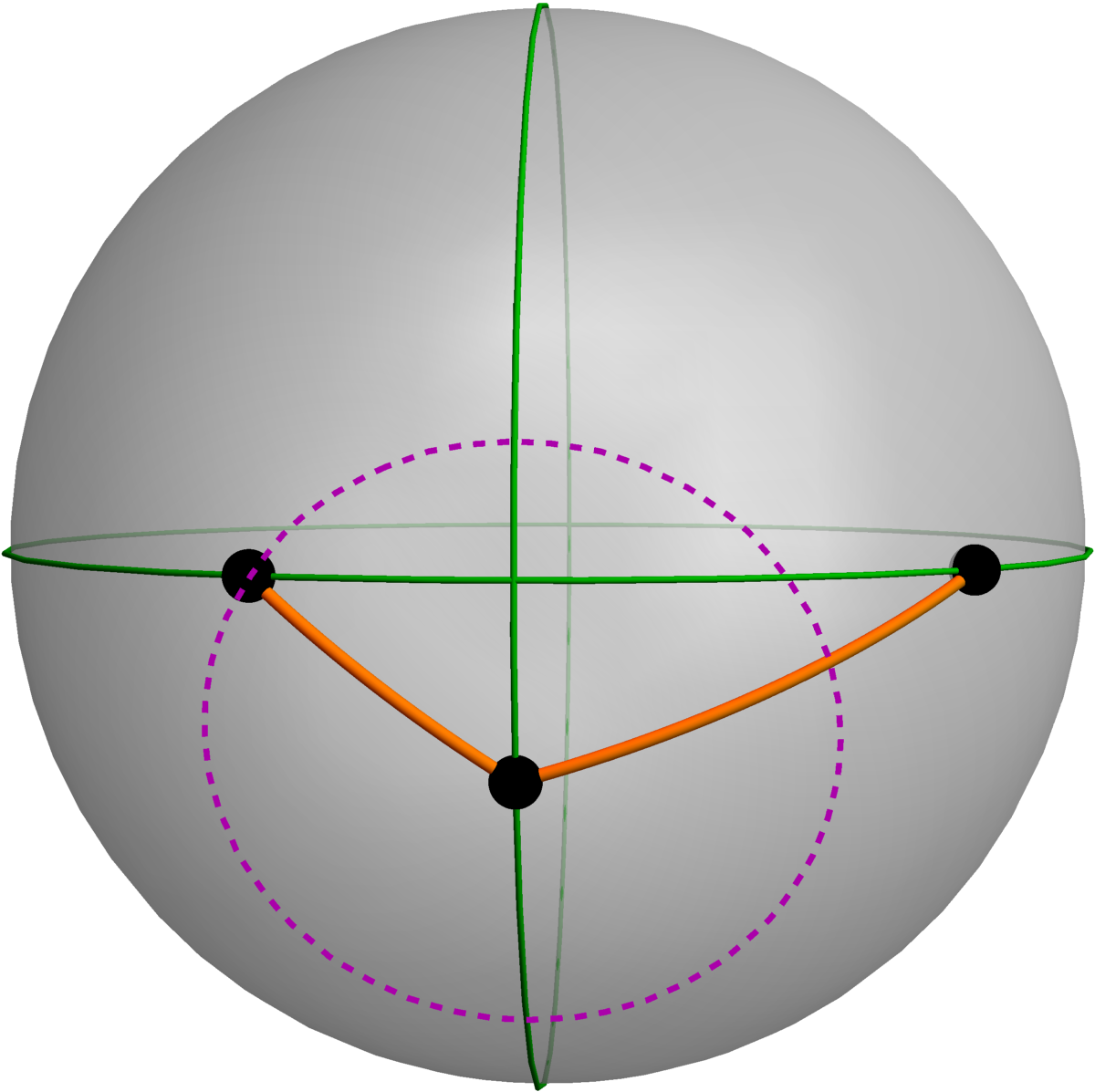}};
				      \node at (0.21,-1.1) {1=3};
				      \node at (-1.2,0.2) {2};
				      \node at (1.6,0.2) {4};
				    \end{tikzpicture}
				    \begin{tikzpicture}
				      \node at (0,0) {\includegraphics[width=.3\textwidth]{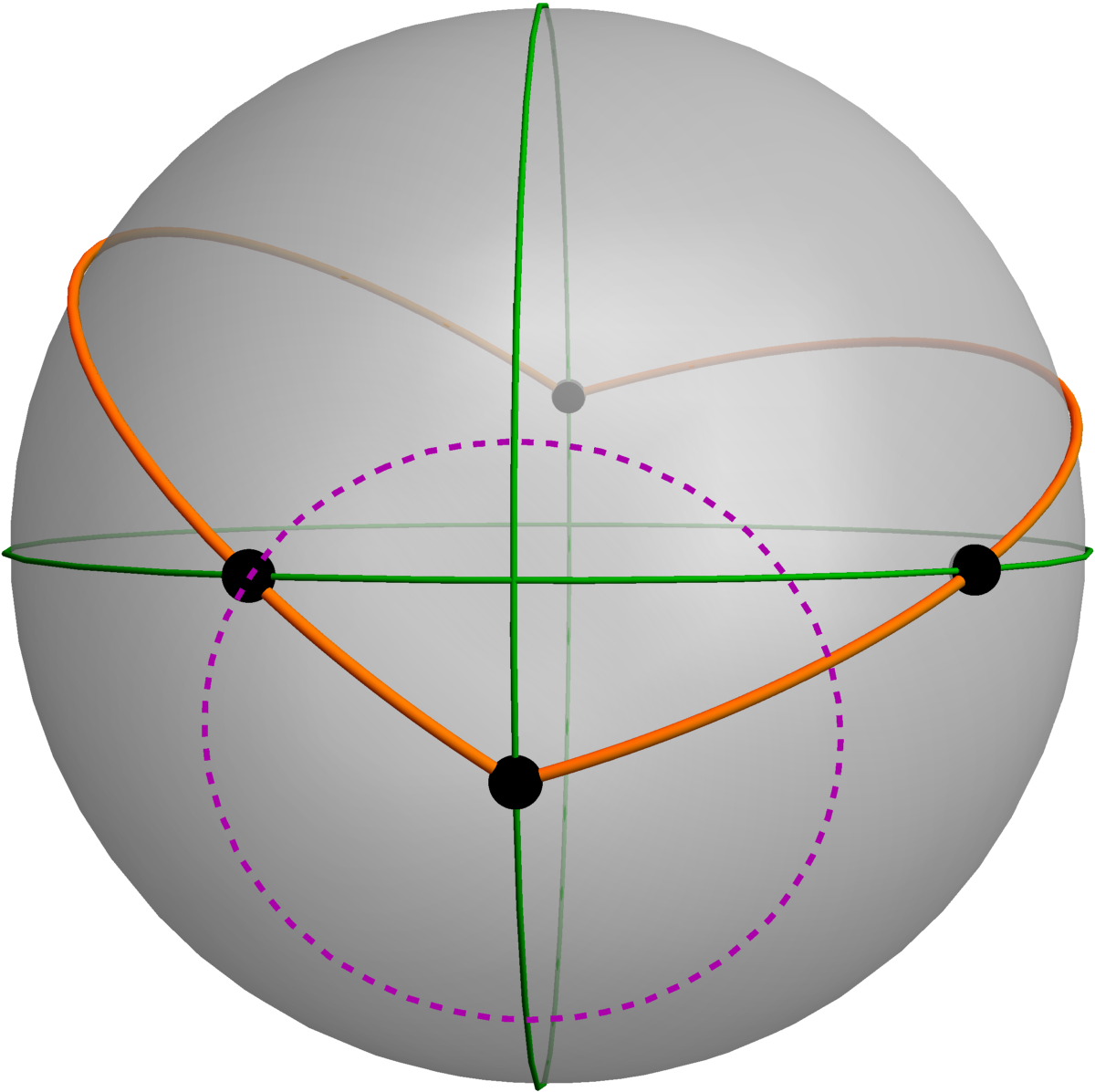}};
				      \node at (-0.3,-1.1) {1};
				      \node at (-1.1,0.2) {2};
				      \node[black!70!white] at (0.21,0.85) {3};
				      \node at (1.5,0.2) {4};
				    \end{tikzpicture}
				  \end{center}
				  \caption{An odd deltoid (left) and two degenerate deltoids where vertices 1 and 3 coincide (middle) or are antipodal (right).
				  Vertex~2 may move along the dashed circle.}
				  \label{figure:degeneratedaltoid}
				\end{figure}
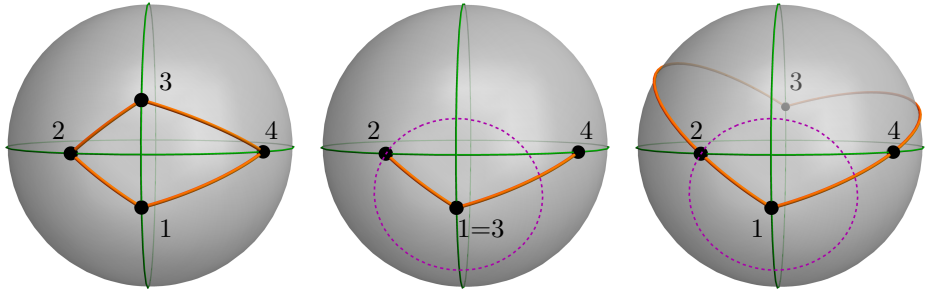
				Here, we have $\deg r_i^{56} = 1$ for $i$ odd, and $\deg r_i^{56} = 2$ for $i$ even.
				Moreover, for all realizations~$(R_1, \dotsc, R_4) \in \bigl(S^2\bigr)^4$ corresponding to elements of~$\curveC_{56} \cap \mscr{M}_{0,8}$, we have that:
  \[
	\delta_{S^2} (1,2) = \alpha \delta_{S^2} (2,3) \,, 
	\quad 
	\delta_{S^2} (3,4) = \alpha \delta_{S^2} (1,4) \,,
  \]
  where $\alpha \in \{1, -1\}$, and $\delta_{S^2} (i,j)$ is a shorthand for $\delta_{S^2} (R_i,R_j)$. The two possible cases are depicted in the 
	following figure, where edges are labeled by the values of~$\delta_{S^2}$ on the corresponding realizations.
	Notice that we can pass from one case to the other by swapping the realization of vertex~$3$ or~$1$ with its antipode.
	For the degenerate component we can see these two cases in \cref{figure:degeneratedaltoid}.
	\begin{center}
	  \begin{tikzpicture}
	    \begin{scope}
	      \node[vertex,label=left:$1$] (1) at (-1,0) {};
	      \node[vertex,label=below:$2$] (2) at (0,-2) {};
	      \node[vertex,label=right:$3$] (3) at (1,0) {};
	      \node[vertex,label=above:$4$] (4) at (0,1) {};
	      \draw[edge] (1)to node[below left] {$a$} (2) (2)to node[below right] {$a$} (3)
	      (3)to node[above right] {$b$} (4) (4)to node[above left] {$b$} (1);
	    \end{scope}
	    \begin{scope}[xshift=4cm]
	      \node[vertex,label=left:$1$] (1) at (-1,0) {};
	      \node[vertex,label=below:$2$] (2) at (0,-2) {};
	      \node[vertex,label=right:$3$] (3) at (1,0) {};
	      \node[vertex,label=above:$4$] (4) at (0,1) {};
	      \draw[edge] (1)to node[below left] {$a$} (2) (2)to node[below right] {$-a$} (3)
	      (3)to node[above right] {$-b$} (4) (4)to node[above left] {$b$} (1);
	    \end{scope}
	  \end{tikzpicture}
	\end{center}
  For an odd deltoid there exists an involution of the sphere fixing the even 
vertices and swapping the odd ones.
  \item[\caseE:] \emph{even deltoid} if and only if the condition for the odd 
	deltoid holds after interchanging even and odd vertices.
    \begin{center}
	  \begin{tikzpicture}
	    \begin{scope}[rotate=90]
	      \node[vertex,label=below:$2$] (1) at (-1,0) {};
	      \node[vertex,label=left:$1$] (2) at (0,1) {};
	      \node[vertex,label=above:$4$] (3) at (1,0) {};
	      \node[vertex,label=right:$3$] (4) at (0,-2) {};
	      \draw[edge] (1)to node[below left] {$a$} (2) (2)to node[above left] {$a$} (3)
	      (3)to node[above right] {$b$} (4) (4)to node[below right] {$b$} (1);
	    \end{scope}
	    \begin{scope}[xshift=5cm,rotate=90]
	      \node[vertex,label=below:$2$] (1) at (-1,0) {};
	      \node[vertex,label=left:$1$] (2) at (0,1) {};
	      \node[vertex,label=above:$4$] (3) at (1,0) {};
	      \node[vertex,label=right:$3$] (4) at (0,-2) {};
	      \draw[edge] (1)to node[below left] {$a$} (2) (2)to node[above left] {$-a$} (3)
	      (3)to node[above right] {$-b$} (4) (4)to node[below right] {$b$} (1);
	    \end{scope}
	  \end{tikzpicture}
	\end{center}
  \item[\caseR:] \emph{rhomboid}%
				\footnote{Here we appeal to the terminology used by Euclid in Book~I, 
Definition~$22$ of the \emph{Elements}, where \emph{rhomboid} denotes a quadrilateral whose opposite sides and 
angles are equal.} if and only if the configuration curve $\curveD_{56}$ 
has two components, $\curveD_{56} = \mcal{Y}_1 \cup \mcal{Y}_2$, both non-degenerate, 
namely no two vertices coincide or are antipodal, and $\curveC_{56}$ is one of them. 
Here, all the maps $r_1^{56}, r_2^{56}, r_3^{56}, r_4^{56}$ are birational. 
For the realizations corresponding to the elements of~$\curveC_{56} \cap \mscr{M}_{0,8}$, we have:
  \[
	\delta_{S^2} (1,2) = \alpha \cdot \delta_{S^2} (3,4) \,, 
	\quad 
	\delta_{S^2} (1,4) = \alpha \cdot \delta_{S^2} (2,3) \,, \\
  \]
  where $\alpha \in \{1, -1\}$. The two possible cases are depicted in the 
following figure: notice that we can pass from one case to the other by 
swapping the realization of vertex~$3$ or~$1$ with its antipode.
  \begin{center}
	  \begin{tikzpicture}
	    \begin{scope}[xscale=1.3]
	      \node[vertex,label=left:$1$] (1) at (-1,0) {};
	      \node[vertex,label=left:$4$] (2) at (-0.8,2) {};
	      \node[vertex,label=right:$3$] (3) at (1.2,1.8) {};
	      \node[vertex,label=right:$2$] (4) at (1,-0.2) {};
	      \draw[edge] (1)to node[left] {$b$} (2) (2)to node[above] {$a$} (3) (3)to node[right] {$b$} (4) (4)to node[below] {$a$} (1);
	    \end{scope}
	    \begin{scope}[xshift=5cm,xscale=1.3]
	      \node[vertex,label=left:$1$] (1) at (-1,0) {};
	      \node[vertex,label=left:$4$] (2) at (-0.8,2) {};
	      \node[vertex,label=right:$3$] (3) at (1.2,1.8) {};
	      \node[vertex,label=right:$2$] (4) at (1,-0.2) {};
	     \draw[edge] (1)to node[left] {$b$} (2) (2)to node[above] {$-a$} (3) (3)to node[right] {$-b$} (4) (4)to node[below] {$a$} (1);
	    \end{scope}
	  \end{tikzpicture}
	\end{center}
For a rhomboid there exists an involution of the sphere swapping both the even 
and the odd vertices. This involution, however, has different nature for the 
two components~$\mcal{Y}_1$ and~$\mcal{Y}_2$: for one it is given by a rotation, while for the other by a 
reflection.
  \begin{figure}[H]
   \begin{tabular}{ccc}
    \includegraphics[width=.3\textwidth]{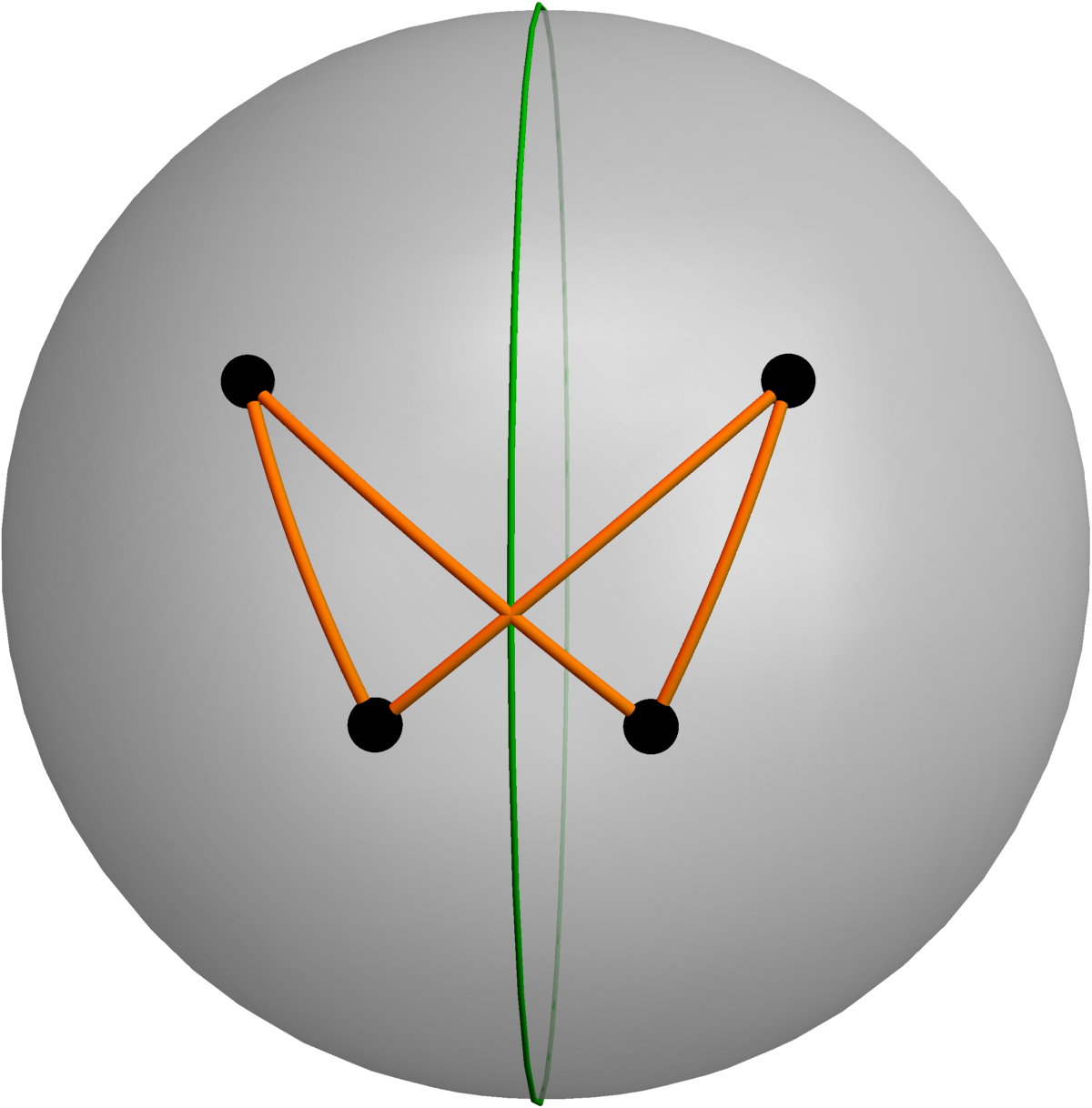} &
    \includegraphics[width=.3\textwidth]{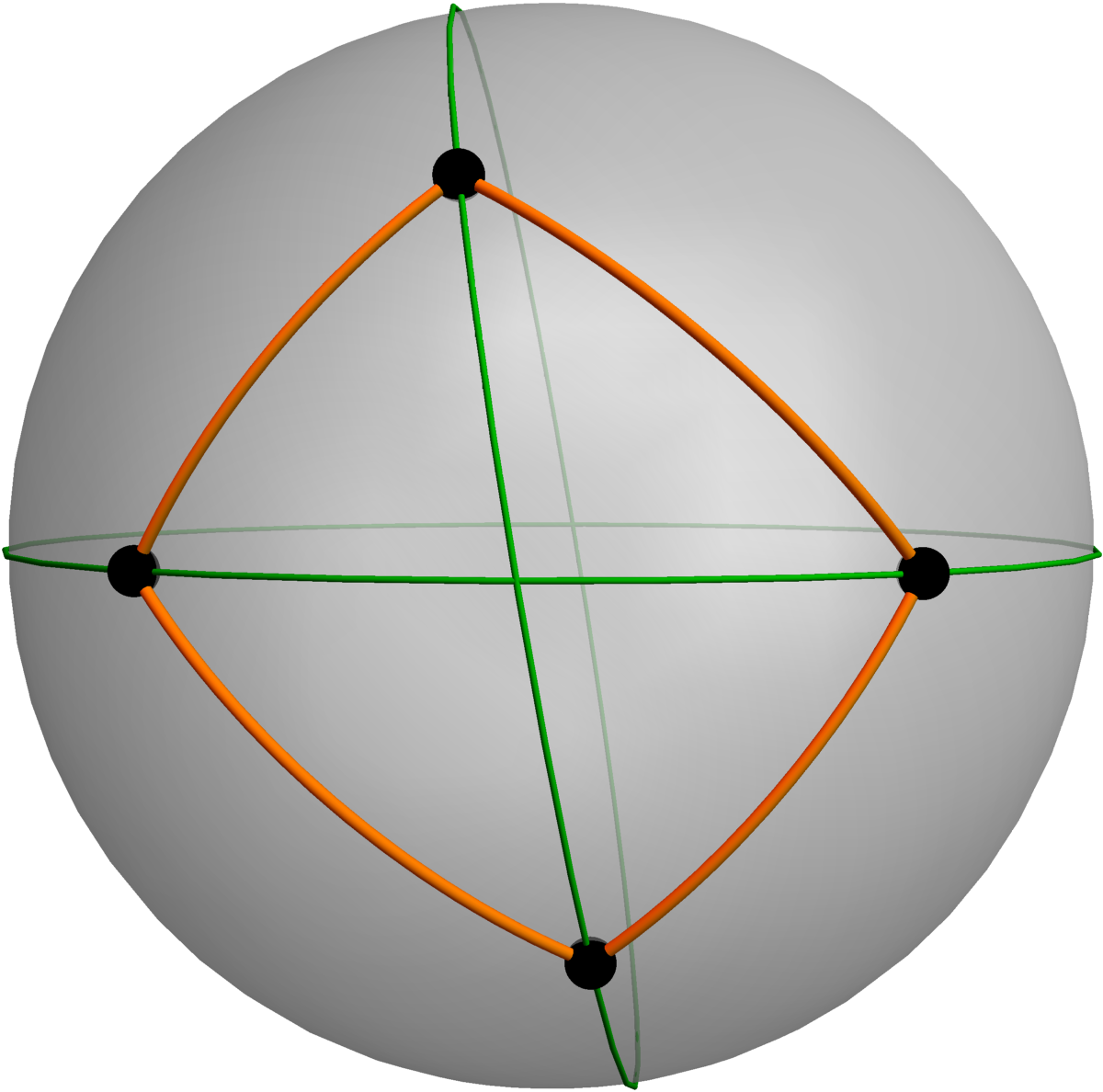} &
    \includegraphics[width=.3\textwidth]{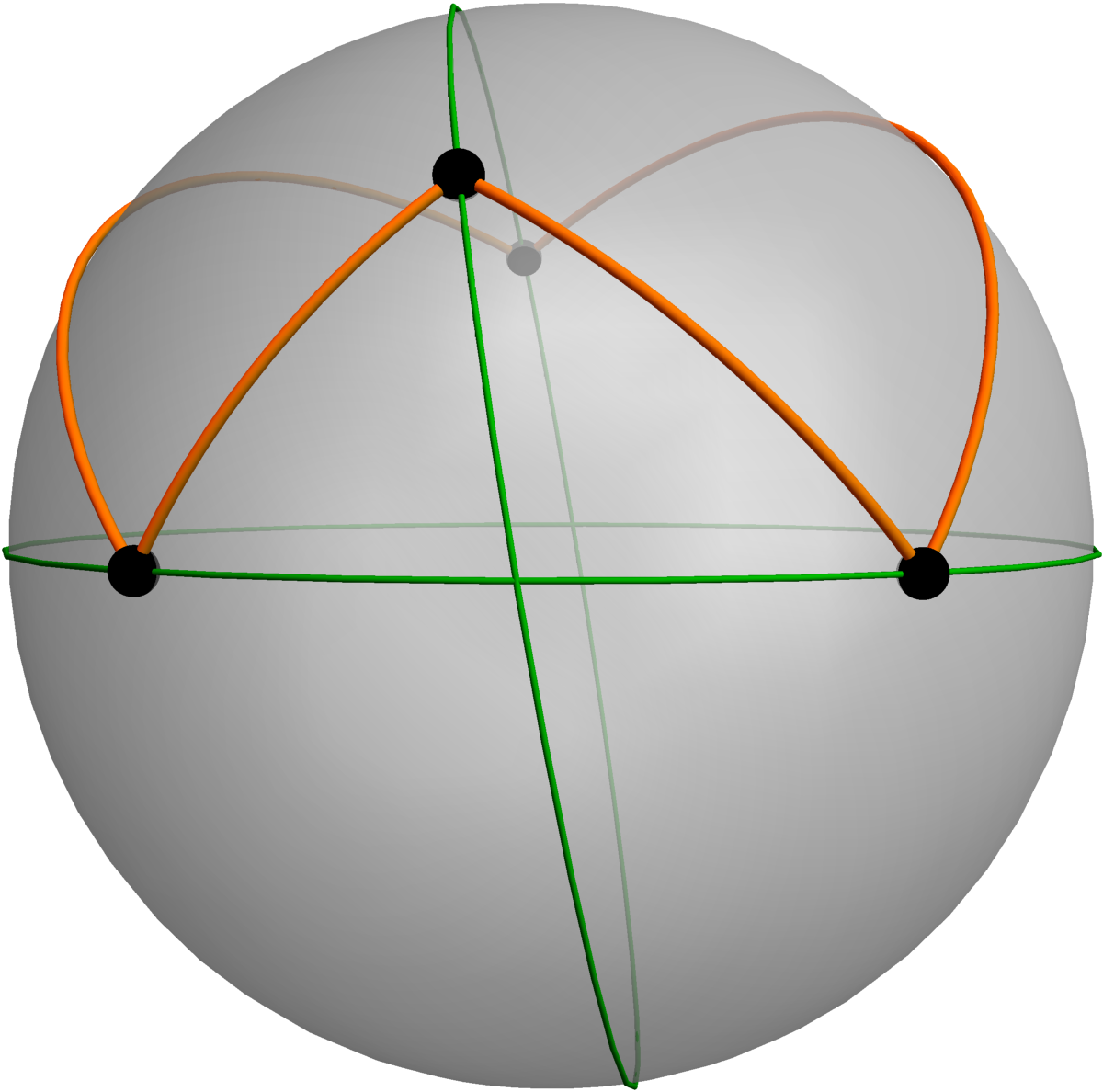}
   \end{tabular}
   \caption{Three rhomboids: The rhomboid on the left has a symmetry plane, which intersects the sphere in the green circle.
    The rhomboid in the middle has a symmetry line, which meets the sphere in the intersection of the two diagonals. The rhomboid
    on the right is obtained from the one in the middle by replacing one of its vertices by its antipodal point.}
  \end{figure}
  \item[\caseL:] \emph{lozenge} if and only if the configuration space~$\curveD_{56}$ 
has three components, $\curveD_{56} = \curveC_{56} \cup \curveZ_1 \cup 
\curveZ_2$, and both $\curveZ_1$ and $\curveZ_2$ are degenerate, namely either the even or the odd vertices coincide. Here, 
all the maps $r_1^{56}, r_2^{56}, r_3^{56}, r_4^{56}$ are birational. Moreover, 
for the realizations corresponding to the elements of~$\curveC_{56} \cap \mscr{M}_{0,8}$, we have that:
  \[
   \delta_{S^2}(1,2) = 
   \alpha \cdot \delta_{S^2}(2,3) = 
   \beta \cdot \delta_{S^2}(3,4) = 
   \gamma \cdot \delta_{S^2}(1,4) \neq 0  \,,
  \]
  where $(\alpha, \beta, \gamma) \in \{(1,1,1),(-1,-1,1), (-1,1,-1), (1,-1,-1)\}$. 
  Notice that the condition on the lengths~$\delta_{S^2}$ being nonzero follows from the hypothesis that no two vertices coincide or are antipodal. The 
four possible cases are depicted in the following figure: we can 
pass from one case to the others by swapping the realizations of some vertices with 
their antipodes.
   \begin{center}
	  \begin{tikzpicture}[scale=0.72]
	    \begin{scope}[]
	      \node[vertex,label=left:$1$] (1) at (-1,0) {};
	      \node[vertex,label=left:$4$] (2) at (-0.8,2) {};
	      \node[vertex,label=right:$3$] (3) at (1.2,1.8) {};
	      \node[vertex,label=right:$2$] (4) at (1,-0.2) {};
	      \draw[edge] (1)to node[left] {$a$} (2) (2)to node[above] {$a$} (3) (3)to node[right] {$a$} (4) (4)to node[below] {$a$} (1);
	    \end{scope}
	    \begin{scope}[xshift=4cm]
	      \node[vertex,label=left:$1$] (1) at (-1,0) {};
	      \node[vertex,label=left:$4$] (2) at (-0.8,2) {};
	      \node[vertex,label=right:$3$] (3) at (1.2,1.8) {};
	      \node[vertex,label=right:$2$] (4) at (1,-0.2) {};
	     \draw[edge] (1)to node[left] {$a$} (2) (2)to node[above] {$-a$} (3) (3)to node[right] {$-a$} (4) (4)to node[below] {$a$} (1);
	    \end{scope}
	     \begin{scope}[xshift=8cm]
	      \node[vertex,label=left:$1$] (1) at (-1,0) {};
	      \node[vertex,label=left:$4$] (2) at (-0.8,2) {};
	      \node[vertex,label=right:$3$] (3) at (1.2,1.8) {};
	      \node[vertex,label=right:$2$] (4) at (1,-0.2) {};
	      \draw[edge] (1)to node[left] {$-a$} (2) (2)to node[above] {$a$} (3) (3)to node[right] {$-a$} (4) (4)to node[below] {$a$} (1);
	    \end{scope}
	    \begin{scope}[xshift=12cm]
	      \node[vertex,label=left:$1$] (1) at (-1,0) {};
	      \node[vertex,label=left:$4$] (2) at (-0.8,2) {};
	      \node[vertex,label=right:$3$] (3) at (1.2,1.8) {};
	      \node[vertex,label=right:$2$] (4) at (1,-0.2) {};
	     \draw[edge] (1)to node[left] {$a$} (2) (2)to node[above] {$a$} (3) (3)to node[right] {$-a$} (4) (4)to node[below] {$-a$} (1);
	    \end{scope}
	  \end{tikzpicture}
	\end{center}
  \begin{figure}[H]
  \centering
   \begin{tabular}{cc}
    \includegraphics[width=.3\textwidth]{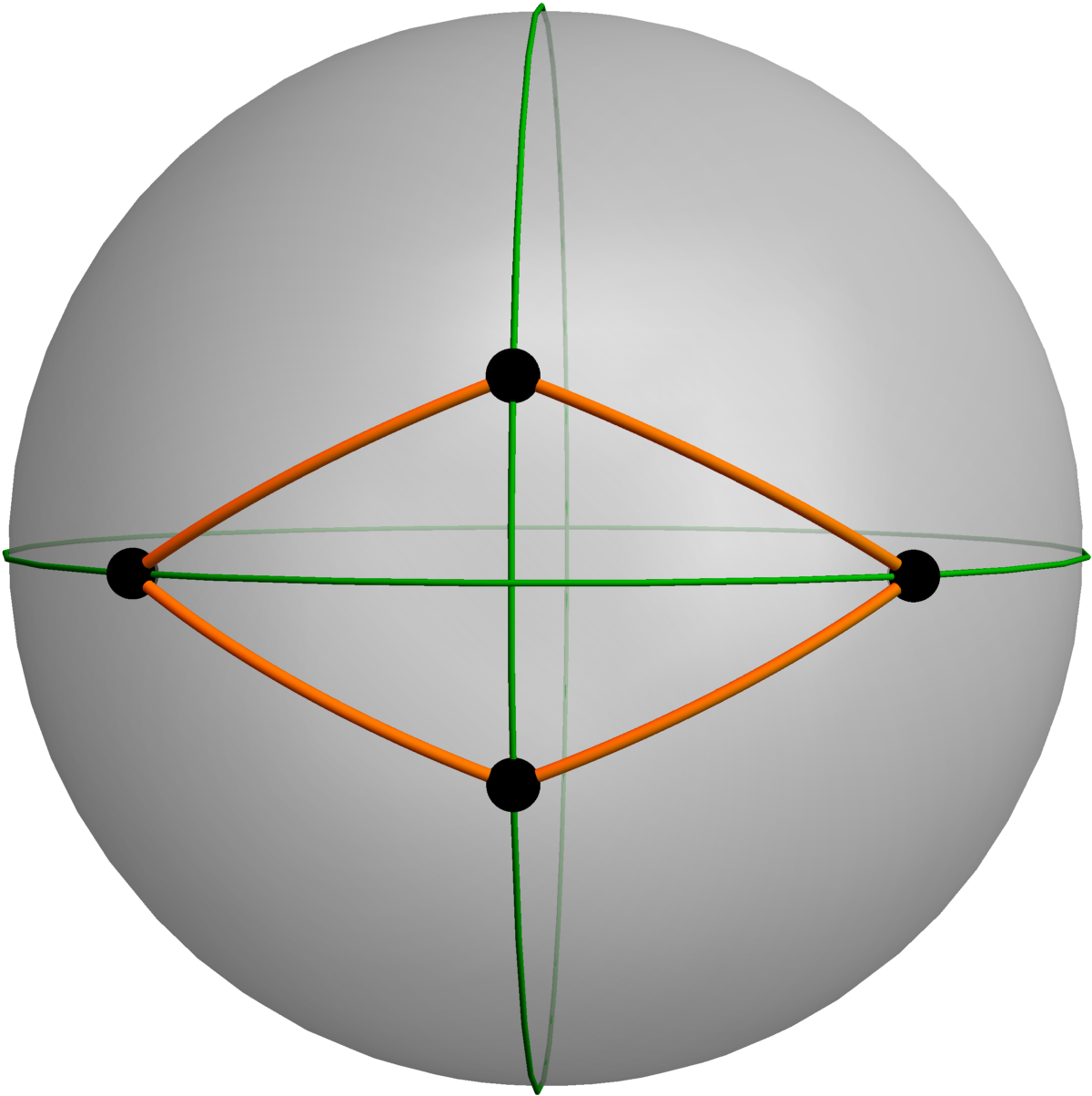} &
    \includegraphics[width=.3\textwidth]{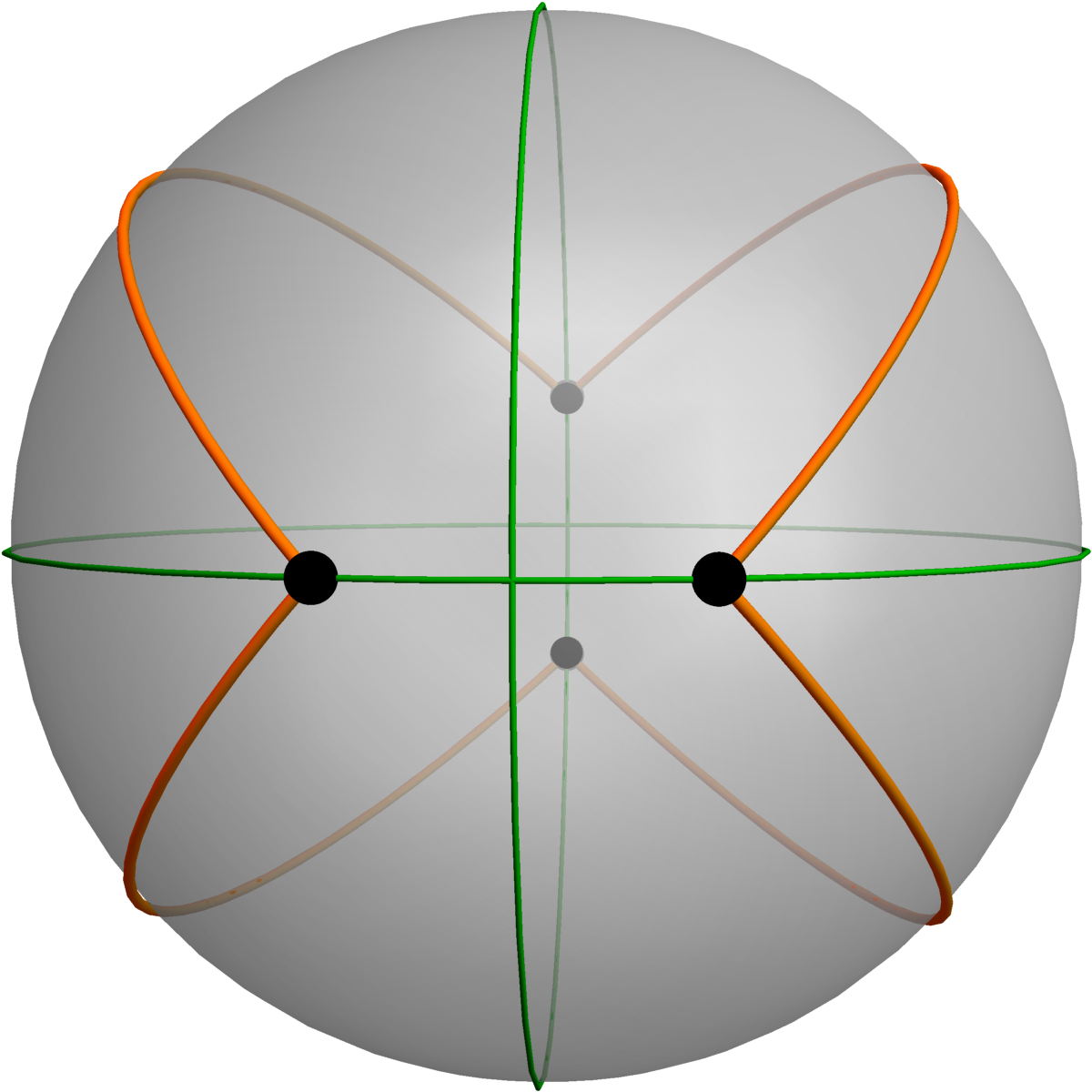}
   \end{tabular}
   \caption{Two lozenges with four equal side lengths. The diagonals are orthogonal and bisect each other.}
  \end{figure}
\end{description}
\end{definition}

\begin{remark}
 A lozenge is not a special case of a rhomboid: in fact, in a 
rhomboid the configuration curve has only two components, and this excludes the 
situation where all edge lengths are equal (up to sign). 
\end{remark}

The fact that the classification provided by 
\cref{definition:classification_quadrilaterals} is complete relies on 
the hypothesis of considering only real proper motions.

\begin{remark}
\label{remark:4cycle}
 A case by case analysis in the classification of
\cref{definition:classification_quadrilaterals} reveals that 
$\deg r_i^{k\ell} = \deg r_j^{k\ell}$ if $i$ and $j$ are vertices of the same parity 
of a $4$-cycle.
\end{remark}

\textbf{Cuts for $K_{2,2}$.}
Due to \cref{proposition:bond_NAP}, there exist four possible cuts~$(I,J)$
for~$K_{2,2}$ (up to swapping $I$ with $J$, and the $P$'s with $Q$'s) that induce a NAP-coloring. They are the following:
\begin{align*}
 \cutou &:= (P_1 Q_1 \, P_2 P_4 | P_3 Q_3 \, Q_2 Q_4) \,, \\
 \cuteu &:= (P_2 Q_2 \, P_1 P_3 | P_4 Q_4 \, Q_1 Q_3) \,, \\
 \cutom &:= (P_1 Q_1 \, P_2 Q_4 | P_3 Q_3 \, Q_2 P_4) \,, \\
 \cutem &:= (P_2 Q_2 \, P_1 Q_3 | P_4 Q_4 \, Q_1 P_3) \,.
\end{align*}
Here, ``$\cuts{o}$'' stands for \emph{odd} vertices being \emph{separated}, 
namely $P_1, Q_1$ belong to~$I$, while $P_3, Q_3$ belong to~$J$. Analogously, 
``$\cuts{e}$'' stands for \emph{even} vertices being \emph{separated}. The 
letters ``$\cuts{m}$'' and ``$\cuts{u}$'' stand for \emph{mixed} and 
\emph{unmixed}, respectively, and describe the relation between the $P$'s and 
$Q$'s in the remaining indices. We use the notation~$\divou$ for the divisor 
corresponding to~$\cutou$, and similarly for the other three cuts.

Our first goal is to understand how a proper motion of a subgraph of~$K_{3,3}$ isomorphic to~$K_{2,2}$ intersects the 
divisors  $\divou, \dotsc, \divem$ in~$\M_{0,8}$. To this end, given a proper 
motion~$\curveC' \subseteq \M_{0,8}$ of~$K_{2,2}$, we define the numbers:
\begin{equation}
\label{equation:mu}
 \begin{array}{ll}
 \mu(\cutou) := \deg( \divou|_{\curveC'} ), &
 \mu(\cutom) := \deg( \divou|_{\curveC'} ), \\
 \mu(\cuteu) := \deg( \diveu|_{\curveC'} ), &
 \mu(\cutem) := \deg( \divem|_{\curveC'} ).
 \end{array}
\end{equation}
Here, the symbol~$\divou|_{\curveC'}$ denotes the restriction of the 
divisor~$\divou$ to~$\curveC'$, and analogously for the other similar 
symbols.
We are going to collect all the possible values for the numbers~$\mu$ in 
\cref{table:mu}.

Let us suppose that $\curveC'$ is general, namely it is the whole 
configuration space of~$K_{2,2}$. Then the curve~$\curveC'$ is defined by 
the four equations
\begin{align*}
 \mathrm{cr}(P_1, Q_1, P_2, Q_2) &= \lambda_{12}, & \mathrm{cr} (P_1, Q_1, P_4, Q_4) &= \lambda_{14}, \\
 \mathrm{cr}(P_2, Q_2, P_3, Q_3) &= \lambda_{23}, & \mathrm{cr} (P_3, Q_3, P_4, Q_4) &= \lambda_{34},
\end{align*}
where the numbers $\lambda_{12}, \dotsc, \lambda_{34}$ are the lengths of the 
corresponding edges of~$K_{2,2}$.
We restrict these equations to the divisor $\divou$. 
This divisor is isomorphic to $\M_{0,5} \times \M_{0,5}$, more precisely 
$\M_{0,P_1 Q_1 P_2 P_4 \ast} \times \M_{0,P_3 Q_3 Q_2 Q_4 \ast}$ (see \cref{proposition:vital_product}). The four equations for the restriction of~$\curveC'$ to $\M_{0,P_1 Q_1 P_2 P_4 \ast} \times \M_{0,P_3 Q_3 
Q_2 Q_4 \ast}$ become two for the left factor, and two for the right one:
\begin{align*}
 \text{left:}  && \mathrm{cr} (P_1, Q_1, P_2, \ast) &= \lambda_{12}, & 
                  \mathrm{cr} (P_1, Q_1, P_4, \ast) &= \lambda_{14}, \\
 \text{right:} && \mathrm{cr} (\ast, Q_2, P_3, Q_3) &= \lambda_{23}, & 
                  \mathrm{cr} (P_3, Q_3, \ast, Q_4) &= \lambda_{34}.
\end{align*}
This system of equations admits a single solution for general values of $\lambda_{12}, \dotsc, \lambda_{34}$. In fact, let us focus on the left factor, and let us first look for solutions in~$\mscr{M}_{0,8}$ (namely, not on the boundary). Then by the action of $\pgl(2,\C)$ we can suppose that $P_1, Q_1$ and $\ast$ are fixed. The two equations determine uniquely $P_2$ and $P_4$, and similarly for the right factor. One can then check that there are no solutions on the boundary. 

We can also use the algorithm developed in~\cite{GalletGraseggerSchicho} to 
compute the degree of the intersection of these equations for general values of 
$\lambda_{12}, \dotsc, \lambda_{34}$, which gives \mbox{$\mu(\cutou)=1$}.
Similarly, we conclude that 
\[
 \mu(\cutou) = 
 \mu(\cutom) = 
 \mu(\cuteu) = 
 \mu(\cutem) = 1\,.
\]

Let us now consider the odd deltoid case. Here, we have a degenerate 
component~$\curveZ$ in the configuration space where the vertices~$1$ and~$3$ 
coincide or are antipodal. This component~$\curveZ$ is then contained in the 
boundary of~$\M_{0,8}$, and in particular, when $1$ and $3$ coincide, in the 
subvariety of the boundary whose general element is:
\begin{center}
\begin{tikzpicture}
  \draw[curveline] (-1,0) -- (-1,3);
  \draw[curveline] (1,0) -- (1,3);
  \draw[curveline] (-4,1.5) -- (4,1.5);
  \draw (-1,0.5) node[markedpoint, label=east:$\mathbf{P_1}$] {};
  \draw (-1,2.5) node[markedpoint, label=east:$\mathbf{P_3}$] {};
  \draw (1,0.5) node[markedpoint, label=west:$\mathbf{Q_1}$] {};
  \draw (1,2.5) node[markedpoint, label=west:$\mathbf{Q_3}$] {};
  \draw (-3,1.5) node[markedpoint, label=north:$\mathbf{P_2}$] {};
  \draw (-2,1.5) node[markedpoint, label=north:$\mathbf{Q_2}$] {};
  \draw (3,1.5) node[markedpoint, label=north:$\mathbf{Q_4}$] {};
  \draw (2,1.5) node[markedpoint, label=north:$\mathbf{P_4}$] {};
 \end{tikzpicture}
\end{center}
Recall in fact from page~\pageref{page:boundary} that when two marked points 
come together (as, in this case, the left lifts of the realizations 
of vertices~$1$ and~$3$, and also the right lifts) then we get a stable curve
with a new component containing the two marked points. Here, this happens 
for the marked points~$P_1$ and~$P_3$, and~$Q_1$ and~$Q_3$. 

Notice that the subvariety from above is the intersection of the two divisors
\[
 D_{P_1 P_3 | P_2 Q_2 P_4 Q_4 Q_1 Q_3} 
 \cap 
 D_{P_1 P_3 P_2 Q_2 P_4 Q_4 | Q_1 Q_3} 
 \,.
\]
We want to understand how $\curveZ$ intersects $\divou$. We have that $\curveZ 
\cdot \divou = 0$ as a consequence of Keel's quadratic relations in the 
description of the Chow ring of~$\M_{0,8}$ (see 
\cite[Section~4, Theorem~1]{Keel1992}). In fact, we have that $\divou \cdot 
D_{P_1 P_3 P_2 Q_2 P_4 Q_4 | Q_1 Q_3} = 0$ since, for example, $P_1$ and $P_3$ 
are split by the partition of~$\divou$, but they are not by the partition 
of $D_{P_1 P_3 P_2 Q_2 P_4 Q_4 | Q_1 Q_3}$. Similarly, we get $\divou \cdot D_{P_1 P_3 | P_2 Q_2 P_4 Q_4 Q_1 Q_3} = 0$.
For the same reason, we have $\curveZ \cdot \divom = \curveZ \cdot 
\divem = 0$. The only number left to compute is $\curveZ \cdot \diveu$.
Notice that $\curveZ \cap \diveu$ is a subvariety of
\[
 \diveu \cap 
 D_{P_1 P_3 | P_2 Q_2 P_4 Q_4 Q_1 Q_3} \cap 
 D_{P_1 P_3 P_2 Q_2 P_4 Q_4 | Q_1 Q_3}
\]
defined by two equations. The last intersection is a surface whose general 
element has the form:
\begin{center}
\begin{tikzpicture}
  \draw[curveline] (-3,0) -- (-3,3);
  \draw[curveline] (1,-2) -- (1,3);
  \draw[curveline] (-4,1.5) -- (3,1.5);
  \draw[curveline] (0,-1) -- (5,-1);
  \draw (-3,0.5) node[markedpoint, label=west:$\mathbf{P_1}$] {};
  \draw (-3,2.5) node[markedpoint, label=west:$\mathbf{P_3}$] {};
  \draw (2,-1) node[markedpoint, label=north:$\mathbf{Q_1}$] {};
  \draw (3,-1) node[markedpoint, label=north:$\mathbf{Q_3}$] {};
  \draw (-2,1.5) node[markedpoint, label=north:$\mathbf{P_2}$] {};
  \draw (-1,1.5) node[markedpoint, label=north:$\mathbf{Q_2}$] {};
  \draw (1,0.5) node[markedpoint, label=west:$\mathbf{Q_4}$] {};
  \draw (1,2.5) node[markedpoint, label=west:$\mathbf{P_4}$] {};
 \end{tikzpicture}
\end{center}
We can then again use the algorithm from~\cite{GalletGraseggerSchicho} to 
compute $\curveZ \cdot \diveu$, which turns out to be~$1$, or do it via a direct computation, as we did before. 
Now that we know how $\curveZ$ intersects the divisors, we can determine how $\curveC'$ 
intersects them by relying on the following fact:

\begin{lemma}
\label{lemma:complete_intersection}
 The configuration space of any~$K_{2,2}$ in~$K_{3,3}$ is a complete 
intersection; in particular, the class of its intersection with the 
divisors~$D_{I,J}$ does not depend on the type of quadrilateral.
\end{lemma}
\begin{proof}
 The configuration space of any subgraph of~$K_{3,3}$ isomorphic to~$K_{2,2}$ 
 given by edge lengths induced by the ones of a flexible assignment of $K_{3,3}$
 is always a curve. At the same time, the configuration space is determined by
 four equations in~$\M_{0,8}$, which is five-dimensional. Hence, it is 
 a complete intersection.
\end{proof}

Denote by~$\curveD'$ the configuration curve of~$K_{2,2}$. 
Using the result about general quadrilaterals, we know that the product of~$\curveD'$ 
with all the four special divisors $\divou, \dotsc, \divem$ in~$\M_{0,8}$ is~$1$. Hence, when $1$ and $3$ 
coincide in the degenerate component of an odd deltoid, then for the non-degenerate component we have:
\begin{align*}
 \curveC' \cdot \divou &= \curveD' \cdot \divou - \curveZ \cdot \divou = 1 - 0 = 1 \,, \\ 
 \curveC' \cdot \divom &= \curveD' \cdot \divom - \curveZ \cdot \divom = 1 - 0 = 1 \,, \\
 \curveC' \cdot \divem &= \curveD' \cdot \divem - \curveZ \cdot \divem = 1 - 0 = 1 \,, \\
 \curveC' \cdot \diveu &= \curveD' \cdot \diveu - \curveZ \cdot \diveu = 1 - 1 = 0 \,.
\end{align*}
If, instead, the points $1$ and $3$ are antipodal, then the degenerate component~$\curveZ$ 
is contained in the subvariety whose general element is
\begin{center}
\begin{tikzpicture}
  \draw[curveline] (-1,0) -- (-1,3);
  \draw[curveline] (1,0) -- (1,3);
  \draw[curveline] (-4,1.5) -- (4,1.5);
  \draw (-1,0.5) node[markedpoint, label=east:$\mathbf{P_1}$] {};
  \draw (-1,2.5) node[markedpoint, label=east:$\mathbf{Q_3}$] {};
  \draw (1,0.5) node[markedpoint, label=west:$\mathbf{Q_1}$] {};
  \draw (1,2.5) node[markedpoint, label=west:$\mathbf{P_3}$] {};
  \draw (-3,1.5) node[markedpoint, label=north:$\mathbf{P_2}$] {};
  \draw (-2,1.5) node[markedpoint, label=north:$\mathbf{Q_2}$] {};
  \draw (3,1.5) node[markedpoint, label=north:$\mathbf{Q_4}$] {};
  \draw (2,1.5) node[markedpoint, label=north:$\mathbf{P_4}$] {};
 \end{tikzpicture}
\end{center}
Arguing as before, we get
\begin{align*}
 \curveC' \cdot \divou &= 1 \,, &
 \curveC' \cdot \divom &= 1 \,, \\
 \curveC' \cdot \divem &= 0 \,, &
 \curveC' \cdot \diveu &= 1 \,.
\end{align*}
The case of the even deltoid is dealt similarly. 

The rhomboid case is different, since there we have two non-degenerate 
components: $\curveD' = \mcal{Y}_1 \cup \mcal{Y}_2$. We know that 
\begin{align*}
 \mcal{Y}_1 \cdot \divou + \mcal{Y}_2 \cdot \divou &= \curveD' \cdot \divou = 1 
\end{align*}
and similarly for the other three divisors $\divom$, $\diveu$, and
$\divem$. In fact, the degrees of the intersections of the whole configuration curve 
with the divisors are the same as in the general case due to \cref{lemma:complete_intersection}.
Consider the situation where the edge lengths are 
\[
 \delta_{S^2}(1,2) = a \,, \quad 
 \delta_{S^2}(2,3) = b \,, \quad 
 \delta_{S^2}(3,4) = a \,, \quad 
 \delta_{S^2}(1,4) = b \,.
\]
The other situation for the edge lengths, where $\delta_{S^2}(1,2) = - \delta_{S^2}(3,4)$ 
and $\delta_{S^2}(1,4) = -\delta_{S^2}(2,3)$, can be analyzed similarly. 
Consider the automorphism~$\sigma$ of~$\M_{0,8}$ that acts as follows on the marked points:
\begin{align*}
   P_1 &\leftrightarrow P_3, & P_2 &\leftrightarrow P_4, &
   Q_1 &\leftrightarrow Q_3, & Q_2 &\leftrightarrow Q_4.
\end{align*}
Look at the equations in~$\M_{0,8}$ that define the configuration curve~$\curveD'$ 
for the situation of edge lengths under consideration. They are of the form:
\begin{align*}
 \crossrat(P_1, Q_1, P_2, Q_2) &= \alpha, & \crossrat(P_2, Q_2, P_3, Q_3) &= \beta, \\
 \crossrat(P_3, Q_3, P_4, Q_4) &= \alpha, & \crossrat(P_1, Q_1, P_4, Q_4) &= \beta,
\end{align*}
where $\alpha$ and $\beta$ depend on~$a$ and~$b$. Then one sees that these equations
are invariant under the action of~$\sigma$. Hence $\sigma$ leaves~$\curveD'$ invariant.
Now consider the component of~$\curveD'$, say~$\mcal{Y}_1$, whose corresponding realizations $(R_1, \dotsc, R_4)$
on the sphere are symmetric with respect to a rotation of the sphere, which swaps~$R_1$
with~$R_3$, and~$R_2$ with~$R_4$. Since, by construction, both a realization and 
its symmetric with respect to the rotation determine the same element on the moduli space~$\M_{0,8}$,
we get that $\sigma$ leaves the component~$\mcal{Y}_1$ pointwise invariant. By inspecting the action of~$\sigma$ on the four cuts~$\cutou$, $\cutom$, $\cuteu$, and~$\cutem$, and taking into account \cref{remark:cut_divisor}, one finds that 
\begin{align*}
  \sigma(\divou) &= \overline{\divou}, & 
  \sigma(\divom) &= \divom, \\
  \sigma(\diveu) &= \overline{\diveu}, & 
  \sigma(\divem) &= \divem,
\end{align*}
where $\overline{(\cdot)}$ denotes complex conjugation. Consider now the automorphism~$\rho$ of~$\M_{0,8}$ that acts as follows on the marked points:
\begin{align*}
   P_1 &\leftrightarrow Q_3, & P_2 &\leftrightarrow Q_4, &
   Q_1 &\leftrightarrow P_3, & Q_2 &\leftrightarrow P_4.
\end{align*}
With similar arguments as before, we see that $\rho$ leaves~$\curveD'$ invariant. Moreover, we know that the component~$\mcal{Y}'$ has a symmetry given by a reflection which swaps~$R_1$ with~$R_3$, and~$R_2$ with~$R_4$. Composing this symmetry with the antipodal map yields a direct isometry of the sphere. Again, a realization and its image under this direct isometry give the same element in the moduli space~$\M_{0,8}$, hence we get that $\rho$ fixes~$\mcal{Y}_2$ pointwise. As we did for~$\sigma$, we find that 
\begin{align*}
  \rho(\divou) &= \divou, & 
  \rho(\divom) &= \overline{\divom}, \\
  \rho(\diveu) &= \diveu, & 
  \rho(\divem) &= \overline{\divem}.
\end{align*}
Now, let us suppose that the curve~$\curveC'$ giving the rhomboid motion is~$\mcal{Y}_1$. 
We claim that $\curveC' \cdot \divou = 0$. In fact, if we had $\curveC' \cdot \divou = 1$, 
then by applying~$\sigma$ we would get $\curveC' \cdot \overline{\divou} = 1$. 
However, due to the relations in the Chow ring of~$\M_{0,8}$ (see \cite[Section~4, Theorem~1]{Keel1992}), 
we know that $\divou \cdot \overline{\divou} = 0$, and so this situation cannot happen. 
Hence we get $\curveC' \cdot \divou = 0$. 
Using the map~$\rho$ on the other component~$\mcal{Y}_2$ one can find that $\curveC' \cdot \divom = 1$, 
and similarly for the intersection with the other two divisors. 
In this way we recover the row of Type~$1$ for the rhomboid in \cref{table:mu}. 
If the curve~$\curveC'$ giving the rhomboid motion is~$\mcal{Y}_2$, 
we recover the row of Type~$4$ for the rhomboid. 
In the situation where the edge lengths are 
\[
 \delta_{S^2}(1,2) = a \,, \quad 
 \delta_{S^2}(2,3) = b \,, \quad 
 \delta_{S^2}(3,4) = -a \,, \quad 
 \delta_{S^2}(1,4) = -b \,,
\]
we perform analogous arguments, this time using the two automorphisms of~$\M_{0,8}$:
\[
\sigma \colon 
\begin{aligned}
   P_1 &\leftrightarrow Q_3, & P_2 &\leftrightarrow P_4, &
   Q_1 &\leftrightarrow P_3, & Q_2 &\leftrightarrow Q_4,
\end{aligned}
\]
and
\[
\rho \colon
\begin{aligned}
   P_1 &\leftrightarrow P_3, & P_2 &\leftrightarrow Q_4, &
   Q_1 &\leftrightarrow Q_3, & Q_2 &\leftrightarrow P_4.
\end{aligned}
\]
In this way we compute all the values in \cref{table:mu} for the rhomboid.

In the case of lozenges, we have that the configuration curve splits into three connected components, two of which are degenerate because either the odd or the even vertices coincide. We can then adopt the same technique we used in the case of the odd deltoid to compute the contributions provided by degenerate components, and so obtaining the $\mu$-numbers for the non-degenerate one. We have four cases, each for a possible configuration of edge lengths as shown in \cref{definition:classification_quadrilaterals}.

We can sum up what we have obtained so far in the \cref{table:mu}.
\begin{table}[ht]
 \caption{Intersections of the projection of a configuration curve on a 
quadrilateral in~$K_{3,3}$ with the boundary divisors of $\M_{0,8}$.
In the deltoid cases, the subcases are distinguished by properties
of the respective residual degenerate components.}
\centering
 \begin{tabular}{cccccc}
  \toprule
   case & subcase & $\mu(\cutom)$ & $\mu(\cutou)$ & 
   $\mu(\cutem)$ & $\mu(\cuteu)$ \\
  \midrule
   \caseG & & 1 & 1 & 1 & 1 \\
  \midrule
   \multirow{2}{*}{\caseO} & $1,3$ \phantomas[l]{ antipodal}{\text{ coincide}} 
    & 1 & 1 
    & 1 & 0 \\
    & $1,3$ \text{ antipodal} & 1 & 1 & 0 & 1 \\
  \midrule
   \multirow{2}{*}{\caseE} & $2,4$ \phantomas[l]{ antipodal}{\text{ coincide}} 
   & 1 & 0 & 1 & 1 \\
   & $2,4$ \text{ antipodal} & 0 & 1 & 1 & 1 \\
  \midrule
   \multirow{4}{*}{\caseR} & Type 1 & 1 & 0 & 1 & 0 \\
   & Type 2 & 0 & 1 & 1 & 0 \\
   & Type 3 & 1 & 0 & 0 & 1 \\
   & Type 4 & 0 & 1 & 0 & 1 \\
  \midrule
   \multirow{4}{*}{\caseL} & Type 1 & 1 & 0 & 1 & 0 \\
   & Type 2 & 0 & 1 & 1 & 0 \\
   & Type 3 & 1 & 0 & 0 & 1 \\
   & Type 4 & 0 & 1 & 0 & 1 \\
  \bottomrule
 \end{tabular}
 \label{table:mu}
\end{table}
\subsection{Classification of forgetful maps}

The goal of this subsection is to determine all the possible cases for the 
degrees of the maps~$\{ p_i \}$, $\{ q_i^k \}$, and~$\{r_i^{k\ell}\}$ once
a proper motion~$\curveC$ of~$K_{3,3}$ is fixed. The first 
result we want to show (\cref{proposition:nodixon_birational}) is 
that, for a proper motion of~$K_{3,3}$ which is not a Dixon~$1$ motion, all the 
maps $p_1, \dotsc, p_6$ are birational.

\Cref{proposition:nodixon_birational} will follow once we prove a 
couple of auxiliary results. We start by stating an elementary fact.

\begin{lemma}
\label{lemma:involution_coplanar}
 Let $T$ and $U$ be points on the sphere $S^2 \subset \R^3$.
 If an involution swaps~$T$ and~$U$, 
 then $T$, $U$ and the axis of the involution are coplanar.
\end{lemma}

To keep notation simple, in the next lemmas we focus on the degree of specific maps 
from \cref{definition:maps}, but the results still hold once 
indices are modified according to automorphisms of~$K_{3,3}$.

\begin{lemma}
\label{lemma:same_degree}
 $\deg q_{13}^6 = \deg q_{15}^6 = \deg q_{35}^6$. 
\end{lemma}
\begin{proof}
Notice that we have the following factorizations:
\begin{center}
\begin{tikzpicture}
  \node (16) at (0,0) {$\curveC_{16}$};
  \node (156) at (0,1.5) {$\curveC_{156}$};
  \node (136) at (0,-1.5) {$\curveC_{136}$};
  \node (6) at (-1.5,0) {$\curveC_{6}$};
  \draw[->] (6) to node[above,font=\tiny] {$q_1^6$} (16);
  \draw[->] (6) to node[above left,font=\tiny] {$q_{15}^6$} (156);
  \draw[->] (6) to node[below left,font=\tiny] {$q_{13}^6$} (136);
  \draw[->] (16) to node[right,font=\tiny] {$r_{5}^{16}$} (156);
  \draw[->] (16) to node[right,font=\tiny] {$r_{3}^{16}$} (136);
\end{tikzpicture}
\end{center}
Since by \cref{remark:4cycle} we have $\deg r_3^{16} = \deg r_5^{16}$,
it follows that $\deg q_{13}^6 = \deg q_{15}^6$. The other equality is proved 
similarly.
\end{proof}

\begin{lemma}
\label{lemma:r2to1_qsnotbirational}
 If $r_1^{56},r_3^{16}$ and $r_5^{36}$ are \twotoone\ maps, then $q_1^6, q_3^6$ 
and $q_5^6$ are not all birational.
\end{lemma}
\begin{proof}
 Since $r_1^{56}$ is \twotoone, then the subgraph~$1234$ is either \caseG\ or \caseE, and 
similarly for the subgraphs~$1245$ and~$2345$. Suppose that all $q_1^6$, $q_3^6$, 
and~$q_5^6$ are birational. We show that this configuration 
is not possible, and to do so we use cuts. We consider the cut~$\cutem $ for~$1234$, namely: 
  \[
   \cutem = (P_2 Q_2 \, P_1 Q_3 | P_4 Q_4 \, Q_1 P_3) \,. \\
  \]
  From \cref{table:mu}, we get $\mu^{56}(\cutem) = 1$, where $\mu^{56}(\cdot)$ are the $\mu$-values for the quadrilateral $1234$ as defined in Equation~\eqref{equation:mu}. Similar relations hold for the other two quadrilaterals $2345$ and $1245$, namely 
$\mu^{16}(\cutem)=1$ and $\mu^{36}(\cutem) = 1$, where $\cutem$ must be 
interpreted according to the context.
   
   By taking the pullback of~$\divem |_{\curveC_{56}}$ under~$q_5^6$ we find the relation:
   \begin{equation}
   \label{equation:ramification}
     \deg q_5^6 \cdot \mu^{56}(\cutem) = 
     \sum_{\substack{(I,J) \\ \text{extending } \cutem}} \mu^{6}(I,J)  \,,
   \end{equation}
   where the numbers $\mu^{6}(I,J)$ are the degrees of the restrictions $D_{I,J}|_{\curveC_6}$.
   In the summation above, we say that $(I,J)$ ``extends'' $(I', J')$ if $I' 
\subseteq I$ and $J' \subseteq J$. The sum on the right hand side has two 
summands --- because we must separate~$P_5$ and $Q_5$ due to 
\cref{lemma:cuts_no_two} --- namely:  
   \begin{align*}
    (I_{P_1 Q_3 P_5}, J_{Q_1 P_3 Q_5}) &:= (P_2 Q_2 \, P_1 Q_3 \, P_5 | P_4 Q_4 \, Q_1 P_3 \, Q_5) \,, \\  
    (I_{P_1 Q_3 Q_5}, J_{Q_1 P_3 P_5}) &:= (P_2 Q_2 \, P_1 Q_3 \, Q_5 | P_4 Q_4 \, Q_1 P_3 \, P_5) \,. 
   \end{align*}
   So Equation~\eqref{equation:ramification} reads (here we use the assumption that $q_5^6$ is birational)
   \[
    1 = \mu^6(I_{P_1 Q_3 P_5}, J_{Q_1 P_3 Q_5}) + \mu^6(I_{P_1 Q_3 Q_5}, J_{Q_1 P_3 P_5}) \,.
   \]
   The same operation can be done starting from the cut~$\cutou$ for 
$2345$ and $1245$, using~$q_1^6$ and $q_3^6$, respectively. This leads to two other 
linear equations for the quantities~$\mu^6(\cdot, \cdot)$. The three linear equations that we 
obtain are of the form
   \begin{align*}
    \mu^6(I_{P_1 Q_3 P_5}, J_{Q_1 P_3 Q_5}) + \mu^6(I_{P_1 Q_3 Q_5}, J_{Q_1 P_3 P_5}) &= 1, \\
    \mu^6(I_{P_1 Q_5 P_3}, J_{Q_1 P_5 Q_5}) + \mu^6(I_{P_1 Q_5 Q_3}, J_{Q_1 P_5 P_3}) &= 1, \\
    \mu^6(I_{P_3 Q_5 P_1}, J_{Q_3 P_5 Q_1}) + \mu^6(I_{P_3 Q_5 Q_1}, J_{Q_3 P_5 P_1}) &= 1. 
   \end{align*}
   Taking into account that 
   \[
    \mu^6(I_{P_i Q_j P_k}, J_{Q_i P_j Q_k}) = \mu^6(I_{Q_i P_j Q_k}, J_{P_i Q_j P_k})
   \]
   because the corresponding divisors are complex conjugated, while $\curveC$ is real, one immediately
   finds that this system has no integer solutions. In fact, summing the left hand sides gives an even
   number, while summing the right hand sides gives~$3$. Thus we reached a contradiction, hence the proof is complete.
\end{proof}

\begin{lemma}
\label{lemma:cocircular_nobirational}
 Suppose that $1$, $3$, and $5$ are cocircular. Then 
not all three maps~$q_1^6$, $q_3^6$, and~$q_5^6$ can be birational.
\end{lemma}
\begin{proof}
 Suppose, for a contradiction, that all the maps~$q_1^6$, 
$q_3^6$, and~$q_5^6$ are birational. \Cref{lemma:same_degree} tells us 
that the degrees of the maps~$q_{13}^6$, $q_{15}^6$, $q_{35}^6$ are the same. 
We distinguish two cases, depending on $\deg q_{13}^6$. 
\begin{description}[leftmargin=.6cm]
  \item[$\deg q_{13}^6 = 1$:]
   Looking at the commutative diagrams as in the proof of 
\cref{lemma:same_degree}, we see that $\deg r_1^{56} = \deg r_3^{16} = 
\deg r_5^{36} = 1$. 
Then, by the classification of moving~$K_{2,2}$ we get that $\curveC_{56}$, 
$\curveC_{16}$, and $\curveC_{36}$ can be \caseO, \caseR, or \caseL. Let us 
suppose that $\curveC_{56}$ is \caseL. This 
forces $\curveC_{16}$ to be \caseL\ as well. However, this contradicts the 
general assumption that no two vertices coincide or are antipodal (in fact, in this case either 
$2$ and $4$, or $1$ and $3$ must coincide or be antipodal). Hence, none of~$\curveC_{56}$, 
$\curveC_{16}$, and~$\curveC_{36}$ can be \caseL. A direct inspection shows 
that not all three of them can be~\caseR, and if one of them is \caseO, then the 
other two must be \caseR, so the latter is the only possibility left. Recall 
from \cref{definition:classification_quadrilaterals} that in this 
case there exist isometries~$o$, $p_1$, and~$p_2$ of the sphere such that:
\begin{itemize}
 \item $o$ swaps $1$ and $3$, and fixes $2$ and $4$;
 \item $p_1$ swaps $3$ and $5$, and $2$ and $4$;
 \item $p_2$ swaps $1$ and $5$, and $2$ and $4$.
\end{itemize}
Using \cref{lemma:involution_coplanar} we see that all the points $1$, 
$2$, $3$, $4$, and $5$ are cocircular, and this violates the hypothesis of 
non-collapsing and non-antipodal vertices.
Hence we reached a contradiction when $\deg q_{13}^6 = 1$.
\item[$\deg q_{13}^6 = 2$:]
From the diagram
\begin{center}
\begin{tikzpicture}
  \node (56) at (0,0) {$\curveC_{56}$};
  \node (156) at (0,-1.5) {$\curveC_{156}$};
  \node (6) at (-1.5,0) {$\curveC_{6}$};
  \draw[->] (6) to node[above,font=\tiny] {$q_5^6$} (56);
  \draw[->] (6) to node[below left,font=\tiny] {$q_{15}^6$} (156);
  \draw[->] (56) to node[right,font=\tiny] {$r_{1}^{56}$} (156);
\end{tikzpicture}
\end{center}
we have that $r_1^{56}$ is a \twotoone\ map, and similarly for $r_3^{16}$ and 
$r_5^{36}$. Hence by \cref{lemma:r2to1_qsnotbirational} we get that not 
all maps $q_1^6$, $q_3^6$ and $q_5^6$ are birational.
 \end{description}
\end{proof}

With this result at hand, we can prove 
\cref{proposition:nodixon_birational}. 

\begin{proposition}
\label{proposition:nodixon_birational}
 Suppose that $\curveC$ is a proper motion of~$K_{3,3}$ which is not a Dixon~1 
motion. Then all the maps $p_1, \dotsc, p_6$ are birational.
\end{proposition}
\begin{proof}
Without loss of generality, we prove that $p_6$ is birational.
By \cref{lemma:dixon}, the maps $p_{16}$, $p_{36}$ and $p_{56}$ can only be
\twotoone\ or birational.
If any of $q_1^6$, $q_3^6$, $q_5^6$ is \twotoone, then $p_6$ is birational, so 
the proof is concluded.
If all maps~$\{q_i^6\}$ are birational, then the vertices $1$, $3$, and $5$ are 
not cocircular by \cref{lemma:cocircular_nobirational}.
Hence, in this case $p_6$ must be birational,
because it is not possible on the sphere that two distinct points are 
equidistant from three non-cocircular points, which would be the case for the 
two elements in a fiber of~$p_6$ if this map were \twotoone.
\end{proof}

We proceed now by classifying the possible degrees of the maps~$\{ q_{i}^k \}$ 
and~$\{ r_{i}^{k\ell} \}$.

\begin{lemma}
\label{lemma:determined}
 Suppose that $\curveC$ is a proper motion of~$K_{3,3}$. The numbers~$\{ \deg 
r_i^{k\ell} \}$ are completely determined by the numbers~$\{ \deg q_i^j \}$. 
\end{lemma}
\begin{proof}
 Consider the $K_{3,2}$ subgraph of~$K_{3,3}$ spanned by the vertices 
$1,2,3,4,5$: there are exactly three maps from the motion~$\curveC_6$ to a 
motion of a subgraph of this~$K_{3,2}$ isomorphic to~$K_{2,2}$, namely:
\[
 q_1^6 \colon \curveC_6 \longrightarrow \curveC_{16}, \quad
 q_3^6 \colon \curveC_6 \longrightarrow \curveC_{36}, \quad
 q_5^6 \colon \curveC_6 \longrightarrow \curveC_{56}.
\]
For each $q$-map, say $q_1^6$, we consider two $r$-maps, namely $r_5^{16}$ and 
$r_3^{16}$. We can arrange all these maps in a hexagonal diagram: 
\begin{center}
\begin{tikzpicture}[scale=2]
  \node[] (c16) at (0,1) {$\curveC_{16}$};
  \node[] (c6) at (0,0) {$\curveC_{6}$};
  \node[] (c136) at (-1,0.5) {$\curveC_{136}$};
  \node[] (c156) at (1,0.5) {$\curveC_{156}$};
  \node[] (c36) at (-1,-0.5) {$\curveC_{36}$};
  \node[] (c56) at (1,-0.5) {$\curveC_{56}$};
  \node[] (c356) at (0,-1) {$\curveC_{356}$};
  
  \draw[->] (c6) to node[right,font=\scriptsize] {$q_1^6$} (c16);
  \draw[->] (c6) to node[below,font=\scriptsize] {$q_3^6$} (c36);
  \draw[->] (c6) to node[below,font=\scriptsize] {$q_5^6$} (c56);
  
  \draw[->] (c16) to node[above,font=\scriptsize] {$r_3^{16}$} (c136);
  \draw[->] (c16) to node[above,font=\scriptsize] {$r_5^{16}$} (c156);
  \draw[->] (c36) to node[left,font=\scriptsize] {$r_1^{36}$} (c136);
  \draw[->] (c36) to node[below,font=\scriptsize] {$r_5^{36}$} (c356);
  \draw[->] (c56) to node[right,font=\scriptsize] {$r_1^{56}$} (c156);
  \draw[->] (c56) to node[below,font=\scriptsize] {$r_3^{56}$} (c356);
\end{tikzpicture}
\end{center}
We can now start with the proof. Let us first suppose that not all $q$-maps 
have the same degree. Then without loss of generality we can suppose that 
\[
 \deg q_3^6 = 1, \quad \deg q_5^6 = 2 \,.
\]
This implies that $\deg r_5^{36} = 2$ and $\deg r_3^{56} = 1$. We know that 
$\deg r_5^{36} = \deg r_1^{36}$, and similarly for the other pairs of $r$-maps 
(see \cref{remark:4cycle}). If we denote by~$x$ the number~$\deg q_1^6$, 
and by~$y$ the number $\deg r_3^{16} = \deg r_5^{16}$, we find the equation $xy 
= 2$ from the commutativity of the diagram. Hence the missing degrees of the 
$r$-maps are determined by the degrees of the $q$-maps. 

Let us now suppose that all the $q$-maps have the same degree. Then we know 
from \cref{remark:4cycle} that all the $r$-maps have the same degree, but 
we do not know whether it equals the degree of the $q$-maps or not.  
We claim that in this case the degrees of the $q$- and $r$-maps are the same.
Assume that the $r$-maps have degree two.
Then, the $q$-maps cannot be all birational by 
\cref{lemma:r2to1_qsnotbirational},
namely they all must have degree two.
On the other hand, assume that all $r$-maps are birational.
Looking at the preimages of the elements of~$\curveC_{356}$,
the positions of vertices $3$ and $5$ are uniquely determined by~$1$, $2$ 
and~$4$.
Therefore, $q^6_3$ and $q^6_5$ are birational. Similarly for the map~$q^6_1$.
\end{proof}

\begin{remark}
\label{remark:same_degree_q}
 As a consequence of the proof of \cref{lemma:determined} we have that,
 for a fixed $k \in \{1, \dotsc, 6\}$:
\[
 \deg q^k_{ij} = \deg q^k_{i \ell} = \deg q^k_{j \ell} = 
 \vartheta
 \bigl(
  \deg q^k_i, \deg q^k_j, \deg q^k_{\ell}
 \bigr) \,,
\]
where
\[
 \vartheta (d_1, d_2, d_3) := 
 \begin{cases}
  1 & \text{if }  d_1 = d_2 = d_3 = 1, \\
  4 & \text{if }  d_1 = d_2 = d_3 = 2, \\
  2 & \text{otherwise},
 \end{cases}
\]
and $i, j, \ell$ have parity different from~$k$.
\end{remark}

\subsection{Degree tables and type tables}
\Cref{lemma:determined} implies that, if we suppose that we are not in the 
Dixon~1 case (so that all maps~$p_i$ are birational by \cref{proposition:nodixon_birational}), then the degrees of the 
maps $\{ p_{ij} \}$ determine the degrees of all the other maps. Hence, we have 
$2^9$ possible situations for the degrees of all maps determined by a 
configuration curve of~$K_{3,3}$. We can fit the numbers $\{ \deg 
p_{ij} \}$ into a table, which we call a \emph{degree table}: 
\begin{center}
\begin{tabular}{cccc}
 & 2 & 4 & 6 \\
 1 & $\ddots$ & $\vdots$ & $\fixediddots$ \\
 3 & $\cdots$ & $\deg p_{ij}$ & $\cdots$ \vphantom{$\fixediddots$}\\
 5 & $\fixediddots$ & $\vdots$ & $\ddots$
\end{tabular}
\end{center}
Notice that the set of $2^9$ possibilities admits several symmetries: we can 
permute the indices $1,3,5$ and $2,4,6$ obtaining an isomorphic configuration, 
and we can also swap the roles of the odd and even indices. Hence we obtain an 
action by the group $(S_3 \times S_3) \rtimes \Z_2$ on the set of possible degree tables.

\begin{lemma}
 The group action has $26$ orbits.
\end{lemma}
\begin{proof}
 To each degree table we associate a subgraph of~$K_{3,3}$ as follows: two 
vertices $i$ and $j$ are connected by an edge if and only if there is a~$2$ in the 
corresponding entry of the table. In this way we see that the 
group of symmetries of the degree tables gets translated into the group of 
automorphisms of~$K_{3,3}$. Therefore, the number of orbits is the number of 
subgraphs of~$K_{3,3}$ up to automorphisms, which is~$26$, as one can see, for 
example, from a direct computation using computer algebra tools. 
\end{proof}

By recalling that we are supposing that the maps $\{ p_i \}$ are birational and
using \cref{remark:same_degree_q}, we get the following result. 

\begin{lemma}
\label{lemma:lcm}
 Consider the column~$w_i$ of a degree table corresponding to a label~$i \in \{2,4,6\}$.
 Then $\vartheta(w_i)$ is the degree of the 
maps~$p_{i13}$, $p_{i15}$, and~$p_{i35}$ (see \cref{definition:maps}). 
A similar statement holds for the rows of a degree table. 
\end{lemma}

Taking into account \cref{lemma:lcm}, we can enhance the degree tables by 
the degrees of the maps~$\{ p_{ijk} \}$. For example, we have a table of the 
form:
\begin{center}
 \begin{tabular}{cccc}
  2 & 2 & \multicolumn{1}{c|}{2} & 4 \\
  2 & 2 & \multicolumn{1}{c|}{2} & 4 \\
  2 & 2 & \multicolumn{1}{c|}{1} & 2 \\ \cline{1-3}
  4 & 4 & 2 
 \end{tabular}
\end{center}

The crucial fact is that, once a degree table is given, then it prescribes the 
types of all quadrilaterals in~$K_{3,3}$, with only one ambiguity, as the 
following proposition clarifies.

\begin{proposition}
 Given an enhanced degree table, the types of the $9$ subgraphs of~$K_{3,3}$ 
isomorphic to~$K_{2,2}$ are determined as follows. For every entry of 
the table, associate a type according to the following rules: 
\begin{center}
 \begin{tabular}{ccc}
 \begin{tabular}{ccc}
 \begin{tabular}{ccc}
  $2$ & \multicolumn{1}{c|}{$\cdots$} & $4$ \\
  $\vdots$ \\ \cline{1-1}
  $4$
 \end{tabular} 
 & $\leadsto$ & \caseG
 \end{tabular}
 & $\qquad$ &
 \begin{tabular}{ccc}
 \begin{tabular}{ccc}
  $1$ & \multicolumn{1}{c|}{$\cdots$} & $2$ \\
  $\vdots$ \\ \cline{1-1}
  $2$
 \end{tabular} 
 & $\leadsto$ & \caseG
 \end{tabular}
 \\[10ex]
 \begin{tabular}{ccc}
 \begin{tabular}{ccc}
  $2$ & \multicolumn{1}{c|}{$\cdots$} & $2$ \\
  $\vdots$ \\ \cline{1-1}
  $4$
 \end{tabular} 
 & $\leadsto$ & \caseE
 \end{tabular}
 &&
 \begin{tabular}{ccc}
 \begin{tabular}{ccc}
  $2$ & \multicolumn{1}{c|}{$\cdots$} & $4$ \\
  $\vdots$ \\ \cline{1-1}
  $2$
 \end{tabular} 
 & $\leadsto$ & \caseO
 \end{tabular}
 \end{tabular}
\end{center}
\begin{center}
 \begin{tabular}{ccc}
 \begin{tabular}{ccc}
 \begin{tabular}{ccc}
  $1$ & \multicolumn{1}{c|}{$\cdots$} & $1$ \\
  $\vdots$ \\ \cline{1-1}
  $2$
 \end{tabular} 
 & $\leadsto$ & \caseE
 \end{tabular}
 &&
 \begin{tabular}{ccc}
 \begin{tabular}{ccc}
  $1$ & \multicolumn{1}{c|}{$\cdots$} & $2$ \\
  $\vdots$ \\ \cline{1-1}
  $1$
 \end{tabular} 
 & $\leadsto$ & \caseO
 \end{tabular}
 \\[10ex]
 \begin{tabular}{ccc}
 \begin{tabular}{ccc}
  $2$ & \multicolumn{1}{c|}{$\cdots$} & $2$ \\
  $\vdots$ \\ \cline{1-1}
  $2$
 \end{tabular} 
 & $\leadsto$ & \caseR/\caseL
 \end{tabular}
 &&
 \begin{tabular}{ccc}
 \begin{tabular}{ccc}
  $1$ & \multicolumn{1}{c|}{$\cdots$} & $1$ \\
  $\vdots$ \\ \cline{1-1}
  $1$
 \end{tabular} 
 & $\leadsto$ & \caseR/\caseL
 \end{tabular}
 \end{tabular}
\end{center}
 The table that we obtain listing the types of the $9$ quadrilaterals is called 
a \emph{type table}.
\end{proposition}
\begin{proof}
 The proof follows from a direct inspection of the classification of 
quadrilaterals in \cref{definition:classification_quadrilaterals}.
\end{proof}

However, only few configurations of quadrilaterals are allowed.

\begin{proposition}
\label{proposition:allowed}
 Up to permutations, the only allowed rows of a type table are the following:
 \begin{center}
  \caseG\caseG$\ast$, \caseR\caseR\caseE, \caseO\caseO\caseO, 
\caseO\caseO\caseL, \caseO\caseO\caseG,
 \end{center}
 where $\ast$ denotes any symbol in $\{$\caseG, \caseO, \caseE, \caseR, 
\caseL$\}$.
 Up to permutations, the only allowed columns of a type table are the following:
 \begin{center}
  \caseG\caseG$\ast$, \caseR\caseR\caseO, \caseE\caseE\caseE, 
\caseE\caseE\caseL, \caseE\caseE\caseG,
 \end{center}
 where $\ast$ denotes any symbol in $\{$\caseG, \caseO, \caseE, \caseR, 
\caseL$\}$.
\end{proposition}
\begin{proof}
 We only prove the statement about the columns of a type table; the statement 
about the rows follows analogously. Suppose that both $1234$ (\tikz{\draw[Apricot,line width=1] (0,0) -- (0.2,0);}) and $2345$ (\tikz{\draw[Aquamarine,line width=1] (0,0) -- (0.2,0);}) are 
\caseE. Then the situation is as follows, where the edges are labeled by the value of the function~$\delta_{S^2}$:
    \begin{center}
	  \begin{tikzpicture}[quad1/.style={Apricot},quad2/.style={Aquamarine},quad3/.style={Fuchsia!40!white},scale=0.9]
	    \begin{scope}
	      \node[vertex,label=left:1] (1) at (-2,0) {};
	      \node[vertex,label=above:2] (2) at (0,2) {};
	      \node[vertex,label=left:3] (3) at (0.2,0) {};
	      \node[vertex,label=below:4] (4) at (0,-2) {};
	      \node[vertex,label=right:5] (5) at (2.5,0) {};
	      \draw[edge] (1)to node[below=2pt,quad1] {$a$} (2)
										(1)to node[above=5pt,quad1] {$\alpha a$} (4)
										(3)to node[left,quad1] {$b$} (2)
										(3)to node[left,quad1] {$\alpha b$} (4)
										(3)to node[right,quad2] {$b$} (2)
										(3)to node[right,quad2] {$\beta b$} (4)
										(5)to node[below=2pt,quad2] {$c$} (2)										
										(5)to node[above=5pt,quad2] {$\beta c$} (4);
				\draw[quad1,line width=0.75] ($(1)+(0.1,0)$) -- ($(2)-(0.0707107,0.16)$) -- ($(3)-(0.1,0)$) -- ($(4)+(-0.0707107,0.16)$) -- ($(1)+(0.1,0)$) -- cycle;
				\draw[quad2,line width=0.75] ($(5)-(0.1,0)$) -- ($(2)-(-0.0707107,0.16)$) -- ($(3)+(0.1,0)$) -- ($(4)+(0.0707107,0.16)$) -- ($(5)-(0.1,0)$) -- cycle;
				\draw[quad3,line width=0.75] ($(1)-(0.16,0)$) -- ($(2)+(0,0.16)$) -- ($(5)+(0.16,0)$) -- ($(4)-(0,0.16)$) -- ($(1)-(0.16,0)$) -- cycle;
	    \end{scope}
	  \end{tikzpicture}%
	\end{center}
with $\alpha, \beta \in \{1,-1\}$. 
We now try to determine the type of $1245$ (\tikz{\draw[Fuchsia!40!white,line width=1] (0,0) -- (0.2,0);}). 
We distinguish two cases. 
Suppose that $\alpha = \beta$. 
If $c = a$, then $1245$ is \caseL; 
if $c \neq a$, then $1245$ is \caseE. 
Suppose now that $\alpha \neq \beta$. 
This forces $b = 0$, and so we must have $a \neq 0$ and $c \neq 0$, therefore $1245$ is \caseG. 
These three cases are depicted here:
\begin{center}
	  \begin{tikzpicture}[quad1/.style={Apricot},quad2/.style={Aquamarine},quad3/.style={Fuchsia!40!white},scale=0.55]
	    \begin{scope}
	      \node[vertex,label=left:1] (1) at (-2,0) {};
	      \node[vertex,label=above:2] (2) at (0,2) {};
	      \node[vertex,label=left:3] (3) at (0.2,0) {};
	      \node[vertex,label=below:4] (4) at (0,-2) {};
	      \node[vertex,label=right:5] (5) at (2.5,0) {};
	      \draw[edge] (1)to node[above left,quad3] {$a$} (2)
										(1)to node[below left,quad3] {$\pm a$} (4)
										(3)to node[left,quad1] {$b$} (2)
										(3)to node[pos=0.3,left,quad1] {$\pm b$} (4)
										(3)to node[right,quad2] {$b$} (2)
										(3)to node[pos=0.3,right,quad2] {$\pm b$} (4)
										(5)to node[above right,quad3] {$a$} (2)										
										(5)to node[below right,quad3] {$\pm a$} (4);
				\draw[quad1,line width=0.75] ($(1)+(0.1,0)$) -- ($(2)-(0.0707107,0.16)$) -- ($(3)-(0.1,0)$) -- ($(4)+(-0.0707107,0.16)$) -- ($(1)+(0.1,0)$) -- cycle;
				\draw[quad2,line width=0.75] ($(5)-(0.1,0)$) -- ($(2)-(-0.0707107,0.16)$) -- ($(3)+(0.1,0)$) -- ($(4)+(0.0707107,0.16)$) -- ($(5)-(0.1,0)$) -- cycle;
				\draw[quad3,line width=0.75] ($(1)-(0.16,0)$) -- ($(2)+(0,0.16)$) -- ($(5)+(0.16,0)$) -- ($(4)-(0,0.16)$) -- ($(1)-(0.16,0)$) -- cycle;
	    \end{scope}
	    \begin{scope}[xshift=7cm]
	      \node[vertex,label=left:1] (1) at (-2,0) {};
	      \node[vertex,label=above:2] (2) at (0,2) {};
	      \node[vertex,label=left:3] (3) at (0.2,0) {};
	      \node[vertex,label=below:4] (4) at (0,-2) {};
	      \node[vertex,label=right:5] (5) at (2.5,0) {};
	      \draw[edge] (1)to node[above left,quad3] {$a$} (2)
										(1)to node[below left,quad3] {$\pm a$} (4)
										(3)to node[left,quad1] {$b$} (2)
										(3)to node[pos=0.3,left,quad1] {$\pm b$} (4)
										(3)to node[right,quad2] {$b$} (2)
										(3)to node[pos=0.3,right,quad2] {$\pm b$} (4)
										(5)to node[above right,quad3] {$c$} (2)										
										(5)to node[below right,quad3] {$\pm c$} (4);
				\draw[quad1,line width=0.75] ($(1)+(0.1,0)$) -- ($(2)-(0.0707107,0.16)$) -- ($(3)-(0.1,0)$) -- ($(4)+(-0.0707107,0.16)$) -- ($(1)+(0.1,0)$) -- cycle;
				\draw[quad2,line width=0.75] ($(5)-(0.1,0)$) -- ($(2)-(-0.0707107,0.16)$) -- ($(3)+(0.1,0)$) -- ($(4)+(0.0707107,0.16)$) -- ($(5)-(0.1,0)$) -- cycle;
				\draw[quad3,line width=0.75] ($(1)-(0.16,0)$) -- ($(2)+(0,0.16)$) -- ($(5)+(0.16,0)$) -- ($(4)-(0,0.16)$) -- ($(1)-(0.16,0)$) -- cycle;
	    \end{scope}
	    \begin{scope}[yshift=-5cm,xshift=3.5cm]
	      \node[vertex,label=left:1] (1) at (-2,0) {};
	      \node[vertex,label=above:2] (2) at (0,2) {};
	      \node[vertex,label=left:3] (3) at (0.2,0) {};
	      \node[vertex,label=below:4] (4) at (0,-2) {};
	      \node[vertex,label=right:5] (5) at (2.5,0) {};
	      \draw[edge] (1)to node[above left,quad3] {$a$} (2)
										(1)to node[below left,quad3] {$\pm a$} (4)
										(3)to node[left,quad1] {$0$} (2)
										(3)to node[left,quad1] {$0$} (4)
										(3)to node[right,quad2] {$0$} (2)
										(3)to node[right,quad2] {$0$} (4)
										(5)to node[above right,quad3] {$c$} (2)										
										(5)to node[below right,quad3] {$\mp c$} (4);
				\draw[quad1,line width=0.75] ($(1)+(0.1,0)$) -- ($(2)-(0.0707107,0.16)$) -- ($(3)-(0.1,0)$) -- ($(4)+(-0.0707107,0.16)$) -- ($(1)+(0.1,0)$) -- cycle;
				\draw[quad2,line width=0.75] ($(5)-(0.1,0)$) -- ($(2)-(-0.0707107,0.16)$) -- ($(3)+(0.1,0)$) -- ($(4)+(0.0707107,0.16)$) -- ($(5)-(0.1,0)$) -- cycle;
				\draw[quad3,line width=0.75] ($(1)-(0.16,0)$) -- ($(2)+(0,0.16)$) -- ($(5)+(0.16,0)$) -- ($(4)-(0,0.16)$) -- ($(1)-(0.16,0)$) -- cycle;
	    \end{scope}
	  \end{tikzpicture}
\end{center}

We now show that if $1234$ is \caseE\ and $2345$ is \caseL, then $1245$ is \caseE. 
The situation is:
\begin{center}
	  \begin{tikzpicture}[quad1/.style={Apricot},quad2/.style={Aquamarine},quad3/.style={Fuchsia!40!white},scale=0.9]
	    \begin{scope}
	      \node[vertex,label=left:1] (1) at (-2,0) {};
	      \node[vertex,label=above:2] (2) at (0,2) {};
	      \node[vertex,label=left:3] (3) at (0.2,0) {};
	      \node[vertex,label=below:4] (4) at (0,-2) {};
	      \node[vertex,label=right:5] (5) at (2.5,0) {};
	      \draw[edge] (1)to node[below=2pt,quad1] {$a$} (2)
										(1)to node[above=5pt,quad1] {$\alpha a$} (4)
										(3)to node[left,quad1] {$b$} (2)
										(3)to node[left,quad1] {$\alpha b$} (4)
										(3)to node[right,quad2] {$b$} (2)
										(3)to node[right,quad2] {$\beta b$} (4)
										(5)to node[below=2pt,quad2] {$\gamma b$} (2)										
										(5)to node[above=5pt,quad2] {$\delta b$} (4);
				\draw[quad1,line width=0.75] ($(1)+(0.1,0)$) -- ($(2)-(0.0707107,0.16)$) -- ($(3)-(0.1,0)$) -- ($(4)+(-0.0707107,0.16)$) -- ($(1)+(0.1,0)$) -- cycle;
				\draw[quad2,line width=0.75] ($(5)-(0.1,0)$) -- ($(2)-(-0.0707107,0.16)$) -- ($(3)+(0.1,0)$) -- ($(4)+(0.0707107,0.16)$) -- ($(5)-(0.1,0)$) -- cycle;
				\draw[quad3,line width=0.75] ($(1)-(0.16,0)$) -- ($(2)+(0,0.16)$) -- ($(5)+(0.16,0)$) -- ($(4)-(0,0.16)$) -- ($(1)-(0.16,0)$) -- cycle;
	    \end{scope}
	  \end{tikzpicture}
	\end{center}
with $\alpha, \beta, \gamma, \delta \in \{1,-1\}$.
Since $b$ cannot be $0$, we have $\alpha = \beta$. If $\gamma = \delta$, then the lozenge condition forces $\alpha = 1$, so $1245$ is \caseE. If $\gamma = -\delta$, then the lozenge condition forces $\alpha = -1$, so again $1245$ is \caseE. 

We now exclude that 1234 is \caseE\ and $2345$ is \caseO. In this case:
\begin{center}
	  \begin{tikzpicture}[quad1/.style={Apricot},quad2/.style={Aquamarine},quad3/.style={Fuchsia!40!white},scale=0.9]
	    \begin{scope}
	      \node[vertex,label=left:1] (1) at (-2,0) {};
	      \node[vertex,label=above:2] (2) at (0,2) {};
	      \node[vertex,label=left:3] (3) at (0.2,0) {};
	      \node[vertex,label=below:4] (4) at (0,-2) {};
	      \node[vertex,label=right:5] (5) at (2.5,0) {};
	      \draw[edge] (1)to node[below=2pt,quad1] {$a$} (2)
										(1)to node[above=5pt,quad1] {$\alpha a$} (4)
										(3)to node[left,quad1] {$b$} (2)
										(3)to node[left,quad1] {$\alpha b$} (4)
										(3)to node[right,quad2] {$b$} (2)
										(3)to node[right,quad2] {$\alpha b$} (4)
										(5)to node[below=2pt,quad2] {$\beta b$} (2)										
										(5)to node[above=5pt,quad2] {$\beta \alpha b$} (4);
				\draw[quad1,line width=0.75] ($(1)+(0.1,0)$) -- ($(2)-(0.0707107,0.16)$) -- ($(3)-(0.1,0)$) -- ($(4)+(-0.0707107,0.16)$) -- ($(1)+(0.1,0)$) -- cycle;
				\draw[quad2,line width=0.75] ($(5)-(0.1,0)$) -- ($(2)-(-0.0707107,0.16)$) -- ($(3)+(0.1,0)$) -- ($(4)+(0.0707107,0.16)$) -- ($(5)-(0.1,0)$) -- cycle;
				\draw[quad3,line width=0.75] ($(1)-(0.16,0)$) -- ($(2)+(0,0.16)$) -- ($(5)+(0.16,0)$) -- ($(4)-(0,0.16)$) -- ($(1)-(0.16,0)$) -- cycle;
	    \end{scope}
	  \end{tikzpicture}
	\end{center}
Here we must have $b \neq 0$. This is not possible since $2345$ would be actually a lozenge for any $\alpha,\beta\in\{1,-1\}$. 
In a similar way we exclude the case where $1234$ is \caseE\ and $2345$ is \caseR, and $1234$ is \caseL\ and $2345$ is \caseR.

We now exclude that $1234$ is \caseO\ and $2345$ is \caseO. In this case:
\begin{center}
	  \begin{tikzpicture}[quad1/.style={Apricot},quad2/.style={Aquamarine},quad3/.style={Fuchsia!40!white},scale=0.9]
	    \begin{scope}
	      \node[vertex,label=left:1] (1) at (-2,0) {};
	      \node[vertex,label=above:2] (2) at (0,2) {};
	      \node[vertex,label=left:3] (3) at (0.2,0) {};
	      \node[vertex,label=below:4] (4) at (0,-2) {};
	      \node[vertex,label=right:5] (5) at (2.5,0) {};
	      \draw[edge] (1)to node[below=2pt,quad1] {$a$} (2)
										(1)to node[above=5pt,quad1] {$b$} (4)
										(3)to node[left,quad1] {$\alpha a$} (2)
										(3)to node[left,quad1] {$\alpha b$} (4)
										(3)to node[right,quad2] {$\alpha a$} (2)
										(3)to node[right,quad2] {$\alpha b$} (4)
										(5)to node[below=2pt,quad2] {$\beta \alpha a$} (2)										
										(5)to node[above=5pt,quad2] {$\beta \alpha b$} (4);
				\draw[quad1,line width=0.75] ($(1)+(0.1,0)$) -- ($(2)-(0.0707107,0.16)$) -- ($(3)-(0.1,0)$) -- ($(4)+(-0.0707107,0.16)$) -- ($(1)+(0.1,0)$) -- cycle;
				\draw[quad2,line width=0.75] ($(5)-(0.1,0)$) -- ($(2)-(-0.0707107,0.16)$) -- ($(3)+(0.1,0)$) -- ($(4)+(0.0707107,0.16)$) -- ($(5)-(0.1,0)$) -- cycle;
				\draw[quad3,line width=0.75] ($(1)-(0.16,0)$) -- ($(2)+(0,0.16)$) -- ($(5)+(0.16,0)$) -- ($(4)-(0,0.16)$) -- ($(1)-(0.16,0)$) -- cycle;
	    \end{scope}
	  \end{tikzpicture}
	\end{center}
By swapping $3$ and $5$ with their antipodes, we can suppose $\alpha = \beta = 1$.
However, in this case we get three points that are at the same distance from two other ones,
and this forces vertices $2$ and $4$ to be antipodal, which is not allowed.
Similarly, we can exclude the cases where $1234$ is \caseO\ and $2345$ is \caseL, $1234$ and $2345$ are both \caseL.

We now show that if both $1234$ and $2345$ are \caseR, then $1245$ must be \caseO. We have:
\begin{center}
	  \begin{tikzpicture}[quad1/.style={Apricot},quad2/.style={Aquamarine},quad3/.style={Fuchsia!40!white},scale=0.9]
	    \begin{scope}
	      \node[vertex,label=left:1] (1) at (-2,0) {};
	      \node[vertex,label=above:2] (2) at (0,2) {};
	      \node[vertex,label=left:3] (3) at (0.2,0) {};
	      \node[vertex,label=below:4] (4) at (0,-2) {};
	      \node[vertex,label=right:5] (5) at (2.5,0) {};
	      \draw[edge] (1)to node[below=2pt,quad1] {$a$} (2)
										(1)to node[above=5pt,quad1] {$b$} (4)
										(3)to node[left,quad1] {$\alpha b$} (2)
										(3)to node[left,quad1] {$\alpha a$} (4)
										(3)to node[right,quad2] {$\alpha b$} (2)
										(3)to node[right,quad2] {$\alpha a$} (4)
										(5)to node[below=2pt,quad2] {$\beta \alpha a$} (2)										
										(5)to node[above=5pt,quad2] {$\beta \alpha b$} (4);
				\draw[quad1,line width=0.75] ($(1)+(0.1,0)$) -- ($(2)-(0.0707107,0.16)$) -- ($(3)-(0.1,0)$) -- ($(4)+(-0.0707107,0.16)$) -- ($(1)+(0.1,0)$) -- cycle;
				\draw[quad2,line width=0.75] ($(5)-(0.1,0)$) -- ($(2)-(-0.0707107,0.16)$) -- ($(3)+(0.1,0)$) -- ($(4)+(0.0707107,0.16)$) -- ($(5)-(0.1,0)$) -- cycle;
				\draw[quad3,line width=0.75] ($(1)-(0.16,0)$) -- ($(2)+(0,0.16)$) -- ($(5)+(0.16,0)$) -- ($(4)-(0,0.16)$) -- ($(1)-(0.16,0)$) -- cycle;
	    \end{scope}
	  \end{tikzpicture}
	\end{center}
By swapping $3$ and $5$ with their antipodes we can suppose that $\alpha = \beta = 1$, so $1245$ is~\caseO. With the same technique we can prove that if $1234$ is \caseO\ and $2345$ is \caseR, then $1245$ must be \caseR.

To conclude the proof, one notices that if $1245$ is of any of the five types \caseG, \caseO, \caseE, \caseR, or \caseL, then by picking any two general edge lengths for the edges~$23$ and~$34$, one can construct an instance where both $1234$ and $2345$ are \caseG.
\end{proof}

By inspecting the $26$ different cases for the degree table in the light of \cref{proposition:allowed}, one finds that the 
only allowed type tables are the following four: 
\begin{center}
 \begin{tabular}{m{0.8cm}ccc}
  \newCase\label{case:general} &
  \begin{tabular}{cccc}
   2 & 2 & \multicolumn{1}{c|}{2} & 4 \\
   2 & 2 & \multicolumn{1}{c|}{2} & 4 \\
   2 & 2 & \multicolumn{1}{c|}{2} & 4 \\ \cline{1-3}
   4 & 4 & 4 
  \end{tabular}
  & $\leadsto$ &
  \begin{tabular}{ccc}
   \caseG & \caseG & \caseG \\
   \caseG & \caseG & \caseG \\
   \caseG & \caseG & \caseG 
  \end{tabular}
  \\[6ex]
  \newCase\label{case:dixon1} &
  \begin{tabular}{cccc}
   1 & 1 & \multicolumn{1}{c|}{1} & 1 \\
   1 & 1 & \multicolumn{1}{c|}{1} & 1 \\
   1 & 1 & \multicolumn{1}{c|}{2} & 2 \\ \cline{1-3}
   1 & 1 & 2 
  \end{tabular}
  & $\leadsto$ &
  \begin{tabular}{ccc}
   \caseR & \caseR & \caseE \\
   \caseR & \caseR & \caseE \\
   \caseO & \caseO & \caseL 
  \end{tabular}
  \\[6ex]
  \newCase\label{case:dixon2} &
  \begin{tabular}{cccc}
   2 & 1 & \multicolumn{1}{c|}{1} & 2 \\
   1 & 2 & \multicolumn{1}{c|}{1} & 2 \\
   1 & 1 & \multicolumn{1}{c|}{2} & 2 \\ \cline{1-3}
   2 & 2 & 2 
  \end{tabular}
  & $\leadsto$ &
  \begin{tabular}{ccc}
   \caseR & \caseG & \caseG \\
   \caseG & \caseR & \caseG \\
   \caseG & \caseG & \caseR 
  \end{tabular}
  \\[6ex]
  \newCase\label{case:dixon3} &
  \begin{tabular}{cccc}
   1 & 1 & \multicolumn{1}{c|}{2} & 2 \\
   1 & 1 & \multicolumn{1}{c|}{2} & 2 \\
   2 & 2 & \multicolumn{1}{c|}{2} & 4 \\ \cline{1-3}
   2 & 2 & 4 
  \end{tabular}
  & $\leadsto$ &
  \begin{tabular}{ccc}
   \caseG & \caseG & \caseE \\
   \caseG & \caseG & \caseE \\
   \caseO & \caseO & \caseG 
  \end{tabular}
 \end{tabular}
\end{center}

We analyze the four possible cases one by one.

We introduce some notation to deal with nontrivial cuts of~$K_{3,3}$, namely 
cuts that induce surjective colorings. In order 
to make the notation lighter, we will always suppose that the cuts are in 
``normal form'', as described by the following lemma.
\begin{lemma}
\label{lemma:cut_normal_form}
 Let $(I, J)$ be a nontrivial cut of~$K_{3,3}$, namely the 
coloring~$\varepsilon_{I,J}$ is surjective. Then, up to swapping~$I$ 
and~$J$, and the~$P's$ and~$Q's$, we can always suppose that $I$ is of the 
following form: 
 \[
  \{ P_i, Q_i, 
  \begin{array}{c}
  P_{j_1} \\[-3pt]
  \text{\tiny or} \\[-2pt]
  Q_{j_1}
  \end{array},
  \begin{array}{c}
  P_{j_2} \\[-3pt]
  \text{\tiny or} \\[-2pt]
  Q_{j_2}
  \end{array},
  \begin{array}{c}
  P_{j_3} \\[-3pt]
  \text{\tiny or} \\[-2pt]
  Q_{j_3}
  \end{array}
  \} \,,
 \]
 where $i \in \{1, \dotsc, 6\}$ and $j_1 < j_2 < j_3$ are the three numbers in 
$\{1, \dotsc, 6\}$ of parity different from the one of~$i$. 
\end{lemma}
\begin{proof}
 By assumption, the coloring~$\varepsilon_{I,J}$ is NAP. By symmetry, we can 
suppose 
 that the three edges incident with vertex $1$ are red and all other edges are 
blue. This implies that $P_1, Q_1 \in I$, since otherwise we would have, in 
particular, $P_2, Q_2 \in I$, and this is not compatible with the fact that 
$\{2,3\}$ is blue. The same argument shows that for $j \in \{2,4,6\}$, 
either $P_j \in I$ or $Q_j \in I$. This concludes the proof.
\end{proof}

\begin{notation}
 According to \cref{lemma:cut_normal_form}, we denote by $(i, T_1T_2T_3)$,
 where $T_k \in \{ P, Q \}$, the cut for which $I = \{P_i, Q_i, (T_1)_{j_1}, 
(T_2)_{j_2}, (T_3)_{j_3} \}$. For example, the cut $(3,PQP)$ is the one in 
which $I = \{ P_3, Q_3, P_2, Q_4, P_6 \}$, while $(2,PPQ)$ is the one in which 
$I = \{ P_2, Q_2, P_1, P_3, Q_5\}$. 
\end{notation}

\begin{definition}
\label{definition:mu_K33}
 As we did for motions of~$K_{2,2}$, given a proper motion~$\curveC$ of~$K_{3,3}$ and a cut~$(i, T_1T_2T_3)$, 
 we define the number $\mu(i, T_1 T_2 T_3)$ as the degree of the restriction of~$\curveC$ 
 to the divisor corresponding to~$(i, T_1T_2T_3)$. 
 Notice that, if $(i, V_1 V_2 V_3)$ is the cut obtained from $(i, T_1 T_2 T_3)$ by swapping $P$'s and $Q$'s, 
 then $D_{i, T_1 T_2 T_3}$ is complex conjugated to $D_{i, V_1 V_2 V_3}$; 
 since $\curveC$ is real, we have that $\mu(i, T_1 T_2 T_3) = \mu(i, V_1 V_2 V_3)$.
\end{definition}

\begin{proposition}
\label{proposition:general}
 \Cref{case:general} cannot happen.
\end{proposition}
\begin{proof}
 We compute the pullbacks under the maps~$p_{ij}$ of the divisors~$\divom$ on each of the 
subgraphs isomorphic to~$K_{2,2}$. For example, the pullback 
of~$\divom$ under~$p_{56}$ consists of the four divisors $D_{1, PQP}$, $D_{1,QPP}$, $D_{3,PQP}$ and~$D_{3,QPP}$. 
By restricting this pullback to the proper motion~$\curveC$ and taking the degree, we get the equation 
 \[
  \mu(1, PQP) + \mu(1, QPP) + \mu(3, PQP) + \mu(3, QPP) = 
  \mu^{56}(\cutom) \cdot \deg p_{56} = 2 \,.
 \]
 Summing all these equations gives
 \[
  \sum_{i \in \{1, 3, 5\}} 4 \mu(i,PPQ) + 4 \mu(i,PQP) + 4 \mu(i,QPP) = 18 \,,
 \]
 which does not have integer solutions, so this case cannot happen.
\end{proof}

\begin{lemma}
\label{lemma:diagonals}
A rhomboid cannot have orthogonal diagonals.
\end{lemma}
\begin{proof}
 For a contradiction, let us suppose that $1234$ is a rhomboid with 
orthogonal diagonals. Then we can take coordinates for the vertices 
of the rhomboid as follows:
\[
 \underbrace{(x_1, 0, z_1)}_{\text{vertex 1}} \quad
 \underbrace{(x_2, y_2, 0)}_{\text{vertex 2}} \quad
 \underbrace{(x_3, 0, z_3)}_{\text{vertex 3}} \quad
 \underbrace{(x_4, y_4, 0)}_{\text{vertex 4}}
\]
By assumption, we have
\begin{align*}
 x_2 \, x_3 &= a, & x_1 \, x_2 &= b, \\
 x_1 \, x_4 &= \alpha \, a, & x_3 \, x_4 &= \alpha \, b,
\end{align*}
with $\alpha \in \{1, -1\}$. From this it follows that $a^2 = b^2$, hence the 
contradiction.
\end{proof}

\begin{proposition}
\label{proposition:dixon1}
 \Cref{case:dixon1} cannot happen.
\end{proposition}
\begin{proof}
 The type table implies that $1234$ is a lozenge, $1245$ is an 
even deltoid, and $1346$ is an odd deltoid. This forces the vertices $1$, 
$3$, and~$5$ to be cocircular, and similarly for $2$, $4$, and~$6$; moreover 
the two great circles are orthogonal. However, by assumption $2356$ is a 
rhomboid, which would have orthogonal diagonals. This contradicts 
\cref{lemma:diagonals}.
\end{proof}

\begin{proposition}
\label{proposition:dixon2}
 \Cref{case:dixon2} is an example of a Dixon 2 motion.
\end{proposition}

\begin{proof}
 In this case we have three rhomboids, namely $1234$, $1256$ and $3456$, 
 and the motions of the other $K_{2,2}$-subgraphs are general ones. 
 At first sight, we would need to analyze $4^3 = 64$ cases, 
 since each of the rhomboids may have type from~$1$ to~$4$. 
 We employ some automorphisms of~$\M_{0,12}$ to decrease the total number of cases to analyze. 
 For a given $i \in \{1, \dotsc, 6\}$, 
 consider the automorphism~$\sigma_i$ of~$\M_{0,12}$ that swaps~$P_i$ and~$Q_i$ and fixes the other marked points. 
 This automorphism does not, in general, preserve the proper motion~$\curveC$ of~$K_{3,3}$ we fixed once and for all. 
 However, it preserves the fact that $1234$, $1256$ and $3456$ are rhomboids, as well as the types of the other quadrilaterals. 
 Hence we can employ it to simplify our setting. 
 For example, consider the automorphism~$\sigma_1$ swapping $P_1 \leftrightarrow Q_1$ and its action on $\curveC_{56} \subseteq \M_{0,8}$. 
 A direct computation shows that the divisors~$\divou$ and~$\divom$ are preserved, 
 while~$\divem$ is mapped to~$\overline{\diveu}$. 
 Thus the type of the rhomboid $1234$ is changed by~$\sigma_1$ as follows:
\[
  \text{Type 1} \leftrightarrow \text{Type 3},
  \qquad
  \text{Type 2} \leftrightarrow \text{Type 4}.
\]
Instead, if we use the automorphism~$\sigma_2$ swapping $P_2 \leftrightarrow Q_2$, 
we get that the type of $1234$ changes as follows:
\[
  \text{Type 1} \leftrightarrow \text{Type 2},
  \qquad
  \text{Type 3} \leftrightarrow \text{Type 4}.
\]
Therefore, here we have an action of the group $\mathbb{G} = \left\langle \sigma_1, \dotsc, \sigma_6 \right\rangle \cong (\Z_2)^6$ on the set of triples of types of the rhomboids $1234$, $1256$, and $3456$. 
We then use this action to simplify the type of the three rhomboids as much as possible. 
We see that the action is not transitive, since every triple admits as stabilizer the subgroup $\left\langle \mathrm{id}, \sigma_1 \circ \sigma_3 \circ \sigma_5, \sigma_2 \circ \sigma_4 \circ \sigma_6 \right\rangle$, which has cardinality~$4$. 
There are then four triples of types of rhomboids that are not equivalent under~$\mathbb{G}$, and these are:
\[
\begin{array}{ll}
 (\text{Type 1}, \text{Type 1}, \text{Type 1}), & (\text{Type 2}, \text{Type 2}, \text{Type 2}), \\
 (\text{Type 3}, \text{Type 3}, \text{Type 3}), & (\text{Type 4}, \text{Type 4}, \text{Type 4}).
\end{array}
\]
We want to show that only the situation where all the three rhomboids are of Type~1 can happen. 
To do so, we first consider the equations for the numbers $\mu(i, T_1 T_2 T_3)$ 
coming from taking the pullbacks under the maps~$p_{ij}$ of the divisors~$\divou$ and~$\divom$ on the subgraphs of the three rhomboids. 
For example, the pullback $p_{56}^{\ast}(\divou)$ equals the sum of the four divisors $D_{1, PPQ}$, $D_{1,PPP}$, $D_{3, QQP}$, and $D_{3, QQQ}$, 
and we get the equation:
\[
  \mu(1, PPQ) + \mu(1, PPP) + \mu(3, QQP) + \mu(3, QQQ) = 
  \mu^{56}(\cutou) \cdot \deg p_{56} \,.
\]
When all the three rhomboids are of Type 1 we have $\mu(\cutou) = 0$ and $\mu(\cutom) = 1$, 
and so we get the system of equations:
\begin{gather*}
 \mu(1, PPQ) = \mu(1, PPP) = \mu(3, QQQ) = \mu(3, QQP) = 0 \,, \\
 \mu(1, PQP) = \mu(1, PPP) = \mu(5, QQQ) = \mu(5, QPQ) = 0 \,, \\
 \mu(3, QPP) = \mu(3, PPP) = \mu(5, QQQ) = \mu(5, PQQ) = 0 \,,
\end{gather*}
and
\begin{gather*}
 \mu(1, PQQ) + \mu(1, PQP) + \mu(3, QPQ) + \mu(3, QPP) = 2 \,, \\
 \mu(1, PQQ) + \mu(1, PPQ) + \mu(5, QQP) + \mu(5, QPP) = 2 \,, \\
 \mu(3, QPQ) + \mu(3, PPQ) + \mu(5, QQP) + \mu(5, PQP) = 2 \,.
\end{gather*}
Recalling from \cref{definition:mu_K33} that we have equalities $\mu(1, PPQ) = \mu(1, QQP)$ and similar ones, 
we see that these equations admit the unique solution
\[
 \mu(1, PQQ) = \mu(3, QPQ) = \mu(5, QQP) = 1 
\]
(and all other non-conjugated quantities are zero).
One can then check that this solution satisfies also all similar equations coming from considering other subgraphs in~$K_{3,3}$ that are isomorphic to~$K_{2,2}$ and that are not rhomboids. Instead, if all the three rhomboids are of Type 2, the situation is different. In that case we have $\mu(\cutou) = 1$ and $\mu(\cutom) = 0$, and the system of equations coming from the three rhomboids has solution
\[
 \mu(1,PPP)=\mu(3,PPP)=\mu(5,PPP)=1 \,.
\]
This solution, however, does not satisfy all the equations coming from the other subgraphs. For example, if we consider the subgraph $1236$, we have that the pullback $p_{45}^{\ast}(\divou)$ equals the sum of the four divisors $D_{1, PPQ}$, $D_{1,PQQ}$, $D_{3, QPP}$, and $D_{3, QQP}$. Hence the subgraph determines the equation
\[
  \mu(1, PPQ) + \mu(1, PQQ) + \mu(3, QPP) + \mu(3, QQP) = 
  \mu^{45}(\cutom) \cdot \deg p_{45} = 1 \,,
\]
which is not satisfied by the previous solution. Hence the situation where all the three rhomboids are of Type 2 cannot occur. The same argument shows that neither Type 4 can occur. To exclude Type 3, instead of taking the divisors~$\divou$ and~$\divom$, we consider~$\diveu$ and~$\divem$.
Therefore we can suppose that all the three rhomboids are of Type 1. Then we know that there exist three rotations $o_{1234}$, $o_{1256}$, and~$o_{3456}$ of the sphere~$S^2$, each of which is a symmetry of the corresponding rhomboid. Moreover, as we clarified in \Cref{K33:quadrilaterals}, these three rotations are related to automorphisms $\tau_{56}$, $\tau_{34}$, and $\tau_{12}$ of~$\M_{0,12}$ that preserve the proper motions of the rhomboids. Since $\tau_{56} \circ \tau_{34} = \tau_{12}$, the same relation holds for the rotations, and so these three rotations commute and are involutions. Thus we are in the situation of a Dixon~2 motion.
\end{proof}

To describe the motion that we find in \Cref{case:dixon3}, 
let us recall a duality in spherical geometry between pairs of antipodal points and great circles. 
Given a pair of antipodal points, we can consider the plane orthogonal to the line connecting these two points: 
we say that the great circle determined by this plane on~$S^2$ is the dual to the original pair of points. 
Conversely, the axis of a great circle determines a pair of antipodal points.

\begin{proposition}
\label{proposition:new_dixon}
 In \Cref{case:dixon3}, there exists a proper motion of~$K_{3,3}$ described as follows. 
 We consider a quadrilateral~$1234$ with three sides of equal spherical length~$(1-a)/2$, 
 and the fourth of length~$(1+a)/2$. 
 This quadrilateral has the property that, while it moves, the angle between its diagonals stays the same. 
 We pick as realizations of~$5$ and $6$ the duals of the diagonals of~$1234$. 
 By duality, the spherical distance~$(1-e)/2$ between the realizations of~$5$ and~$6$ 
 is the cosine of the angle between the diagonals of~$1234$. 
 We call this motion a \emph{constant diagonal angle motion}.
\end{proposition}
\begin{proof}
 The last row and the last column of the type table imply that 
 certain values of the lengths~$\delta_{S^2}$ must be zero. 
 The two even deltoids in the last column, namely $2345$ and $1245$, 
 imply that we have the following situation, where the edges are labeled by the value of the function~$\delta_{S^2}$, 
 and $\alpha, \beta \in \{1, -1\}$:
 \begin{center}
	  \begin{tikzpicture}[quad1/.style={Apricot},quad2/.style={Aquamarine},quad3/.style={Fuchsia!40!white},scale=0.75]
	    \begin{scope}
	      \node[vertex,label=left:1] (1) at (-2,0) {};
	      \node[vertex,label=above:2] (2) at (0,2) {};
	      \node[vertex,label=left:3] (3) at (0.2,0) {};
	      \node[vertex,label=below:4] (4) at (0,-2) {};
	      \node[vertex,label=right:5] (5) at (2.5,0) {};
	      \draw[edge] (1)to node[above left,quad3] {$a$} (2)
										(1)to node[below left,quad3] {$\alpha a$} (4)
										(3)to node[left,quad1] {} (2)
										(3)to node[pos=0.3,left,quad1] {} (4)
										(3)to node[right,quad2] {$b$} (2)
										(3)to node[right,quad2] {$\beta b$} (4)
										(2)to node[below left,quad2] {$c$} (5)
										(4)to node[above=5pt,quad2] {$\beta c$} (5)
										(5)to node[above right,quad3] {$c$} (2)
										(5)to node[below right,quad3] {$\alpha c$} (4);
				\draw[quad1,line width=0.75] ($(1)+(0.1,0)$) -- ($(2)-(0.0707107,0.16)$) -- ($(3)-(0.1,0)$) -- ($(4)+(-0.0707107,0.16)$) -- ($(1)+(0.1,0)$) -- cycle;
				\draw[quad2,line width=0.75] ($(5)-(0.1,0)$) -- ($(2)-(-0.0707107,0.16)$) -- ($(3)+(0.1,0)$) -- ($(4)+(0.0707107,0.16)$) -- ($(5)-(0.1,0)$) -- cycle;
				\draw[quad3,line width=0.75] ($(1)-(0.16,0)$) -- ($(2)+(0,0.16)$) -- ($(5)+(0.16,0)$) -- ($(4)-(0,0.16)$) -- ($(1)-(0.16,0)$) -- cycle;
	    \end{scope}
	\end{tikzpicture}
\end{center}
By swapping the realization of vertex~$4$ with its antipode, we can suppose without loss of generality that $\delta_{S^2}(1,4) = a$, so $\alpha = 1$ in the diagram.
This forces $\delta_{S^2}(3,4) = -b$, so $\beta = -1$ in the diagram, because otherwise $1234$ would be \caseE, while from the table we know that it is \caseG.
Hence $c = 0$. 

A similar argument, using the last row of the type table, implies that $\delta_{S^2}(1,6) = \delta_{S^2}(3,6) = 0$. Moreover, since $1346$ is \caseO, we have $\delta_{S^2}(1,4) = \pm \delta_{S^2}(3,4)$, so $a = \pm b$. By swapping the realization of vertex~$3$ we can assume that $a = b$. The situation for the whole $K_{3,3}$ is then as follows, where $\delta_{S^2}(5,6) = e$:
 \begin{center}
	  \begin{tikzpicture}
	    \begin{scope}[scale=0.9]
	      \node[vertex,label=left:1] (1) at (-2,0) {};
	      \node[vertex,label=above:2] (2) at (0,2) {};
	      \node[vertex,label=left:3] (3) at (0.2,0.2) {};
	      \node[vertex,label=below:4] (4) at (0,-2) {};
	      \node[vertex,label=right:5] (5) at (2.5,0) {};
	      \node[vertex,label=right:6] (6) at (2,-2) {};
	      \draw[edge] (1)to node[above left] {$a$} (2)
										(1)to node[below left] {$a$} (4)
										(3)to node[left] {$a$} (2)
										(3)to node[pos=0.2,left] {$-a$} (4)
										(5)to node[above right] {$0$} (2)
										(5)to node[pos=0.3, right, below] {$0$} (4)
										(5)to node[below right] {$e$} (6)
										(3)to node[pos=0.3, right] {$0$} (6)
										(1)to node[pos=0.8, left, below] {$0$} (6);
	    \end{scope}
	\end{tikzpicture}
\end{center}
Due to the general type of several quadrilaterals, we know that $e \neq a$ and $e \neq -a$.

The fact that this kind of realizations of~$K_{3,3}$ is mobile and that, during the motion, the angle between the diagonals of~$1234$ stays constant (namely, the distance~$e$ stays constant), will be proven via symbolic computation techniques. In particular, the fact that $e$ is constant follows from the fact that $a$ and $e$ satisfy an algebraic relation. We challenge the reader to find a ``geometric'' proof (we could not find one). 

We set up a symbolic computation by Gr\"obner bases that will give us the result. We assume that $R_1, \dotsc, R_6 \in S^2$ are the realizations on the sphere of the vertices of~$K_{3,3}$. Without loss of generality, we can assume that $R_1 = (1,0,0)$ and $R_6 = (0,1,0)$, while we keep symbolic coordinates $(x_i, y_i, z_i)$ for all other points $R_2, \dotsc, R_5$. These $12$ variables satisfy the sphere equations 
\[
 x_i^2 + y_i^2 + z_i^2 = 1 
 \qquad
 \text{for } i \in \{2,3,4,5\} \,.
\]
More equations come from the pieces of information we have about the lengths~$\delta_{S^2}$: from $\delta_{S^2}(3,6) = 0$ we derive $y_3 = 0$, and from $\delta_{S^2}(5,6) = e$ we get $y_5 = e$. Similarly, we obtain $x_2 = x_4 = a$. We have four other equations:
\begin{align*}
 x_4 x_5 + y_4 y_5 + z_4 z_5 &= 0 \,, \\
 x_2 x_5 + y_2 y_5 + z_2 z_5 &= 0 \,, \\
 x_3 x_4 + y_3 y_4 + z_3 z_4 &= -a \,,\\
 x_2 x_3 + y_2 y_3 + z_2 z_3 &= a \,.
\end{align*}
Altogether, we have $12$ equations in $14$ variables. 
Via Gr\"obner bases we see that their solution space $\mcal{B} \subseteq \C^{14}$ has dimension~$2$. 
If we eliminate the variables $\{ x_i, y_i, z_i \}_{i=2}^{5}$ from these equations, 
we obtain the relation $a^3 e^2 + a^3 - a e^2$ between the parameters~$a$ and~$e$, 
which defines a curve $\mcal{A} \subseteq \C^2$. 
Hence we have a map $\mcal{B} \longrightarrow \mcal{A}$ which, by dimension reasons, must have infinite fibers. 
This implies that there are mobile instances of~$K_{3,3}$ in this case.
\end{proof}

\begin{example}
\begin{figure}
\centering
  \includegraphics[width=0.2\textwidth]{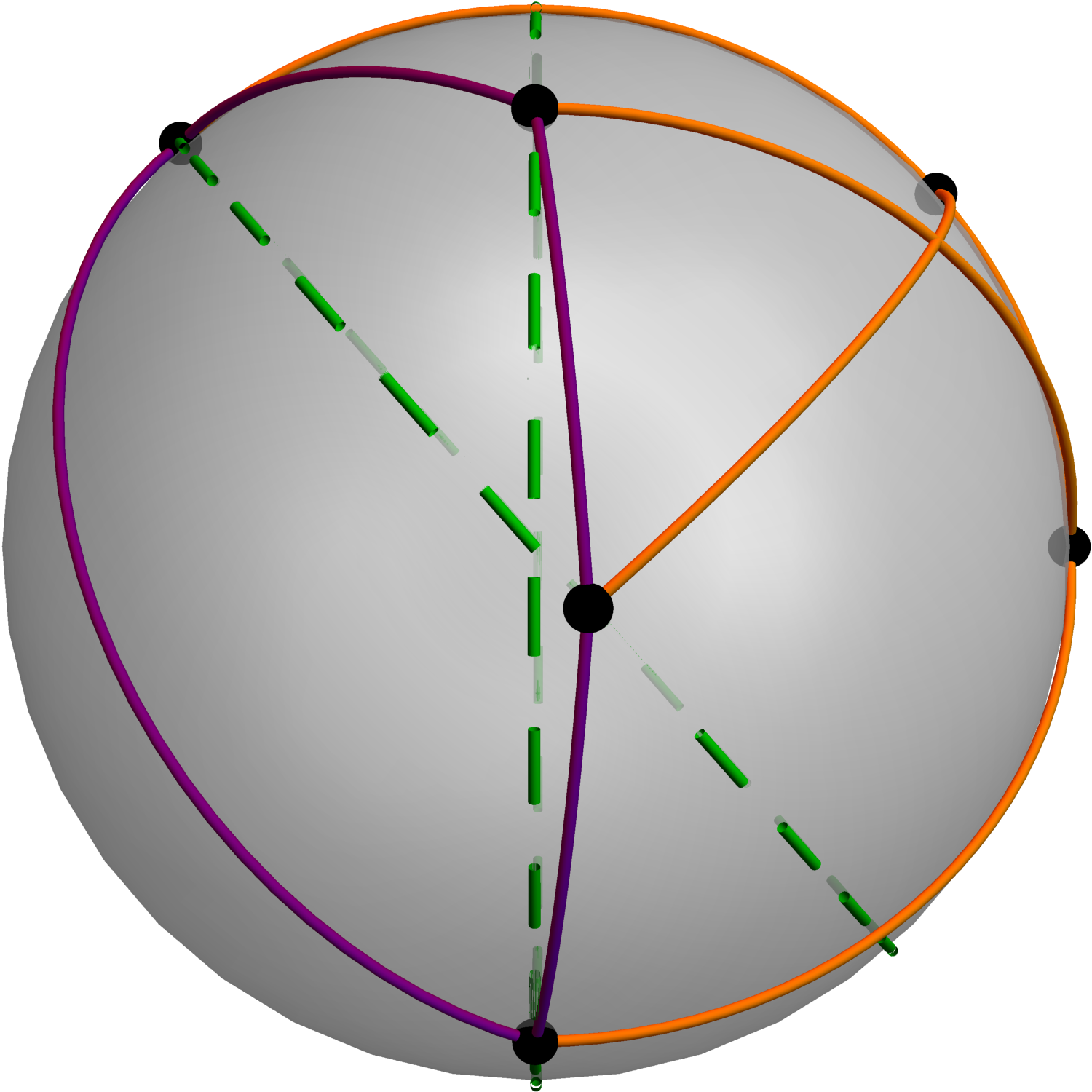} 
  \includegraphics[width=0.2\textwidth]{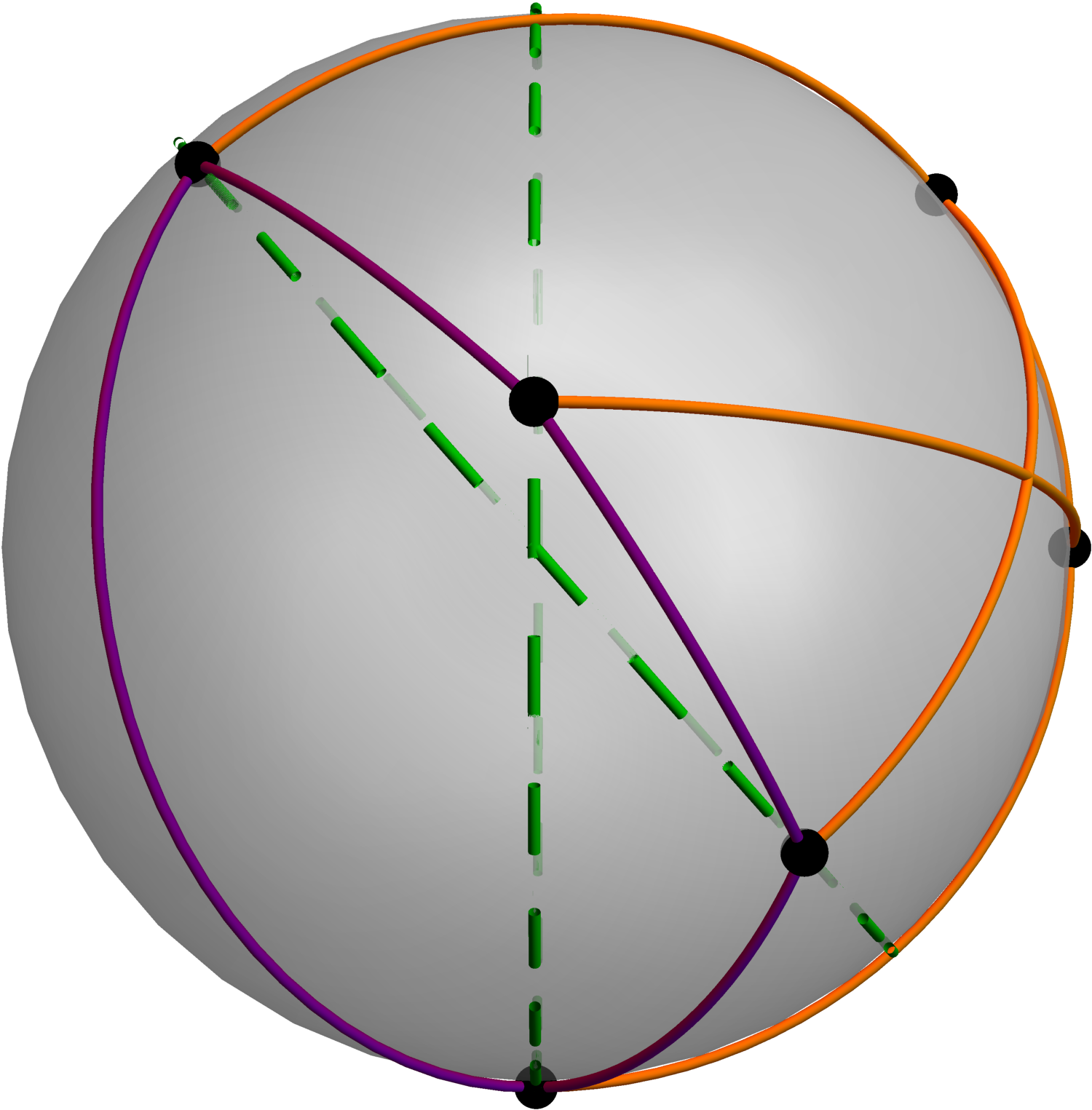} 
  \includegraphics[width=0.2\textwidth]{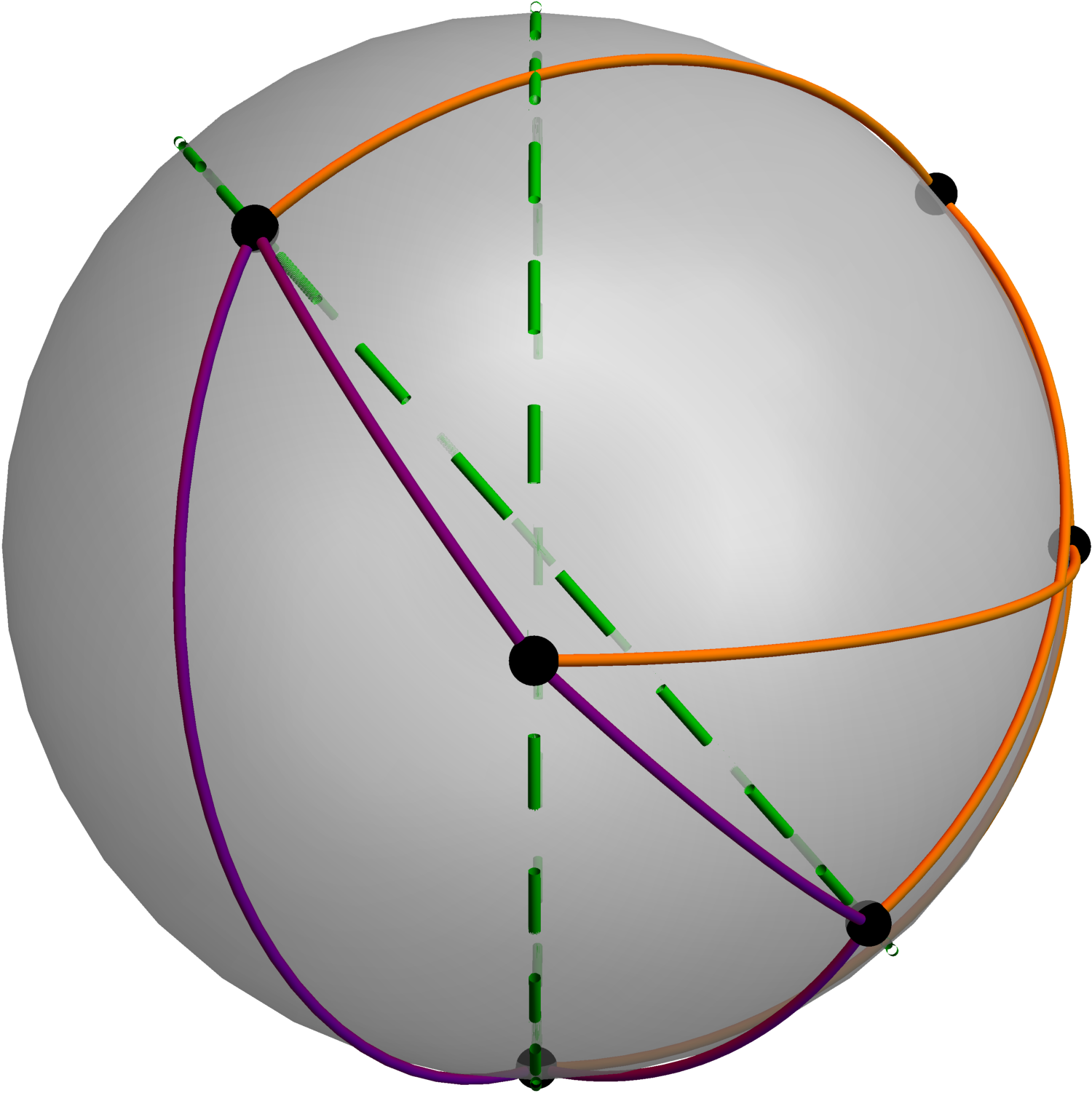} 
  \includegraphics[width=0.2\textwidth]{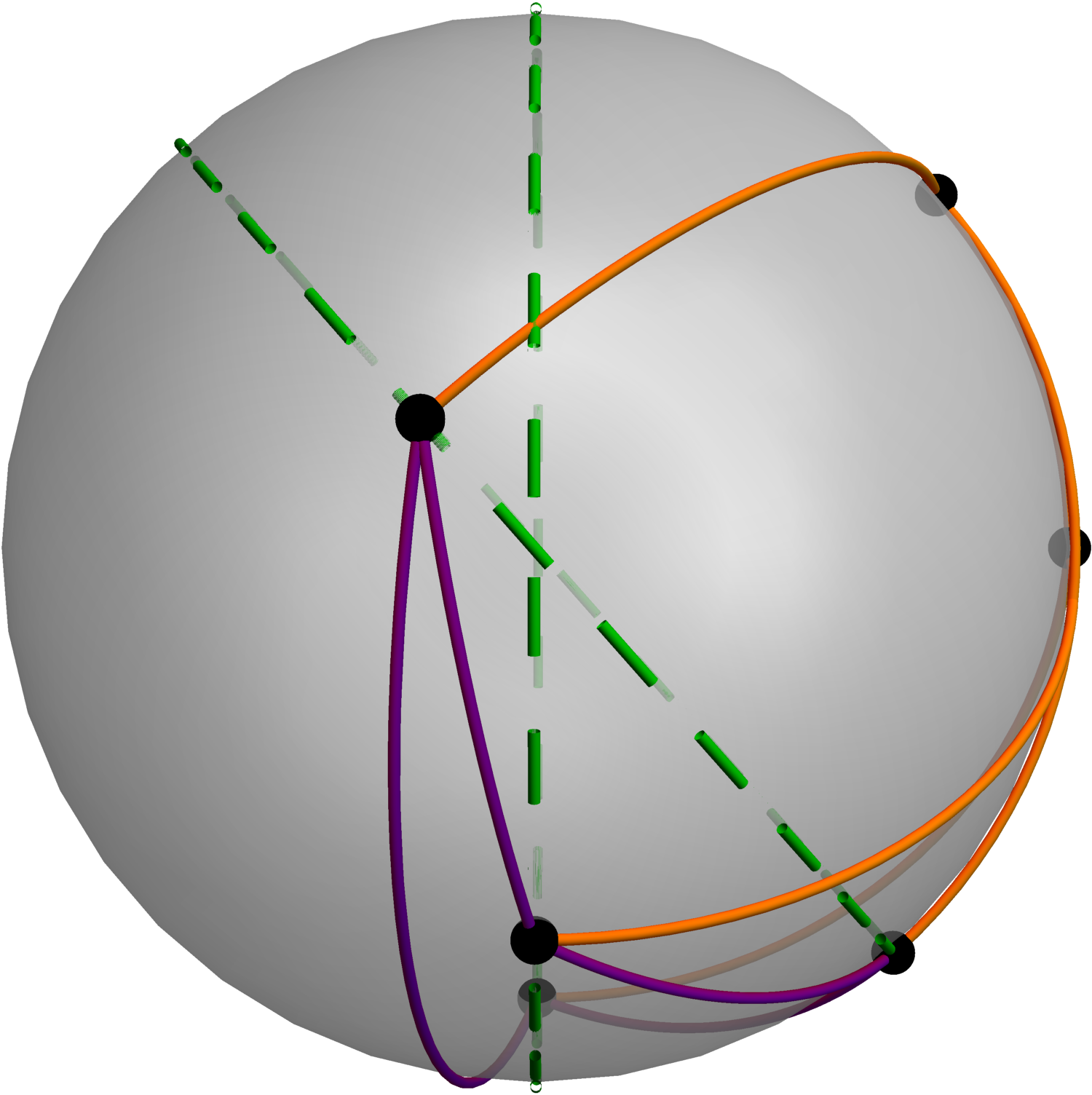} 
%
\caption{Visualization of four realizations of~$K_{3,3}$ during a constant diagonal angle motion. 
The diagonals (in green) and their poles on the silhoutte circle remain fixed.
The cosine of the orange edge between the poles is $\frac{3}{4}$.
The angles of the other orange edges are $\frac{\pi}{2}$.
The cosines of the purple edges are $\frac{3}{5}, \frac{3}{5}, \frac{3}{5}$ and $-\frac{3}{5}$.  
}
\label{figure:const_angle_frames}
\end{figure}
 If, in the construction provided by \cref{proposition:new_dixon}, 
we set $a = 3/5$ and $e = 3/4$, we can take the following (radical) 
parametrization of the constant diagonal angle motion:
\[
\left\{
\begin{aligned}
 R_1 &= (1, 0, 0), \\
 R_i &= \bigl( x_i(t), y_i(t), z_i(t) \bigr) \quad \text{for } i \in \{2,3,4,5\}, \\
 R_6 &= (0, 1, 0),
\end{aligned}
\right.
\]
where
\begin{align*}
 y_3(t) &= 0, &
 y_5(t) &= 3/4, \\
 x_2(t) &= 3/5, &
 x_4(t) &= 3/5,
\end{align*}
\begin{align*}
 x_3(t) &= \frac{2t}{t^2+1}, &
 z_3(t) &= \frac{t^2-1}{t^2+1}, \\
 z_2(t) &= \frac{3}{5} \cdot \frac{t-1}{t+1}, &
 z_4(t) &= -\frac{3}{5} \cdot \frac{t+1}{t-1},
\end{align*}
\begin{align*}
 y_2(t) &= \pm \frac{\sqrt{(t+7)(7t+1)}}{5 t+5}, \\
 z_5(t) &= \frac{-5 y_2 t^2 + 5 y_2 \pm \sqrt{25 t^4 y_2^2 - 50 t^2 y_2^2 - 72 t^3 + 25 y_2^2 - 72 t}}{8 (t^2+1)}, \\
 x_5(t) &= \frac{t (16 z_5^2 + 9)}{8 z_5 (t^2-1)}, \\
 y_4(t) &= y_2 + \frac{8 (t^2+1) z_5}{5 (t^2-1)}.
\end{align*}
\Cref{figure:const_angle_frames} shows some realizations of~$K_{3,3}$ during this motion.
\end{example}

Now that the analysis is concluded, we sum up the results we have obtained in the following theorem:

\begin{theorem}
\label{theorem:classification_K33}
 All the possible (real) motions of~$K_{3,3}$ on the sphere, for which no two vertices 
coincide or are antipodal, are the following: 
 \begin{itemize}
  \item spherical Dixon 1;
  \item spherical Dixon 2;
  \item constant diagonal angle motion.
 \end{itemize}
\end{theorem}

\end{document}